\documentclass[12pt]{book}
\usepackage{amsmath}
\usepackage{amssymb}
\usepackage{amsthm}
\usepackage{verbatim}
\usepackage{mathrsfs}
\usepackage{psfrag}
\usepackage{graphicx}
\usepackage[paper=a4paper,left=30mm,right=30mm,top=45mm,bottom=45mm]{geometry}

\newcommand{\R}[0]{\mathbb R}
\newcommand{\Rn}[0]{\mathbb{R}^n}

\newcommand{\Ds}[0]{\mathcal D}
\newcommand{\N}[0]{\mathbb N}

\newtheorem{Th}{Theorem}[chapter]
\newtheorem{Lemma}{Lemma}[chapter]
\newtheorem{Prop}[Lemma]{Proposition}
\newtheorem{Coro}[Th]{Corollary}
\newtheorem{Def}{Definition}[chapter]
\newtheorem{Rem}{Remark}[chapter]
\newtheorem{Example}{Example}[chapter]

\begin{document}

\begin{center}

{\LARGE On the Well-posedness of the Incompressible Euler Equation} \\

\vspace{2cm}

{\large \textbf{Dissertation}} \\

\vspace{.3cm}

zur \\

\vspace{.3cm}

Erlangung der naturwissenschaftlichen Doktorw\"urde \\

(Dr. sc. nat.) \\

\vspace{.3cm}

vorgelegt der \\

\vspace{.3cm}

Mathematisch-naturwissenschaftlichen Fakult\"at \\

\vspace{.3cm}

der \\

\vspace{.3cm}

Universit\"at Z\"urich \\

\vspace{.3cm}

von \\

\vspace{1cm}

{\Large \textbf{Hasan Inci}} \\

\vspace{1cm}

von \\

\vspace{.3cm}

Z\"urich ZH \\

\vspace{\stretch{1}}

Promotionskomitee \\

\vspace{.5cm}

Prof. Dr. Thomas Kappeler (Vorsitz) \\

Prof. Dr. Camillo De Lellis \\

\vspace{1cm}

Z\"urich, 2012

\end{center}

\chapter*{}
\begin{center}
\textbf{Zusammenfassung}
\end{center}

In dieser Arbeit beweisen wir, dass sich die homogene inkompressible Euler-Gleichung der Hydrodynamik auf den Sobolev-R\"aumen $H^s(\R^n)$, $n \geq 2$ und $s > n/2+1$, als Geod\"aten-Gleichung auf einer unendlich dimensionalen Mannigfaltigkeit beschreiben l\"asst. Als Anwendung dieser geometrischen Formulierung beweisen wir, dass die L\"osungs-Abbildung der inkompressiblen Euler-Gleichung nirgends lokal uniform stetig und nirgends differenzierbar von den Anfangswerten abh\"angt. 

\vspace{4cm}

\begin{center}
\textbf{Abstract}
\end{center}

In this thesis we prove that the homogeneous incompressible Euler equation of hydrodynamics on the Sobolev spaces $H^s(\R^n)$, $n \geq 2$ and $s > n/2+1$, can be expressed as a geodesic equation on an infinite dimensional manifold. As an application of this geometric formulation we prove that the solution map of the incompressible Euler equation, associating intial data in $H^s(\R^n)$ to the corresponding solution at time $t > 0$, is nowhere locally uniformly continuous and nowhere differentiable.

\tableofcontents

\chapter{Introduction}\label{section_introduction}

The initial value problem for the Euler equation in dimension $n \geq 2$ of a incompressible, homogeneous ideal fluid is given by
\begin{eqnarray}
\nonumber
\partial_t u + (u \cdot \nabla) u &=& - \nabla p \\
\label{E}
\operatorname{div} u &=& 0 \\
\nonumber
u(0) &=& u_0  
\end{eqnarray}
where $u(t,x)=(u_1,\ldots,u_n)$ is the velocity field at $(t,x) \in \R \times \R^n$, $p(t,x)$ is the (time-dependent) pressure and $\nabla$ the gradient operator in $\R^n$. The operator $u \cdot \nabla=\sum u_k \partial_k$ acts componentwise on $u$ and $\operatorname{div} u=\sum \partial_k u_k$ is the divergence of $u$. The initial value $u_0$ is assumed to be divergence-free. Equation \eqref{E} was introduced by Euler \cite{euler} in 1757. The first equation in \eqref{E} reflects Newton's second law, i.e. the left side describes the acceleration of the fluid particles, whereas the right side the acting force. The second equation says that the fluid flow is incompressible.\\ \\
An important question for a PDE such as \eqref{E}, is, whether it is ''well-posed''. In \cite{hadamard,hadamard2} Hadamard discusses the following properties for various problems such as initial value problems or boundary value problems arising in physics:\\
\begin{tabular}{l}
 (1) existence of solutions; (2) uniqueness of solutions; (3) continuous dependence \\ 
 of solutions on the data.
\end{tabular}
\newline
Today it is common to say that such a problem is well-posed in the sense of Hadamard if it has properties (1), (2) and (3). To give a precise meaning one has to choose appropriate function spaces and specify the notion of solution. Once this is done, well-posedness says that, there is a unique continuous map, called ''solution map'', mapping the ''data'' to the solution. For initial value problems there is an additional distinction between ''local'' resp. ''global'' well-posedness. If it is known that the time of existence for all solutions can be extended to arbitrarily large times, then the problem is said to be globally well-posed. Otherwise the problem is called locally well-posed.\\ \\
A lot of research activity is concerned with investigating regularity properties of solution maps such as differentiability, $C^k$-smoothness, analyticity, local Lipschitz continuity, local H\"older continuity or uniform continuity on bounded sets. In fact, in some research, property (3) in the definition of well-posedness is replaced by the following stronger condition (see e.g. \cite{ill_kpv})
\begin{itemize}
 \item[(3')] the solution map is uniformly continuous on bounded sets; 
\end{itemize}
In this thesis we will use the notion of ''well-posedness'' always in the sense of Hadamard. In the cases where the solution map is $C^k$-smooth $(k \geq 1)$ or analytic, one calls the problem $C^k$-well-posed or analytically well-posed (see e.g. \cite{tao_book}). If not then the problem is called $C^k$-ill-posed resp. analytically ill-posed. If a contraction mapping argument is used to get the solution, then the solution map is at least Lipschitz. As a consequence, lack of the Lipschitz property of the solution map means that one cannot use a straightforward contraction mapping argument to establish existence of solutions.\\
More recently there has been a lot of research activity in establishing that solution maps in certain function spaces have poor regularity properties. As a first example we mention a discussion of such a result in \cite{kato_burgers}, for the inviscid Burgers' equation
\begin{equation}\label{burgers_eq}
 \partial_t u + u \partial_x u=0,\; u(0)=u_0, \quad t \in \R,\; x \in \R.
\end{equation}
Kato \cite{kato_burgers} proves that for no $0 < \alpha \leq 1$ and no $t > 0$ the solution map of equation \eqref{burgers_eq}, $u_0 \mapsto u(t)$, is locally $\alpha$-H\"older continuous in the Sobolev space $H^s(\R), s \geq 2$. A further example is given by the mKdV equation
\begin{equation}\label{mkdv_eq}
 \partial_t u + \partial_x^3 u + u^2 \partial_x u=0,\; u(0)=u_0, \quad t \in \R,\; x \in \R.
\end{equation}
Whereas it is not known whether \eqref{mkdv_eq} is well-posed in $H^s(\R)$ for $s < 1/4$, Kenig, Ponce and Vega \cite{ill_kpv} prove that in any case the solution map of equation \eqref{mkdv_eq} does not have property (3') in these spaces. A similar result holds for the focusing cubic NLS \cite{ill_kpv}. Another example is the Benjamin-Ono equation
\begin{equation}\label{bm_eq}
\partial_t u + H \partial_x^2 u + u \partial_x u = 0,\; u(0)=u_0, \quad t \in \R,\;x \in \R
\end{equation}
where $H$ denotes the Hilbert transform. It is known that \eqref{bm_eq} is well-posed in $H^s(\R^n)$, $s \geq 1$ -- see \cite{tao}. On the other hand Koch and Tzvetkov \cite{ill_benjamin_ono} show that the solution map of equation \eqref{bm_eq}, doesn't have property (3') for these spaces. They also show that the solution map of \eqref{bm_eq} cannot have property (3') for $H^s(\R)$, $0 < s < 1$. A similar result holds for the Camassa-Holm equation \cite{camassa_holm}. Finally let us mention the periodic KdV-equation
\begin{equation}\label{kdv_eq}
 \partial_t u + \partial_x^3 u + 6u \partial_x u = 0,\; u(0)=u_0, \quad t \in \R,\;x \in \mathbb T
\end{equation}
for initial data $u_0$ with average $0$. In \cite{kdv_kpv} it is shown, that equation \eqref{kdv_eq} is well-posed and satisfies (3') in $H_0^s(\mathbb T)$ with $s \geq -1/2$. On the other hand, in \cite{i_team} it is proven that the solution map of \eqref{kdv_eq} does not have property (3') in $H^s_0(\mathbb T)$, $-2 \leq s < -1/2$. Moreover in \cite{kappeler}, Kappeler and Topalov show that \eqref{kdv_eq} is well-posed in $H^s_0(\mathbb T)$ for any $-1 \leq s < -1/2$. 
\\ \\
One of the main results of this thesis concerns the regularity of the solution map of \eqref{E}. To state our result we have to recall the notion of locally uniform continuity (see e.g. \cite{izzo}).

\begin{Def}\label{def_locally_uniform}
We say that a map $f:X \to Y$ between metric spaces $X,Y$ is locally uniformly continuous at a point $x \in X$ if there is a neighborhood $U \subseteq X$ of $x$ so that the restriction $\left. f \right|_U$ is uniformly continuous. If it is locally uniformly continuous at every point of $X$, we say that $f$ is locally uniformly continuous. If at no point of $X$, $f$ is locally uniformly continuous, we say that $f$ is nowhere locally uniformly continuous. Equivalently, $f$ is nowhere locally uniformly continuous if on no open nonempty subset of $X$ it is uniformly continuous.
\end{Def}

Examples of locally uniformly continuous functions are easy to find. For any locally compact metric space $X$, every continuous map $f:X \to Y$, $Y$ a metric space, is locally uniformly continuous. Clearly any locally Lipschitz map $f:X \to Y$ between metric spaces $X,Y$ is locally uniformly continuous. It is also easy to find an example of a real valued locally uniformly continuous map $f:H \to \R$ on a Hilbert space -- necessarily of infinite dimension -- with the property that there is a bounded subset on which $f$ is not uniformly continuous. 
\begin{Example}\label{ex_not3}
Let $(H,\langle \cdot,\cdot \rangle)$ be a Hilbert space with an orthonormal basis $(e_k)_{k \geq 1}$. For $x \in H$ and $r > 0$ we define the elementary function $\psi_{x,r}:H \to \R$ by
\[
 \psi_{x,r}(y)=\begin{cases} e^{-\frac{r^2}{r^2-||y-x||^2}}, \quad & ||y-x|| < r \\ 0, \quad & ||y-x|| \geq r \end{cases}
\]
where $||x||=\langle x,x \rangle^{1/2}$ is the induced norm in $H$. Note that $\psi_{x,r}$ is a $C^\infty$-function with support in the ball of radius $r$ centered at $x \in H$. With this elementary functions we construct the following function 
\[
 \Phi:H \to \R, \quad y \mapsto \sum_{k \geq 1} \psi_{e_k,\frac{1}{2k}}(y).
\]
As the $\psi$'s appearing in the latter sum have pairwise disjoint supports, $\Phi$ defines a $C^\infty$-function. In particular it is locally uniformly continuous. But one easily sees that $\Phi$ is not uniformly continuous in the ball of radius $2$, centered at $0$. 
\end{Example}

In contrast nowhere locally uniformly continuous functions are not so easy to find, despite the fact there are many of them (see \cite{izzo} for a precise statement). The following example of a nowhere locally uniformly continuous function is due to Izzo \cite{izzo}.

\begin{Example}\label{ex_not_uniform}
The function $f:\ell^2 \to \R$ defined by
\[
 x=(x_1, x_2, \ldots) \mapsto f(x)=\sum_{k=1}^\infty x_k^2 \cos(k x_k)
\]
is continuous but nowhere locally uniformly continuous. Here $\ell^2$ is the real Hilbert space of square-summable real sequences.
\end{Example}

Another example of a nowhere locally uniformly continuous function is given by the composition map in the groups $\Ds^s(\R^n)$ of orientation preserving diffeomorphisms of Sobolev class $H^s$ -- see Section \ref{functional_setting} for the definition of $\Ds^s(\R^n)$ and other spaces of maps on $\R^n$. This example plays an important role for the proof of the main theorems of this thesis. It will be discussed in Section \ref{functional_setting}.\\ \\
Before stating our results concerning the regularity of the solution map of equation \eqref{E}, let us review the main known facts regarding well-posedness of \eqref{E}. The first local well-posedness results of \eqref{E} were established for solutions with values in H\"older spaces $C^{k,\alpha}$,$\;k \geq 1$ and $0 < \alpha < 1$, by Lichtenstein \cite{lichtenstein} and Gunter \cite{gunter}. In space dimension $n=2$, Wolibner \cite{wolibner} and H\"older \cite{hoelder} proved global well-posedness for solutions in such spaces. In space dimension $n \geq 3$ global well-posedness is still a major open problem.\\ \\
The function spaces for equation \eqref{E} considered in this thesis are the Sobolev spaces $H^s(\R^n;\R^n)$, $s > n/2+1$, resp. their subspaces $H^s_\sigma(\R^n;\R^n)$ consisting of divergence-free maps -- see Section \ref{functional_setting} for more details. The solutions of \eqref{E} considered in this thesis are elements in $C^0\big([0,T];H^s_\sigma(\R^n;\R^n)\big)$, $T > 0$ -- see Section \ref{section_well_posedness} for the notion of solution used. Here $C^0\big([0,T];H^s_\sigma(\R^n;\R^n)\big)$ is the space of continuous curves from $[0,T]$ to $H^s_\sigma(\R^n;\R^n)$ with the supremum norm. Results for such solutions were first established by Ebin and Marsden \cite{ebin}. They proved that on compact manifolds of dimension $n \geq 2$, possibly with boundary, equation \eqref{E} is locally well-posed in the Sobolev spaces $H^s_\sigma$, $s > n/2+1$. For Sobolev spaces on $\R^n$, the analog result is due to Kato, and reads as follows. 

\begin{Th}\label{th_kato}\cite{kato}
Let $n \geq 2$ and $s > n/2+1$. Then equation \eqref{E} is locally well-posed in the Sobolev space $H^s_\sigma(\R^n;\R^n)$, i.e. for any $w \in H^s_\sigma(\R^n;\R^n)$ there is $T > 0$ and a neighborhood $U \subseteq H^s_\sigma(\R^n;\R^n)$ of $w$ such that for any $u_0 \in U$ there exists a unique solution of \eqref{E} in $C^0\big([0,T];H^s_\sigma(\R^n;\R^n)\big)$ such that the solution map
\[
 E:U \to C^0\big([0,T];H^s_\sigma(\R^n;\R^n)\big),\quad u_0 \mapsto [t \mapsto u(t;u_0)]
\]
is continuous.
\end{Th}

As an aside we mention that results on weak solutions of \eqref{E} are only a few. For $n=2$, Yudovich \cite{yudovich} proved existence and uniqueness of weak solutions with bounded vorticity (see Section \ref{section_proof_main} for the notion of vorticity) and existence of solutions with vorticity in $L^p$, $p >1$. In \cite{delort} existence of weak solutions for $n=2$ was established whose vorticities are positive measures. Non uniqueness results of weak solutions were shown in \cite{delellis,scheffer,shnirelman}.\\ \\
For any given $T >0$, denote by $E_T$ the time $T$ solution map in $H^s_\sigma(\R^n;\R^n)$, $s > n/2+1$. More precisely, let $U_T \subseteq H^s_\sigma(\R^n;\R^n)$ be the set of all $u_0 \in H^s_\sigma(\R^n;\R^n)$ such that the solution of \eqref{E} with initial value $u_0$ exists longer then time $T$ and by $E_T$ the map
\begin{equation}\label{ET_def}
 E_T:U_T \to H^s_\sigma(\R^n;\R^n),\quad u_0 \mapsto u(T;u_0).
\end{equation}
It turns out that $U_T$ is open -- see Section \ref{section_exponential_map}. Moreover it is star-shaped with respect to $0$. Indeed, we have the following scaling property for equation \eqref{E}: If $u(t)$ is a solution of \eqref{E} in $H^s_\sigma(\R^n;\R^n)$, then for any $\lambda > 0$, 
\begin{equation}\label{scaling}
w(t):=\lambda u(\lambda t)
\end{equation}
is also a solution to \eqref{E}. This shows that $U_T$ is star-shaped with respect to $0$ and we have
\[
 U_{\lambda T} = \frac{1}{\lambda} U_T
\]
for any $\lambda > 0$. One of the main results of this thesis is the following

\begin{Th}\label{th_not_uniform}
For any $n \geq 2$, $T > 0$ and $s > n/2+1$ the map
\[
 E_T:U_T \to H^s_\sigma(\R^n;\R^n),\quad u_0 \mapsto u(T;u_0)
\]
is nowhere locally uniformly continuous.
\end{Th}

\noindent
An immediate corollary is

\begin{Coro}\label{coro_not_smooth}
For any $n \geq 2$, $T > 0$ and $s > n/2+1$ the map
\[
 E_T:U_T \to H^s_\sigma(\R^n;\R^n)
\]
is nowhere locally Lipschitz, hence nowhere $C^1$-smooth.
\end{Coro}

A further result of the thesis is concerned with the differentiability of $E_T$.

\begin{Th}\label{th_not_differentiable}
For any $n \geq 2$, $T > 0$ and $s > n/2+1$ the map
\[
 E_T:U_T \to H^s_\sigma(\R^n;\R^n)
\]
is nowhere differentiable.
\end{Th}

Note that Theorem \ref{th_not_differentiable} is not implied by Theorem \ref{th_not_uniform}. Indeed consider the function $f$ given in Example \ref{ex_not_uniform}. One easily verifies that $f$ is differentiable at $0 \in \ell^2$ with derivative $d_0 f=0$. Hence differentiability in a point doesn't imply that the function is locally uniformly continuous in that point.

In the following we give a brief description of the method we want to use to prove Theorem \ref{th_not_uniform} and Theorem \ref{th_not_differentiable}. This method doesn't only apply to the incompressible Euler equation, but works for a whole class of equations arising as geodesic equations on the diffeomorphism groups $\Ds^s(\R^n)$, including Burgers' equation, the Camassa-Holm equation \cite{ch_geodesic} and its generalization to higher dimensions, the Degasperis-Procesi equation \cite{dp_geodesic}, the averaged Euler equation \cite{ae_geodesic} and the hyperelastic rod wave equation -- see \cite{geodesic_equation}. The key point is that the symmetries of the geodesic equations on $\Ds^s(\R^n)$ give rise to a conservation law involving the composition map of $\Ds^s(\R^n)$
\[
 \mu:\Ds^s(\R^n) \times \Ds^s(\R^n),\quad (\varphi,\psi) \mapsto \varphi \circ \psi.
\]
One can show that the composition map is continuous (see e.g. \cite{composition}), but nowhere locally uniformly continuous and nowhere differentiable (see Lemma \ref{lemma_composition_not_uniform}). The formulation of \eqref{E} in terms of geodesic flows, the above mentioned conservation law as well as these properties of the composition map are the key ingredients for the proofs of Theorem \ref{th_not_uniform} and Theorem \ref{th_not_differentiable}.\\ \\
Before we state further results of this thesis, we have to make some comments on the formulation of \eqref{E} in terms of geodesic flows, referred to as ''Arnold's formalism'' in the sequel. To do that, let us first explain the Lagrangian formulation of the equation \eqref{E}. Recall that the equation \eqref{E} is referred to as the Eulerian description of the incompressible Euler equation. That is, one fixes a point $x \in \R^n$ in space and observes the velocity $u(t,x)$ of the fluid at $x$ as time passes. Instead of fixing the position in space one can single out a fluid ''particle'' and follow it. More precisely, let's take a fluid particle which is located at time 0 at $x \in \R^n$. At time $t$ the particle is, let's say, at position $\varphi(t,x) \in \R^n$. The description of the fluid flow in terms of the $\varphi$ variable is called the Lagrangian description. The relation between the Eulerian and the Lagrangian description is given by
\[
 \partial_t \varphi(t,x) = u(t,\varphi(t,x)),\quad \varphi(0,x)=x
\]
,i.e. $\varphi$ is the flow of the vector field $u$. The incompressibility condition $\operatorname{div}u=0$ reads in Lagrangian coordinates as $\det(d\varphi) \equiv 1$ (see Lemma \ref{det_derivative}). In our setting it turns out that these $\varphi$'s are elements in $\Ds^s_\mu(\R^n)$, the space of volume-preserving diffeomorphisms in $\Ds^s(\R^n)$. Therefore solutions of \eqref{E} expressed in Lagrangian coordinates are continuous curves in $\Ds^s_\mu(\R^n)$.\\
The advantage of the Lagrangian description is that it leads to an ODE formulation of \eqref{E}. This approach was already used by Lichtenstein \cite{lichtenstein} and Gunter \cite{gunter} to get local well-posedness of \eqref{E}. Later Arnold \cite{arnold} observed that, when expressed in Lagrangian coordinates, equation \eqref{E} becomes a geodesic equation on $\Ds^s(\R^n)$ with an appropriate choice of the metric. The idea is the following: Consider $\Ds^s(\R^n)$ and its submanifold $\Ds_\mu^s(\R^n)$. The tangent space of $\Ds_\mu^s(\R^n)$ at $\operatorname{id} \in \Ds_\mu(\R^n)$ consists of divergence-free vector fields. At an arbitrary $\psi \in \Ds^s_\mu(\R^n)$, the tangent space $T_\psi \Ds^s_\mu(\R^n)$ consists of vector fields of the form $v \circ \psi$ where $v$ is a divergence-free vector field. Moreover for any $C^2$-curve $\varphi$, its second derivative $\partial_t^2 \varphi(t)$ can be canonically identified with an element of $T_{\varphi(t)} \Ds^s(\R^n)$. We endow $\Ds^s(\R^n)$ with the $L^2$-metric 
\[
 \langle u,v \rangle_\varphi := \langle u,v \rangle_{L^2} \quad \varphi \in \Ds^s(\R^n), \; u,v \in T_\varphi \Ds^s(\R^n) 
\]
and $\Ds^s_\mu(\R^n)$ with the induced metric. A curve $\varphi:[0,T] \to \Ds^s_\mu(\R^n)$, $T > 0$, is a $L^2$-geodesic in $\Ds^s_\mu(\R^n)$ if $\partial_t^2 \varphi \in T_\varphi \Ds^s(\R^n)$ is $L^2$-orthogonal to $T_\varphi \Ds^s_\mu(\R^n) \subseteq T_\varphi \Ds^s(\R^n)$. By the Helmholtz decomposition (see e.g. \cite{majda}) we know that elements which are $L^2$-orthogonal to $T_\varphi \Ds^s_\mu(\R^n)$ are vector fields of the form $-(\nabla p) \circ \varphi$. Thus we arrive at
\[
 \partial_t^2 \varphi = \partial_t \big( u \circ \varphi) = \partial_t u \circ \varphi + \big((u \cdot \nabla) u\big) \circ \varphi = -(\nabla p) \circ \varphi
\]
giving the first equation in \eqref{E}. The property that $\varphi$ is volume-preserving implies the second equation in \eqref{E}. We will show in Section \ref{section_exponential_map} that the above equation leads to the following ODE
\[
 \partial_t \varphi = v,\quad \partial_t v = B(v \circ \varphi^{-1}) \circ \varphi
\]
and $B$ is given by \eqref{B_formula}.\\ \\

The discussion above shows that the problem \eqref{E}, expressed in Lagrangian coordinates, can be formulated as finding the stationary points of the action functional
\[
 \mathcal A_0^T[\varphi] = \frac{1}{2} \int_0^T ||\partial_t \varphi ||_{L^2}^2 dt
\]
for $\varphi:[0,T] \to \Ds^s_\mu(\R^n)$. Arnold presented these ideas in a formal way. Later they were made mathematically rigorous by the work of Ebin and Marsden \cite{ebin}, proving that Arnold's formalism works for Sobolev spaces over compact oriented manifolds. More precisely, they showed that for a compact oriented Riemannian manifold $(M,g)$ of dimension $n \geq 2$, possibly with boundary, one can give a differential structure (i.e. a manifold structure) to $\Ds^s(M)$, the group of orientation-preserving diffeomorphisms of $M$, which are of Sobolev class $s$, $s > n/2 +1$. For the corresponding volume element $\mu$ one has the submanifold $\Ds^s_\mu(M)$ consisting of elements in $\Ds^s(M)$ which are volume-preserving. They proved that the $L^2$-geodesics of $\Ds^s_\mu(M)$ are generated by a smooth vector field on the tangent bundle $T\Ds_\mu^s(M)$, obtained by restricting the corresponding vector field on $T\Ds^s(M)$ to $T\Ds_\mu^s(M)$, thus getting an ODE description of the solutions similarly as Lichtenstein \cite{lichtenstein} and Gunter \cite{gunter}. In this way they proved a local well-posedness theory for the initial value problem \eqref{E} in Sobolev spaces $H^s_\sigma$, $s > n/2+1$. In \cite{cantor}, Cantor worked out Arnold's formalism for $M=\R^n$. But instead of using the usual Sobolev spaces $H^s(\R^n;\R^n)$ he worked in weighted Sobolev spaces and needed to choose weights $w,w'$ so that the Laplacian $\Delta$ is a linear isomorphism, $\Delta:H^s_w(\R^n) \to H^{s-2}_{w'}(\R^n)$. In contrast, $\Delta:H^s(\R^n) \to H^{s-2}(\R^n)$ is not Fredholm, which makes it harder to verify Arnold's formalism for these spaces.\\
After these preliminary remarks we can now state the two remaining main results of this thesis. Theorem \ref{th_lagrangian} below, says that Arnold's formalism works for the Sobolev spaces $H^s_\sigma(\R^n;\R^n)$, $s>n/2+1$. 

\begin{Th}\label{th_lagrangian}
Let $s > n/2+1$ with $n \geq 2$ and let $B$ be the quadratic form given as in \eqref{B_formula}.\\
(i) The map 
\[
 V:\Ds^s(\R^n) \times H^s(\R^n;\R^n) \to H^s(\R^n;\R^n) \times H^s(\R^n;\R^n),\quad (\varphi,v) \mapsto \big(v,\nabla B(v \circ \varphi^{-1}) \circ \varphi \big) 
\]
is a real analytic vector field. Hence the initial value problem
\begin{equation}\label{E_ivp}
 (\partial_t \varphi, \partial_t v) = \big(v,\nabla B(v \circ \varphi^{-1})\circ \varphi\big),\quad \varphi(0)=\operatorname{id},\;v(0)=u_0 \in H^s(\R^n;\R^n)
\end{equation}
is locally analytically well-posed.\\
(ii) For any $u_0 \in H^s_\sigma(\R^n;\R^n)$ the solution $\varphi$ of \eqref{E_ivp} gives rise to a solution of \eqref{E} by the formula $u=\partial_t \varphi \circ \varphi^{-1}$.
\end{Th}

\noindent
In Ebin and Marsden \cite{ebin} it was shown that for any compact oriented Riemannian manifold $M$, possibly with boundary, the group $\Ds_\mu^s(M)$ is a smooth submanifold of $\Ds^s(M)$. Theorem \ref{th_submanifold} below states an analogous result for $\R^n$.

\begin{Th}\label{th_submanifold}
Let $s > n/2+1$ with $n \geq 2$. Then $\Ds^s_\mu(\R^n)$ is a closed real analytic submanifold of $\Ds^s(\R^n)$.
\end{Th}
\vspace{0.4cm}
\noindent
{\em Related work:}
\noindent
In the case $s > n/2+1$, {\em Theorem \ref{th_not_uniform}} improves on a recent result of Himonas and Misiolek \cite{himonas}, saying that property (3') does not hold for $E_T$ for any $T > 0$. More precisely, Himonas and Misiolek construct a pair of sequences of solutions $(u_k)_{k \geq 1},(\tilde u_k)_{k \geq 1}$ to \eqref{E} with the following property: For all $s > 0$ 
\begin{itemize}
\item[(i)] $\big(u_k(0)\big)_{k \geq 1}$ and $\big(\tilde u_k(0)\big)_{k \geq 1}$ are bounded in $H^s_\sigma(\R^n;\R^n)$ with
\[
 \lim_{k \to \infty} ||u_k(0)-\tilde u_k(0)||_s =0.
\]
\end{itemize}
\noindent
and there is a constant $C_s > 0$ so that
\begin{itemize}
\item[(ii)] for all $0 < t < 1$ 
\[
 \liminf_{k \geq 1} ||u_k(t)-\tilde u_k(t)||_s \geq C_s \sin t.
\]
\end{itemize}
As a consequence one concludes that for any $T > 0$, $E_T$ has not property (3').\\ \\
\noindent
A result analogous to {\em Theorem \ref{th_lagrangian}(i)} was proved by Serfati \cite{serfati} for $C^k$-spaces over $\R^n$ and by Shnirelman (\cite{shnirelman}, Theorem 2.1) for Sobolev spaces $H^s$, $s > 5/2$, on $\R^3/\mathbb Z^3$.\\ \\
\noindent
{\em Outline:}  
In Section \ref{functional_setting} we introduce some further notation and define the function spaces used in the sequel. In Section \ref{section_well_posedness} we dicuss the notion of solution of \eqref{E} used in this thesis and the corresponding concept of well-posedness. We conclude the chapter with Section \ref{section_eulerian_formulation} giving an alternative Eulerian formulation for equation \eqref{E} suitable for our purpose. In Section \ref{section_geodesic_equation} we show that the alternative formulation of \eqref{E} of Section \ref{section_eulerian_formulation}, written in Lagrangian coordinates, can be interpreted as a geodesic equation on $\Ds^s(\R^n)$. Chapter \ref{section_lagrangian_formulation} ends with Section \ref{section_exponential_map}, where it is shown that the geodesics can be described by an analytic exponential map. This shows Theorem \ref{th_lagrangian} and thus proves that Arnold's formalism works for the Sobolev space $H^s_\sigma(\R^n;\R^n)$. In Section \ref{section_local_wellposedness} we reprove Theorem \ref{th_kato} using Arnold's formalism. Theorem \ref{th_not_uniform} and Theorem \ref{th_not_differentiable}, are proved in Section \ref{section_proof_main}. Finally the proof of Theorem \ref{th_submanifold} is given in Section \ref{section_submanifold}. Appendix \ref{appendix_analyticity} reviews the notion of analyticity in real Banach spaces. In Appendix \ref{integration} we prove a result on the integration of $H^s$-vector fields. Finally in Appendix \ref{auxiliary} we have collected some auxiliary lemmas.

\chapter{Eulerian formulation}\label{eulerian_formulation}

\section{Diffeomorphism groups}\label{functional_setting}

Throughout this section assume that $n \geq 1$. For $s \geq 0$ we denote by $H^s(\R^n)$ the Sobolev space of order $s$. We can describe $H^s(\R^n)$ as the following subspace of $L^2(\R^n) \equiv L_\R^2(\R^n)$
\[
 H^s(\R^n) = \{ f \in L^2(\R^n) \; | \; (1+|\xi|^2)^{s/2} |\hat f(\xi)| \in L^2(\R^n) \}
\]
where $\hat f$ denotes the Fourier transform of $f$. Recall that for an $L^1$-function $g$ its Fourier transform is the complex-valued function
\[
 \hat g(\xi) = \frac{1}{(2\pi)^{n/2}} \int_{\R^n} g(x) e^{-ix \cdot \xi}\;dx, \quad \xi \in \R^n
\]
where $x \cdot \xi=x_1 \xi_1 + \cdots + x_n \xi_n$ denotes the Euclidean inner product in $\R^n$. The Fourier transform can also be defined on $L^2(\R^n)$ with values in $L_{\mathbb C}^2(\R^n)$ (see e.g. \cite{majda}). Sometimes we will also write $\mathcal F(g)$ for the Fourier transform $\hat g$. Taking the real part of the complex inner product
\[
\langle f, g \rangle_s:= \Re \int_{\R^n} (1+|\xi|^2)^s \hat f(\xi) \overline{\hat g(\xi)}\;d\xi, \quad f,g \in H^s(\R^n)
\]
makes $H^s(\R^n)$ into a real Hilbert space. We refer to \cite{composition} (among other possible references) for a detailed discussion of these Hilbert spaces and related spaces. In particular we recall that for $s > n/2$, $H^s(\R^n)$ is a Banach-algebra under pointwise multiplication. Often we will need the following stronger result -- see e.g. \cite{composition} for a proof.
\begin{Lemma}\label{lemma_multiplication}
For $s > n/2$ and $0 \leq s' \leq s$, pointwise multiplication
\[
 H^{s'}(\R^n) \times H^s(\R^n) \to H^{s'}(\R^n), \quad (f,g) \mapsto f \cdot g
\]
is continuous.
\end{Lemma}
The space of smooth functions with compact support is denoted by $C_c^\infty(\R^n)$. Note that $C_c^\infty(\R^n)$ is a dense subspace of $H^s(\R^n)$. Further we denote by $C^k_0(\R^n)$ the space of $C^k$-functions on $\R^n$ which vanish at infinity with all their derivatives up to order $k$. Endowed with the norm
\[
 ||f||_{C^k} = \max_{|\alpha| \leq k} ||\partial^\alpha f||_{L^\infty}:=\max_{|\alpha| \leq k} \sup_{x \in R^n} |\partial^\alpha f(x)|
\]
it is a Banach space. Here we used the multi-index notation, i.e. we denote for a multi-index $\alpha=(\alpha_1, \hdots , \alpha_n ) \in  \mathbb Z_{\geq 0}^n$ its length by $|\alpha|$,
\[
 |\alpha| = \alpha_1 + \ldots + \alpha_n
\]
and by $\partial^\alpha$ the differential operator
\[
 \partial^\alpha f = \frac{\partial^{|\alpha|} f}{\partial^{\alpha_1}_1 \cdots \partial^{\alpha_n}_n}.
\]
In some situations we will also need the space $H^s_{\operatorname{loc}}(\R^n)$. It is the space of functions $f:\R^n \to \R$ such that $f \chi$ is in $H^s(\R^n)$ for any $\chi \in C_c^\infty(\R^n)$. For $s > n/2+k$ we have the Sobolev imbedding
\begin{equation}\label{sobolev_imbedding}
 H^s(\R^n) \hookrightarrow C^k_0(\R^n).
\end{equation}
We denote by $H^s(\R^n;\R^n)$, $s \geq 0$, the Sobolev spaces, 
\[
 H^s(\R^n;\R^n)=\big\{ f=(f_1,\ldots,f_n) \; \big| \; f_k \in H^s(\R^n), \quad 1 \leq k \leq n \big\}  
\]
with the corresponding norm
\[
 ||f||_s = \left(||f_1||_s^2 + \ldots + ||f_n||_s^2 \right)^{1/2}
\]
and inner product
\[ 
 \langle f,g \rangle_s = \sum_{j=1}^n \langle f_i,g_i \rangle_s.
\]
Similarly, by $C^k_0(\R^n;\R^n)$, $k \geq 0$, we denote the following space of maps
\[
 C_0^k(\R^n;\R^n)=\big\{ f=(f_1,\ldots,f_n) \; \big| \; f_j \in C^k_0(\R^n), \quad 1 \leq j \leq n \big\}
\]
with the norm 
\[
 ||f||_{C^k}=\max_{1 \leq j \leq n} ||f_j||_{C^k}.
\]
For any $s>n/2+1$, $\Ds^s(\R^n)$ denotes the following space of maps on $\R^n$,
\[
 \Ds^s(\R^n) := \left\{ \varphi:\R^n \to \R^n \; \big| \; \varphi - \operatorname{id} \in H^s(\R^n;\R^n) \mbox{ and } \det(d_x \varphi) > 0, \quad \forall x \in \R^n \right\}
\]
where $\operatorname{id}:\R^n \to \Rn$ is the identity map. In view of the imbedding \eqref{sobolev_imbedding}, $x \mapsto d_x\varphi$ is continuous and hence $\Ds^s(\R^n)$ well-defined. As $\varphi \in \Ds^s(\R^n)$ is of the form $\operatorname{id}+H^s(\R^n;\R^n)$ we get from the Sobolev imbedding \eqref{sobolev_imbedding} that $\varphi$ satisfies
\[ 
\lim_{|x| \to \infty} |\varphi(x)|=\infty. 
\]
From \cite{palais} we then conclude that $\varphi$ is a $C^1$-diffeomorphism. The imbedding \eqref{sobolev_imbedding} shows also that $\Ds^s(\R^n) - \operatorname{id} \subseteq H^s(\R^n;\R^n)$ is an open subset. Therefore $\Ds^s(\R^n)$ has naturally the structure of a (analytic) Hilbert manifold (see Appendix \ref{appendix_analyticity}). Furthermore one can prove that $\Ds^s(\R^n)$ is connected -- see \cite{milnor}. In his thesis \cite{cantor_thesis}, Cantor showed that $\Ds^s(\R^n)$ is a topological group where the group operation is given by composition -- cf also \cite{composition} for a proof. In the sequel we often use the following regularity result for right translation -- see e.g. \cite{composition} for a proof.

\begin{Lemma}\label{lemma_continuous_right_translation}
For any $\varphi \in \Ds^s(\R^n)$ and any $0 \leq s' \leq s$
\[
 R_\varphi:H^{s'}(\R^n) \to H^{s'}(\R^n),\quad f \mapsto R_\varphi f:= f \circ \varphi
\]
is a continuous linear isomorphism. Its inverse is given by
\[
 R_\varphi^{-1}=R_{\varphi^{-1}}:H^{s'}(\R^n) \to H^{s'}(\R^n), \quad f \mapsto R_\varphi^{-1} f = f \circ \varphi^{-1}.
\]
\end{Lemma}

As already mentioned in the introduction, the composition map is an example of a map which is nowhere locally uniformly continuous and nowhere differentiable. More precisely, we consider a set-up similar to the one used in the proof of Theorem \ref{th_not_uniform} and Theorem \ref{th_not_differentiable} and show the following result.

\begin{Th}\label{lemma_composition_not_uniform}
For any $n \geq 1$ and any $s > n/2+1$
\[
 \nu: H^s(\R^n) \times \Ds^s(\R^n) \to H^s(\R^n),\quad (f,\varphi) \mapsto R_\varphi^{-1} f:=f \circ \varphi^{-1}
\]
is continuous but nowhere locally uniformly continuous and nowhere differentiable.
\end{Th}

Before we prove Theorem \ref{lemma_composition_not_uniform} we sketch the idea of the proof in the case $n=1$ at a point $(f_\bullet,\operatorname{id}) \in H^s(\R) \times \Ds^s(\R)$ with $f_\bullet \in C_c^\infty(\R)$ as in Figure \ref{fig1}. Choose $R_\ast > 0$ so small that $\operatorname{id}+g \in \Ds^s(\R)$ for any $g \in H^s(\R)$ with $||g||_s < R_\ast$. To see that $\nu$ is not uniformly continuous on
\[
 B_R\big((f_\bullet,\operatorname{id})\big):=\big\{(f_\bullet+f,\operatorname{id}+g) \;\big|\; f,g \in H^s(\R),\; ||f||_s, ||g||_s < R \big\}
\]
for any $0 < R \leq R_\ast$ we argue as follows: fix an interval $I \subseteq \R$ outside of the support of $f_\bullet$, as shown in Figure \ref{fig1}. For some $N \geq 1$, to be chosen later, construct a sequence $(\varphi_k)_{k \geq N} \subseteq \Ds^s(\R)$ with $\varphi_k \to \operatorname{id}$ in $\Ds^s(\R^n)$ and $||\varphi_k-\operatorname{id}||_s < R$ such that $\varphi_k$ restricted to the interval $I$ is a {\em translation to the right by $\frac{1}{k}$} and $\varphi_k$ restricted to the support of $f_\bullet$ is the {\em identity}. We take $N$ so large, that this is possible. For $k \geq N$ choose $\delta f_k \in C_c^\infty(\R)$ with support given by an interval $I_k$ contained in $I$, of length smaller than $\frac{1}{k}$, and with $||\delta f_k||_s=R/2$. Define the sequences $(p_k)_{k \geq N},(\tilde p_k)_{k \geq N} \subseteq H^s(\R) \times \Ds^s(\R)$ by
\[
 p_k=(f_\bullet+\delta f_k,\operatorname{id}) \quad \mbox{and} \quad \tilde p_k=(f_\bullet+\delta f_k,\varphi_k).
\]
Then $\delta f_k \circ \varphi_k^{-1}$ is the right translate of $\delta f_k$ and the supports of $\delta f_k$ and $\delta f_k \circ \varphi_k^{-1}$ are disjoint -- see Figure \ref{fig2}. As $\varphi_k$ is the identity on the support of $f_\bullet$ one has $f_\bullet=f_\bullet \circ \varphi_k^{-1}$. Therefore
\[
 ||\nu(p_k)-\nu(\tilde p_k)||_s=||\delta f_k - \delta f_k \circ \varphi_k^{-1}||_s = 2 ||\delta f_k||_s=R.
\]
But by contruction $||p_k-\tilde p_k||_s=||\varphi_k-\operatorname{id}||_s \to 0$ as $k \to \infty$. This shows that $\nu$ is not uniformly continuous on $B_R\big((f_\bullet,\operatorname{id})\big)$.

\begin{figure}
\psfrag{x}{$x$}
\psfrag{I}{$I$}
\psfrag{g}{$\operatorname{graph}(f_\bullet)$}
\center
\includegraphics[scale=0.5]{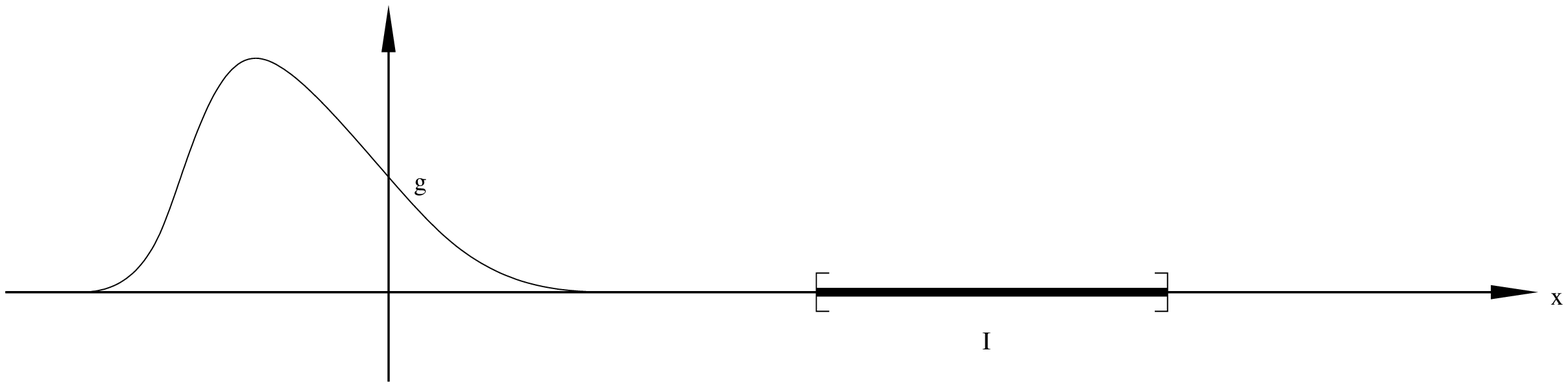}
\caption{The graph of $\nu\big((f_\bullet,\operatorname{id})\big)=f_\bullet$}
\label{fig1}
\end{figure}

\begin{figure}
\psfrag{Ik}{$I_k$}
\psfrag{x}{$x$}
\psfrag{k}{$\frac{1}{k}$}
\psfrag{R}{$\operatorname{graph}(\delta f_k)$}
\psfrag{Q}{$\operatorname{graph}(\delta f_k \circ \varphi_k^{-1})$}
\center
\includegraphics[scale=0.5]{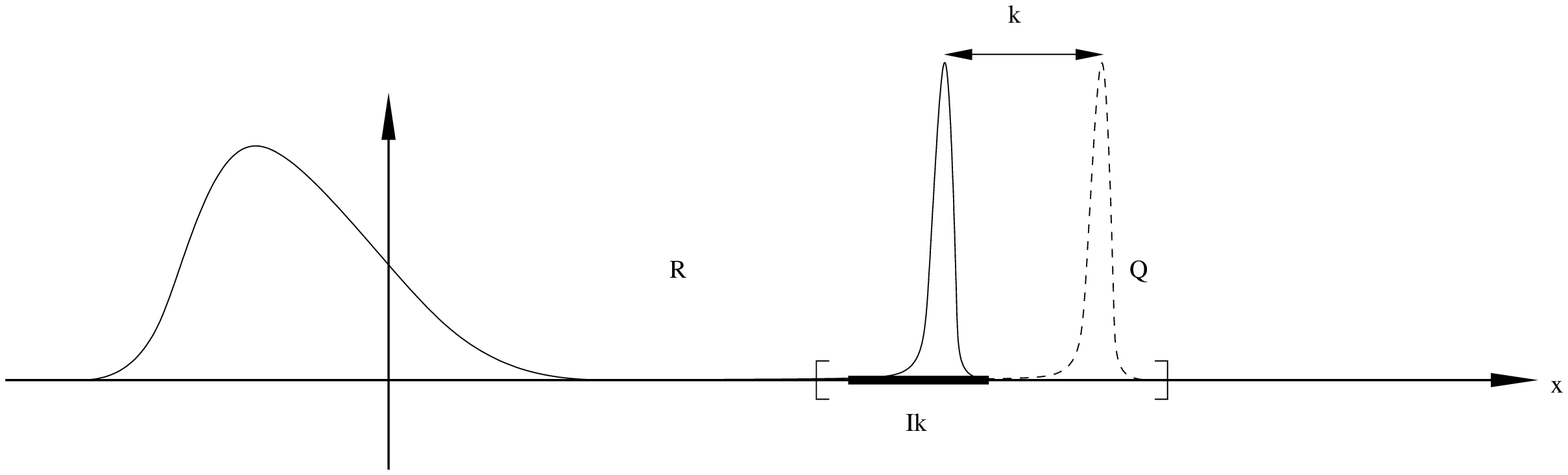}
\caption{The graphs of $\nu\big((f_k,\operatorname{id})\big)$ and $\nu\big((f_k,\varphi_k)\big)$}
\label{fig2}
\end{figure}

\begin{proof}[Proof of Theorem \ref{lemma_composition_not_uniform}]
The continuity of $\nu$ is proved in \cite{composition}. To prove that $\nu$ is nowhere locally uniformly continuous, choose $f_\bullet \in H^s(\R^n)$ and $\varphi_\bullet \in \Ds^s(\R^n)$ arbitrarily. We will show that there exists $R_\ast > 0$ such that $\nu$ is not uniformly continuous on the ball $B_R\big((f_\bullet,\varphi_\bullet)\big)$ in $H^s(\R^n) \times \Ds^s(\R^n)$ for any $0 < R \leq R_\ast$ where 
\[
 B_R\big((f_\bullet,\varphi_\bullet)\big)=\big\{ (f,\varphi) \in H^s(\R^n) \times \Ds^s(\R^n) \;\big|\; ||(f-f_\bullet,\varphi-\varphi_\bullet)||_s < R \big\}
\]
and $||(f-f_\bullet,\varphi-\varphi_\bullet)||_s=\max\{||f-f_\bullet||_s,||\varphi-\varphi_\bullet||_s\}$. By Lemma \ref{lemma_linear_growth} there is $R_0 > 0$, with $B_{R_0}(\varphi_\bullet) \subseteq \Ds^s(\R^n)$ where
\[
 B_{R_0}(\varphi_\bullet):=\big\{ \varphi_\bullet + f \; \big|\; f \in H^s(\R^n;\R^n) \mbox{ with } ||f||_s < R_0 \big\}
\]
and a constant $C >0$ such that
\begin{equation}\label{ineq_up_down}
\frac{1}{C} ||f||_s \leq ||f \circ \varphi||_s \leq C ||f||_s, \quad \forall f \in H^s(\R^n),\;\forall \varphi \in B_{R_0}(\varphi_\bullet).
\end{equation}
By the Sobolev imbedding \eqref{sobolev_imbedding} there exist $0 < R_1 \leq R_0$ and $L > 0$ so that
\begin{equation}\label{less_than_one}
 |\varphi(x) - \varphi_\bullet(x)| < 1,\quad \forall\varphi \in B_{R_1}(\varphi_\bullet),\quad \forall x \in \R^n
\end{equation} 
and
\begin{equation}\label{same_lipschitz}
 |\varphi(x)-\varphi(y)| < L |x-y|,\quad \forall\varphi \in B_{R_1}(\varphi_\bullet),\; \forall x,y \in \R^n.
\end{equation}
Set $R_\ast=R_1$ and choose $0 < R \leq R_\ast$ arbitrarily. By the density of $C_c^\infty(\R^n)$ in $H^s(\R^n)$ choose $\tilde f_\bullet \in C_c^\infty(\R^n) \cap B_{R/4}(f_\bullet)$. Denote by $K \subseteq \R^n$ its support, i.e. $K=\operatorname{supp} \tilde f_\bullet$ and let
\[
 K' = \big\{ x \in \R^n \; \big| \; \operatorname{dist}\big(x,\varphi_\bullet(K)\big) \leq 1 \big\}
\]
where $\operatorname{dist}\big(x,\varphi_\bullet(K)\big)=\inf_{y \in \varphi_\bullet(K)} |x-y|$ is the distance of $x$ to the set $\varphi_\bullet(K)$. Note that $K$ and $K'$ are compact subsets of $\R^n$. By \eqref{less_than_one} we have $\varphi(K) \subseteq K'$ for any $\varphi \in B_R(\varphi_\bullet)$. As $\varphi_\bullet$ is of the form $\operatorname{id}+H^s(\R^n;\R^n)$ we conclude by the Sobolev imbedding $H^s(\R^n) \hookrightarrow C_0^0(\R^n)$ that
\[
 \lim_{|x| \to \infty} |\varphi_\bullet(x)| = \infty.
\]  
Choose $x^\ast \in \R^n$ with $\operatorname{dist}\big(\varphi_\bullet(x^\ast),K'\big) > 2$. By \eqref{less_than_one} we have 
\begin{equation}\label{far_away}
\operatorname{dist}\big(\varphi(x^\ast),K'\big) > 1,\quad \forall \varphi \in B_R(\varphi_\bullet). 
\end{equation}
Take $\delta \varphi \in C_c^\infty(\R^n;\R^n)$ with $||\delta \varphi||_s = R/2$, $\operatorname{supp} \delta \varphi \subseteq B_1(x^\ast)$ and $0 < M:=|\delta \varphi(x^\ast)| < 1$. Define the sequence $(\varphi_k)_{k \geq 1} \subseteq B_R(\varphi_\bullet)$ by
\[
 \varphi_k := \varphi_\bullet + \frac{1}{k} \delta \varphi.
\]
For any $k \geq 1$ let $\delta_k:=\frac{M}{2kL}$, choose $\delta f_k \in C_c^\infty(\R^n)$ with $||\delta f_k||_s = R/2$ and $\operatorname{supp} \delta f_k \subseteq B_{\delta_k} (x^\ast)$ and define
\[
 f_k = \tilde f_\bullet + \delta f_k \in B_R(f_\bullet).
\]
Consider the two sequences $(p_k)_{k \geq 1}, (\tilde p_k)_{k \geq 1} \subseteq B_R\big((f_\bullet,\varphi_\bullet)\big)$ given by
\[
 p_k=(f_k,\varphi_k) \quad \mbox{and} \quad \tilde p_k=(f_k,\varphi_\bullet), \qquad k \geq 1.
\]
Then $||p_k - \tilde p_k||_s \to 0$ as $k \to \infty$. On the other hand we have by the linearity of the right translation
\[
 \nu(p_k)=f_k \circ \varphi_k^{-1}=\tilde f_\bullet \circ \varphi_k^{-1} + \delta  f_k \circ \varphi_k^{-1}
\]
and $\nu(\tilde p_k)$ is given by a similar expression. Then
\begin{eqnarray*}
 ||\nu(p_k)-\nu(\tilde p_k)||_s &=& ||\tilde f_\bullet \circ \varphi_k^{-1} - \tilde f_\bullet \circ \varphi_\bullet^{-1} + \delta f_k \circ \varphi_k^{-1} - \delta f_k \circ \varphi_\bullet^{-1}||_s \\
&\geq& ||\delta f_k \circ \varphi_k^{-1} - \delta f_k \circ \varphi_\bullet^{-1}||_s.
\end{eqnarray*}
where the last inequality follows from the fact that the supports of $\tilde f_\bullet \circ \varphi_k^{-1} - \tilde f_\bullet \circ \varphi_\bullet^{-1}$ and $\delta f_k \circ \varphi_k^{-1} - \delta f_k \circ \varphi_\bullet^{-1}$ are disjoint. Indeed $\tilde f_\bullet \circ \varphi_k^{-1}$ and $\tilde f_\bullet \circ \varphi_\bullet^{-1}$ both are supported in $K'$ and hence
\[
 \operatorname{supp} (\tilde f_\bullet \circ \varphi_k^{-1} - \tilde f_\bullet \circ \varphi_\bullet^{-1}) \subseteq \varphi_k(K) \cup \varphi_\bullet(K) \subseteq K'.
\]
On the other hand $\delta f_k \circ \varphi_k^{-1}$ and $\delta f_k \circ \varphi_\bullet^{-1}$ are both supported in $\R^n \setminus K'$ as by \eqref{same_lipschitz},\eqref{far_away} and the fact that $L \cdot \delta_k < 1$, $k \geq 1$, we have
\begin{eqnarray*}
 \operatorname{supp}(\delta f_k \circ \varphi_k^{-1}) &\subseteq& \varphi_k\big(B_{\delta_k}(x^\ast)\big) \subseteq B_1\big(\varphi_k(x^\ast)\big),\\
\operatorname{supp}(\delta f_k \circ \varphi_\bullet^{-1}) &\subseteq& \varphi_\bullet\big(B_{\delta_k}(x^\ast)\big) \subseteq B_1\big(\varphi_\bullet(x^\ast)\big)
\end{eqnarray*}
and hence
\[
 \operatorname{supp} (\delta f_k \circ \varphi_k^{-1}-\delta f_k \circ \varphi_\bullet^{-1}) \subseteq  \R^n \setminus K'.
\]
We claim that in addition, $\delta f_k \circ \varphi_k^{-1}$ and $\delta f_k \circ \varphi_\bullet^{-1}$ have disjoint supports. Note that
\begin{equation}\label{property1}
 |\varphi_k(x^\ast)-\varphi_\bullet(x^\ast)|=\frac{1}{k} |\delta \varphi(x^\ast)| = \frac{M}{k}.
\end{equation}
By the Lipschitz-property \eqref{same_lipschitz} we have
\begin{equation}\label{property2}
 \varphi\big(B_{\delta_k}(x^\ast)\big) \subseteq B_{\frac{M}{2k}}\big(\varphi(x^\ast)\big),\quad \forall\varphi \in B_R(\varphi_\bullet).
\end{equation}
Combining \eqref{property1} and \eqref{property2}, one sees that
\[
 \varphi_k\big(B_{\delta_k}(x^\ast)\big) \cap \varphi_\bullet\big(B_{\delta_k}(x^\ast)\big) = \emptyset,\quad \forall k \geq 1
\]
showing the claim. Thus we get by \eqref{ineq_up_down}
\begin{eqnarray*}
 ||\delta f_k \circ \varphi_k^{-1} - \delta f_k \circ \varphi_\bullet^{-1}||_s &=& ||\delta f_k \circ \varphi_k^{-1}||_s + ||\delta f_k \circ \varphi_\bullet^{-1}||_s \\
 &\geq& \frac{2}{C} ||\delta f_k||_s = \frac{1}{C} R.
\end{eqnarray*}
In summary, for any $0 < R \leq R_\ast$ we found sequences $(p_k)_{k \geq 1},(\tilde p_k)_{k \geq 1} \subseteq B_R\big((f_\bullet,\varphi_\bullet)\big)$ with
\[
 ||p_k - \tilde p_k||_s \to 0 \quad \mbox{as} \quad k \to \infty
\]
and
\begin{equation}\label{linear_estimate}
 ||\nu(p_k) - \nu(\tilde p_k)||_s \geq \frac{1}{C} R,\quad \forall k \geq 1.
\end{equation}
As $f_\bullet \in H^s(\R^n)$ and $\varphi_\bullet \in \Ds^s(\R^n)$ are arbitrary we get that $\nu$ is nowhere locally uniformly continuous. \\
Note that the constant in \eqref{ineq_up_down}, appearing in \eqref{linear_estimate}, is independent of $0 < R \leq R_\ast$. Moreover the lower bound of \eqref{linear_estimate} is linear in $R$ and independent of the chosen sequences $(p_k)_{k \geq 1}, (\tilde p_k)_{k \geq 1} \subseteq B_R\big((f_\bullet,\varphi_\bullet)\big)$. This is the key ingredient for the second claim of Proposition \ref{lemma_composition_not_uniform} which we now prove. We argue by contradiction and assume that $\nu$ is differentiable in $(f_\bullet,\varphi_\bullet)$. Then the remainder term $\mathcal R(f_\bullet,\varphi_\bullet,\delta f,\delta \varphi)$ in
\[
 \nu(f_\bullet + \delta f,\varphi_\bullet + \delta \varphi)=\nu(f_\bullet,\varphi_\bullet) + d_{(f_\bullet,\varphi_\bullet)} \nu (\delta f,\delta \varphi) + \mathcal R(f_\bullet,\varphi_\bullet,\delta f,\delta \varphi)
\]
can be bounded by
\[
 ||\mathcal R(f_\bullet,\varphi_\bullet,\delta f,\delta \varphi)||_s \leq \frac{1}{4C} ||(\delta f,\delta \varphi)||_s
\]
for all $(\delta f,\delta \varphi) \in H^s(\R^n) \times H^s(\R^n;\R^n)$ with $||(\delta f,\delta \varphi)||_s \leq R$ and $0 < R \leq R_\ast$ sufficiently small. For such a $R$ choose sequences
\[
 p_k=(f_k,\varphi_k) \quad \mbox{and} \quad \tilde p_k=(f_k,\varphi_\bullet),\qquad k \geq 1
\]
as above. Then $\nu(p_k) - \nu(\tilde p_k)$ is given by
\begin{multline*}
d_{(f_\bullet,\varphi_\bullet)} \nu (\delta f_k,\delta \varphi_k) - d_{(f_\bullet,\varphi_\bullet)} \nu (\delta f_k,0)  + \mathcal R(f_\bullet,\varphi_\bullet,\delta f_k,\delta \varphi_k) - \mathcal R(f_\bullet,\varphi_\bullet,\delta f_k,0)\\
= d_{(f_\bullet,\varphi_\bullet)} \nu (0,\delta \varphi_k) + \mathcal R(f_\bullet,\varphi_\bullet,\delta f_k,\delta \varphi_k) - \mathcal R(f_\bullet,\varphi_\bullet,\delta f_k,0).
\end{multline*}
As $\delta \varphi_k \to 0$ for $k \to \infty$, $\limsup_{k \geq 1} ||\nu(p_k) - \nu(\tilde p_k)||_s$ is bounded by
\[
 \limsup_{k \geq 1} ||\mathcal R(f_\bullet,\varphi_\bullet,\delta f_k,\delta \varphi_k)||_s + \limsup_{k \geq 1} ||\mathcal R(f_\bullet,\varphi_\bullet,\delta f_k,0)||_s \leq \frac{1}{2C} R
\]
which contradicts \eqref{linear_estimate} showing that $\nu$ is not differentiable at $(f_\bullet,\varphi_\bullet)$. As $(f_\bullet,\varphi_\bullet) \in H^s(\R^n) \times \Ds^s(\R^n)$ was arbitrary, we get that $\nu$ is nowhere differentiable.
\end{proof}

\noindent
In view of the second equation in \eqref{E} we will also need to consider the subspace $H_\sigma^s(\R^n,\R^n) \subseteq H^s(\R^n;\R^n)$, $s \geq 0$, consisting of divergence-free vector fields, i.e. 
\[
 H_\sigma^s(\R^n;\R^n) = \{ u \in H^s(\R^n;\R^n) \; | \; \operatorname{div} u = 0\}.
\] 
Here $\operatorname{div}$ is in the sense of weak derivatives (see \cite{composition}). As for any $s \geq 1$ the operator $\operatorname{div}:H^s(\R^n;\R^n) \to H^{s-1}(\R^n)$ is continuous, $H^s_\sigma(\R^n;\R^n)$ is a closed subspace of $H^s(\R^n;\R^n)$. Actually it is easy to see that $H^s_\sigma(\R^n;\R^n)$ is a closed subspace of $H^s(\R^n;\R^n)$ for any $s \geq 0$. Introduce
\begin{equation}\label{smooth_V}
 C^\infty_{\sigma,c}(\R^n;\R^n) = \big\{ f \in C_c^\infty(\R^n;\R^n) \; \big| \; \operatorname{div} f = 0 \big\}=H^s_\sigma(\R^n;\R^n) \cap C_c^\infty(\R^n;\R^n).
\end{equation}
In Lemma \ref{lemma_dense} we prove that $C^\infty_{\sigma,c}(\R^n;\R^n)$ is dense in $H^s_\sigma(\R^n;\R^n)$. We will also need the subset $\Ds^s_\mu(\R^n)$ of $\Ds^s(\R^n)$ consisting of volume preserving diffeomorphisms, i.e. for any $s > n/2+1$
\[
 \Ds^s_\mu(\R^n) = \big\{\varphi \in \Ds^s(\R^n) \,\big| \,\det(d_x \varphi) = 1, \quad \forall x \in \R^n \big\}. 
\]
Note that $\Ds^s_\mu(\R^n)$ is a subgroup of $\Ds^s(\R^n)$. Theorem \ref{th_submanifold}, proved in Section \ref{section_submanifold}, says that $\Ds^s_\mu(\R^n)$ is a closed real analytic submanifold of $\Ds^s(\R^n)$.\\
Finally to define our notion of solutions of \eqref{E}, for any given $T > 0$, we introduce the spaces $C^k\big([0,T];H^s(\R^n;\R^n)\big)$ resp. $C^k\big([0,T];\Ds^s(\R^n)\big)$, $k \geq 0$. These are the spaces of $C^k$-curves from the closed interval $[0,T]$ to $H^s(\R^n;\R^n)$ resp. $\Ds^s(\R^n)$. On $C^k\big([0,T];H^s(\R^n;\R^n)\big)$ we define the norm
\[
 ||f||_{k,s}=\max_{0 \leq j \leq k} \sup_{t \in [0,T]} ||\partial_t^j f(t)||_s
\]
and on $C^k\big([0,T];\Ds^s(\R^n)\big)$ we consider the metric induced by this norm, i.e. for any $\varphi, \psi \in \Ds^s(\R^n)$
\[
 \operatorname{dist}(\varphi,\psi) = \max_{0 \leq j \leq k} \sup_{t \in [0,T]} ||\partial_t^j \big(\varphi(t)-\psi(t)\big)||_s.
\]
One can show in a straight forward way that $C^k\big([0,T];H^s(\R^n;\R^n)\big)$ is a Banach space and that $C^k\big([0,T];\Ds^s(\R^n)\big)$ can be naturally identified with an open subset of $C^k\big([0,T];H^s(\R^n;\R^n)\big)$. Furthermore it is convenient to introduce the vector spaces
\begin{multline*}
 C^k\big([0,T);H^s(\R^n;\R^n)\big) \\
 = \big\{ f:[0,T) \to H^s(\R^n;\R^n) \;\big|\; \left. f \right|_{[0,\tau]} \in C^k\big([0,\tau];H^s(\R^n;\R^n),\; \forall 0 < \tau < T\big\},
\end{multline*}
\[
 C^\infty\big([0,T);H^s(\R^n;\R^n)\big) = \cap_{k \geq 0} C^k\big([0,T);H^s(\R^n;\R^n)\big)
\]
and
\[
 C^\infty\big([0,T];H^s(\R^n;\R^n)\big) = \cap_{k \geq 0}\; C^k\big([0,T];H^s(\R^n;\R^n)\big).
\]
The sets $C^k\big([0,T);\Ds^s(\R^n)\big)$, $C^\infty\big([0,T);\Ds^s(\R^n)\big)$ and $C^\infty\big([0,T];\Ds^s(\R^n)\big)$ are defined in a similar way.

\section{Well-posedness}\label{section_well_posedness}

In this section we give further explanation on the issue of well-posedness for equation \eqref{E} in the Sobolev space $H^s_\sigma(\R^n;\R^n)$, $s >n/2+1$. Throughout this section we assume that $n \geq 2$. First we need to specify the notion of solution used. Given $T>0$, $s > n/2+1$ and $u_0 \in H^s_\sigma(\R^n;\R^n)$ we say that the pair $(u,p)$ is a solution to \eqref{E} on $[0,T]$, if
\begin{itemize}
\item[(S1)] $u \in C^0\big([0,T];H^s_\sigma(\R^n;\R^n)\big)$ with $u(0)=u_0$
\item[(S2)] $p(t) \in H^{s+1}_{\operatorname{loc}}(\R^n)$ for all $t \in [0,T]$ with $\nabla p \in C^0\big([0,T];H^s(\R^n;\R^n)\big)$
\end{itemize}
and
\begin{itemize}
\item[(S3)] for all $t \in [0,T]$
\begin{equation}\label{integral_E}
 u(t)=u_0 + \int_0^t -\nabla p(\tau) - \big(u(\tau) \cdot \nabla \big) u(\tau) \,d\tau.
\end{equation}
\end{itemize}
Note that in view of (S1), (S2) and the Banach algebra property of $H^{s-1}(\R^n)$ the integrand in (S3) is an element of $C^0\big([0,T];H^{s-1}(\R^n;\R^n)\big)$. Thus one would expect that instead of (S2), one should ask the following weaker statement to hold
\begin{itemize}
\item[(S2')] $p(t) \in H^s_{\operatorname{loc}}(\R^n)$ for all $t \in [0,T]$ with $\nabla p \in C^0\big([0,T];H^{s-1}(\R^n;\R^n)\big)$.
\end{itemize}
But the following lemma of Kato \cite{kato} says that it is natural to assume (S2). 

\begin{Lemma}\label{nabla_p}
Assume that $(u,p)$ satisfies (S1),(S2') and (S3). Then (S2) holds.
\end{Lemma}

\begin{proof}
Taking the divergence of both sides of the identity \eqref{integral_E} by differentiating under the integral sign, one gets for any $0 \leq t \leq T$
\[
 0 = \int_0^t \left(-\Delta p - \operatorname{div}\big((u \cdot \nabla) u\big)\right)\;d\tau.
\]
Hence, in $H^{s-2}(\R^n)$,
\[
 -\Delta p - \operatorname{div}\big((u \cdot \nabla) u\big) = 0.
\]
As $\operatorname{div} u=0$, one gets
\[
 \operatorname{div}\big((u \cdot \nabla) u\big) = \sum_{1 \leq j,k \leq n} \partial_j u_k \partial_k u_j + u_k \partial_k \partial_j u_j = \sum_{1 \leq j,k \leq n} \partial_j u_k \partial_k u_j.
\]
Hence
\begin{equation}\label{poisson_eq}
-\Delta p = \sum_{1 \leq j,k \leq n} \partial_j u_k \partial_k u_j
\end{equation}
is in $H^{s-1}(\R^n)$. From the properties of the Poisson equation (see e.g. \cite{gilbarg})
\[
 ||\partial_\ell \partial_q p||_{s-1} \leq C ||\sum_{1 \leq j,k \leq n} \partial_j u_k \partial_k u_j||_{s-1} \leq C' ||u||_s^2,\quad \forall 1 \leq \ell,q \leq n.
\]
Combined with assumption (S2'), it follows that (S2) holds.
\end{proof}

\noindent
Actually much more can be said. Namely, consider the Leray-projector (see \cite{majda})
\[
 P:H^{s-1}(\R^n;\R^n) \to H^{s-1}_\sigma(\R^n;\R^n)
\]
given by the $L^2$-orthogonal projection. Then by Corollary \ref{coro_nablap}, $P \nabla p = 0$. Assume that $u$ satisfies (S1)-(S3). Applying the linear $L^2$-projection $1-P$ to equation \eqref{integral_E} we then get for any $0 \leq t \leq T$
\[
0 = \int_0^t \left((1-P)\big((u \cdot \nabla) u\big)+\nabla p\right) \;d\tau.
\]
Hence
\begin{equation}\label{nabla_eq}
\nabla p = -(1-P) \big((u \cdot \nabla) u\big)
\end{equation}
in $C^0\big([0,T];H^{s-1}(\R^n;\R^n)\big)$. Therefore $||\nabla p||_{s-1}$ is bounded up to a constant by
\[
 ||\big((u \cdot \nabla) u||_{s-1} \leq C_1 ||u||_s^2.
\]
for some $C_1 > 0$. Combined with \eqref{poisson_eq} there is $C_2 > 0$ such that $||\nabla p||_s \leq C_2 ||u||_s^2$. This estimate for $||\nabla p||_s$ is crucial for rewriting the PDE \eqref{E} as an ODE. The relation \eqref{nabla_eq} shows also that $\nabla p$ is uniquely determined by $u$.

\section{Eulerian coordinates}\label{section_eulerian_formulation}

Inspired by an approach pioneered by Ebin-Marsden \cite{ebin} and Chemin \cite{chemin} we want to describe in this section a system of equations in $H^s(\R^n;\R^n)$, $s > n/2+1$, whose flow leaves $H_\sigma^s(\R^n;\R^n)$ invariant and describes when restricted to $H^s_\sigma(\R^n;\R^n)$ solutions to the Euler equation \eqref{E}, thus giving an alternative formulation of \eqref{E}. Throughout this section we assume $s > n/2 +1$ and $n \geq 2$.\\ \\
To motivate our approach let us argue formally. Assume that $t \mapsto u(t)$ satisfies $\operatorname{div} u = 0$ and
\begin{equation}\label{first_equation}
 \partial_t u + (u \cdot \nabla)u = -\nabla p.
\end{equation}
Taking the divergence on both sides in \eqref{first_equation}
\begin{equation}\label{laplace_p}
-\Delta p = \operatorname{div}\big((u \cdot \nabla)u\big) 
\end{equation}
One would like to write $-\nabla p=\nabla \Delta^{-1} \left(\operatorname{div}\big((u\cdot \nabla)u\big)\right)$ where $\Delta^{-1}$ is the Fourier multiplier operator with multiplier $-|\xi|^{-2}$. Using $\operatorname{div} u = 0$, one has
\begin{equation}\label{first_version}
\operatorname{div}\big((u \cdot \nabla)u\big)=\sum_{1 \leq i,k \leq n} \partial_i u_k \partial_k u_i
\end{equation}
or
\begin{equation}\label{second_version}
\operatorname{div}\big((u \cdot \nabla) u\big) = \sum_{1 \leq i,k \leq n} \partial_{i} \partial_k (u_i u_k).
\end{equation}
Note that for any $1 \leq i,k \leq n$, $\Delta^{-1} \partial_i \partial_k$ is a Fourier multiplier operator with multiplier $\frac{\xi_i \xi_k}{|\xi|^2}$ showing that $\nabla \Delta^{-1} \partial_\ell \partial_k$ defines a linear operator from $H^s(\R^n)$ to $H^{s-1}(\R^n)$. On the other hand $\nabla p \in H^s(\R^n;\R^n)$ leading to a mismatch of the Sobolev indices in the identity
\[
 -\nabla p = \nabla \Delta^{-1} \sum_{1 \leq i,k \leq n} \partial_i \partial_k (u_i u_k).
\]
Note that the expression on the right-hand side of \eqref{first_version} is in $H^{s-1}(\R^n)$ whereas the one on the right-hand side of \eqref{second_version} is in $H^{s-2}(\R^n)$. To represent $\Delta p$ we will use a combination of these two expressions involving a cut-off Fourier multiplier operator. More precisely let $\chi$ be the indicator function of the closed unit ball in $\R^n$, i.e. $\chi(\xi)=1$ for $|\xi| \leq 1$ and $\chi(\xi)=0$ otherwise. The following lemma describes some properties of the Fourier multiplier operator $\chi(D)$. Recall that for $f \in L^2(\R^n)$ the map $\chi(D)$ is given by
\[
 \chi(D) f =\mathcal F^{-1}\big(\chi(\xi) \hat f(\xi)\big)
\]
where $\mathcal F^{-1}$ denotes the inverse Fourier transform. Clearly $\chi(D)$ defines a bounded linear operator on $L_{\mathbb C}^2(\R^n)$.

\begin{Lemma}\label{properties_chi}
The operator $\chi(D)$ leaves $L^2(\R^n)$ invariant and is $L^2$-symmetric, i.e. for all $f,g \in L^2(\R^n)$
\[
 \int_{\R^n} \chi(D)f(x) \cdot g(x) \,dx = \int_{\R^n} f(x) \cdot \chi(D) g(x) \,dx.
\]
In addition for any $s,s' \geq 0$
\[
 ||\chi(D) f||_{s+s'} \leq 2^{s'/2} ||f||_s, \quad \forall f \in H^s(\R^n)
\]
\end{Lemma}

\begin{proof}
The first claimed statement says that for any $f \in L^2(\R^n)$, $\chi(D)f$ is real valued. By the inversion formula for the Fourier transform
\[
 \chi(D)f(x) = \frac{1}{(2\pi)^{n/2}} \int_{\R^n} \chi(\xi) \hat f(\xi) e^{ix \cdot \xi} \,d\xi.
\]
Taking the complex conjugate and using that $\overline{\hat f(\xi)}=\hat f(-\xi)$ as $f$ is real valued, we get
\[
 \overline{\chi(D)f(x)} = \frac{1}{(2\pi)^{n/2}} \int_{\R^n} \chi(\xi) \hat f(-\xi) e^{-ix \cdot \xi}\,d\xi.
\]
By the change of variable $\xi \mapsto -\xi$, it then follows that
\[
 \overline{\chi(D)f(x)} = \frac{1}{(2\pi)^{n/2}} \int_{\R^n} \chi(\xi) \hat f(\xi) e^{ix \cdot \xi} \,d\xi=\chi(D)f(x)
\]
showing that $\chi(D)f(x)$ is real-valued. By Plancherel's theorem (see e.g. \cite{majda}), for all $f,g \in L^2(\R^n)$
\begin{multline*}
 \int_{\R^n} \chi(D)f(x) \cdot g(x) \, dx = \int_{\R^n} \chi(\xi)\hat f(\xi) \overline{\hat g(\xi)} \,d\xi \\
= \int_{\R^n} \hat f(\xi) \overline{\chi(\xi) \hat g(\xi)} \,d\xi = \int_{\R^n} f(x) \cdot \chi(D) g(x) \, dx 
\end{multline*}
proving the $L^2$-symmetry. Finally for any $s,s' \geq 0$ and $f \in H^s(\R^n)$
\[
 ||\chi(D) f||^2_{s+s'} = \int_{\R^n} (1+|\xi|^2)^{s+s'} \chi^2(\xi) |\hat f(\xi)|^2 \,d\xi \leq 2^{s'} \int_{\R^n} (1+|\xi|^2)^s |\hat f(\xi)|^2 \,d\xi = 2^{s'} ||f||^2_s  
\]
where the inequality holds due to the fact that $\chi(\xi)$ is supported in the unit ball.
\end{proof}

In view of the quadratic expressions in the components of $u$ in the formulas of \eqref{first_version} and \eqref{second_version} we want to define the gradient field $-\nabla p$ as $\nabla B(u)$ where $B(u)$ is a continuous quadratic form
\[
 B:H^s(\R^n;\R^n) \to H^{s+1}(\R^n),\quad v \mapsto B(v).
\]
Similar as in \cite{chemin}, we define $B$ as a sum of two quadratic forms $B_1+B_2$. By \eqref{laplace_p}-\eqref{second_version} 
\[
 -\Delta p = \chi(D) \sum_{1 \leq i,k \leq n} \partial_i \partial_k (u_i u_k) + \big(1-\chi(D)\big) \sum_{1 \leq i,k \leq n} \partial_i u_k \partial_k u_i.
\]
Denote by $P_{ik}$ the Fourier multiplier operator
\[
 P_{ik}=\chi(D) \Delta^{-1} \partial_i \partial_k.
\]
Its Fourier multiplier is given by $\chi(\xi) \xi_i \xi_k /|\xi|^2$ and by Lemma \ref{properties_chi}, $P_{ik}$ is a bounded linear operator from $H^s(\R^n)$ to $H^{s+1}(\R^n)$. In view of the Banach algebra property of $H^s(\R^n)$ we see that
\begin{multline*}
 B_1:H^s(\R^n;\R^n) \to H^{s+1}(\R^n),\\
 v=(v_j)_{1 \leq j \leq n} \mapsto \sum_{1 \leq i,k \leq n} P_{ik}(v_i v_k)=\sum_{1 \leq i,k \leq n} \chi(D) \Delta^{-1} \partial_i \partial_k (v_i v_k)
\end{multline*}
is a continuous quadratic form. To define $B_2(v)$ we introduce for $(f,g) \in H^s(\R^n) \times H^s(\R^n)$
\[
 Q_{ik}(f,g) = \Delta^{-1} \big( 1- \chi(D)\big) (\partial_i f \partial_k g),\quad 1 \leq i,k \leq n
\]
where $\Delta^{-1} \big(1-\chi(D)\big)$ is again a Fourier multiplier operator with multiplier $-|\xi|^{-2} \big(1-\chi(\xi)\big)$. Note that it defines a bounded linear operator,
\begin{equation}\label{two_regular}
 \Delta^{-1} \big(1-\chi(D)\big):H^{s-1}(\R^n) \to H^{s+1}(\R^n).
\end{equation}

\noindent
Indeed for any $f \in H^{s-1}(\R^n)$
\begin{multline*}
||\Delta^{-1}\big(1-\chi(D)\big) f||_{s+1}^2 =  \int_{\R^n} (1+|\xi|^2)^{s+1} |\xi|^{-4} \big(1-\chi(\xi)\big)^2 |\hat f(\xi)|^2 \,d\xi \\
\leq 4 \int_{\R^n} (1+|\xi|^2)^{s-1} |\hat f(\xi)|^2 \,d\xi = 4 ||f||_{s-1}^2.
\end{multline*}
as $(1+|\xi|^2)^2 \leq 4 |\xi|^4$ for $|\xi| \geq 1$. By the Banach algebra property of $H^{s-1}(\R^n)$ we see that 
\begin{multline*}
 B_2:H^s(\R^n;\R^n) \to H^{s+1}(\R^n;\R^n),\\
 v=(v_j)_{1 \leq j \leq n} \mapsto \sum_{1 \leq i,k \leq n} Q_{ik}(v_k,v_i)=\sum_{1 \leq i,k \leq n} \Delta^{-1} \big(1-\chi(D)\big) (\partial_i v_k \partial_k v_i)
\end{multline*}
is continuous. Finally we define $B$ as the sum of $B_1$ and $B_2$,
\begin{multline}\label{B_formula}
 B:H^s(\R^n;\R^n) \to H^{s+1}(\R^n),\\
v \mapsto B_1(v) + B_2(v) = \sum_{i,k=1}^n \chi(D) \Delta^{-1} \partial_i \partial_k (v_i v_k) + \Delta^{-1}\big(1-\chi(D)\big) (\partial_i v_k \partial_k v_i)
\end{multline}
and introduce the following initial value problem
\begin{equation}\label{RE}
 \partial_t u + (u \cdot \nabla)u=\nabla B(u),\quad u(0)=u_0.
\end{equation}
We point out that instead of a pressure term, \eqref{RE} contains $B(u)$. A continuous curve $u:[0,T] \to H^s(\R^n;\R^n)$, $T>0$, is called a (local in time) solution to \eqref{RE} with initial value $u_0 \in H^s(\R^n;\R^n)$ if
\begin{equation}\label{RE_int}
u(t) = u_0 + \int_0^t \nabla B\big(u(\tau)\big) - \big(u(\tau) \cdot \nabla\big) u(\tau) \,d\tau, \quad \forall t \in [0,T].
\end{equation}
Note that the integrand in \eqref{RE_int} is in $C^0\big([0,T];H^{s-1}(\R^n;\R^n)\big)$ due to the Banach algebra property of $H^{s-1}(\R^n)$.\\ 
We finish this section by showing that \eqref{RE} gives an alternative formulation of \eqref{E}. First we prove

\begin{Lemma}\label{lemma_dont_miss}
Let $(u,p)$ satisfy (S1)-(S3). Then $u$ is a solution to \eqref{RE_int}.
\end{Lemma}

\begin{proof}[Proof of Lemma \ref{lemma_dont_miss}]
By \eqref{poisson_eq} one has in $H^{s-1}(\R^n)$
\[
 -\Delta p = -\operatorname{div} (\nabla p)= \sum_{i,k=1}^n \partial_i u_k \partial_k u_i.
\]
\noindent
On the other hand using \eqref{B_formula}
\begin{eqnarray*}
\label{div_B_splitted}
\operatorname{div} \nabla B(u)&=&\Delta B(u) = \sum_{i,k=1}^n \chi(D) \big(\partial_i \partial_k (u_i u_k)\big) + \big(1-\chi(D)\big) (\partial_i u_k \partial_k u_i) \\
\noalign{\noindent and hence using that $\operatorname{div} u =0$}
\label{div_B}
 \operatorname{div} \nabla B(u) &=& \sum_{i,k=1}^n \partial_i u_k \partial_k u_i. 
\end{eqnarray*}
Thus $-\Delta p=\Delta B(u)$. Therefore each component of $\nabla B(u)+\nabla p$ is harmonic. By (S2), $\nabla p$ is in $H^s(\R^n;\R^n)$ and by \eqref{B_formula}, so is $\nabla B(u)$. Therefore we have actually $\nabla B(u)=-\nabla p$. Combined with (S3) it follows that $u$ satisfies \eqref{RE_int}.
\end{proof}

Now let us prove the converse of Lemma \ref{lemma_dont_miss}.

\begin{Lemma}\label{lemma_closedness}
Let $u_0 \in H_\sigma^s(\R^n;\R^n)$ and assume that $u \in C^0\big([0,T];H^s(\R^n;\R^n)\big)$ is a solution to \eqref{RE_int} for some $T>0$ with initial value $u_0$. Then
\[
 u(t) \in H_\sigma^s(\R^n;\R^n), \quad \forall t \in [0,T]
\]
and $\big(u,-B(u)\big)$ satisfies (S1)-(S3).
\end{Lemma}

\begin{proof}[Proof of Lemma \ref{lemma_closedness}]
To show that $\operatorname{div} u(t)=0$ for any $0 \leq t \leq T$ it suffices to prove that $\partial_t \operatorname{div} u = 0$. Applying $\operatorname{div}$ to \eqref{RE_int} and using the assumption $\operatorname{div} u_0=0$ one gets
\begin{equation}\label{div_int}
\operatorname{div}u = \int_0^t \sum_{i,k=1}^n \operatorname{div}\nabla B(u)  - \operatorname{div} \big((u \cdot \nabla)u\big)\, d\tau.
\end{equation}
By \eqref{B_formula}
\[
 \operatorname{div} \nabla B(u) = \Delta B(u) = \sum_{i,k=1}^n \chi(D) \big(\partial_i \partial_k (u_i u_k)\big) + \big(1-\chi(D)\big) (\partial_i u_k \partial_k u_i).
\]
Note that
\[
\sum_{i,k=1}^n \partial_i \partial_k (u_i u_k) = 2 (u \cdot \nabla) \operatorname{div} u + (\operatorname{div} u)^2 + \sum_{i,k=1}^n \partial_i u_k \partial_k u_i.
\]
Therefore
\begin{multline}
\chi(D) \sum_{i,k=1}^n \partial_i \partial_k (u_i u_k) + \big(1-\chi(D)\big) \sum_{i,k=1}^n \partial_i u_k \partial_k u_i \\
= \chi(D) 2 (u \cdot \nabla) \operatorname{div} u + \chi(D) (\operatorname{div} u)^2 + \sum_{i,k=1}^n \partial_i u_k \partial_k u_i.
\end{multline}
Furthermore
\[
 -\operatorname{div}\big( (u \cdot \nabla) u\big) = -(u \cdot \nabla)\operatorname{div} u - \sum_{i,k=1}^n \partial_i u_k \partial_k u_i.
\]
Substituting the two identities above into \eqref{div_int} one gets in $L^2$ (as $n \geq 2$, one has $s > n/2+1 \geq 2$)
\begin{equation}\label{u_derivative}
 \partial_t \operatorname{div} u = \chi(D)\big( 2 (u \cdot \nabla)\operatorname{div} u + (\operatorname{div} u)^2\big) - (u \cdot \nabla) \operatorname{div}u.
\end{equation}
Denoting by $\langle \cdot,\cdot \rangle_{L^2}$ the inner product $\int_{\R^n} f g\,dx$ for two real valued $L^2$-functions we have
\begin{equation}\label{L2_derivative}
 \frac{1}{2} \partial_t \langle \operatorname{div}u,\operatorname{div} u \rangle_{L^2} = 2I + II - III
\end{equation}
where
\begin{eqnarray*}
 I &=& \langle \operatorname{div} u,\chi(D) \big((u \cdot \nabla)\operatorname{div}u\big)\rangle_{L^2}\\ 
II &=& \langle \operatorname{div} u,\chi(D) (\operatorname{div} u)^2\rangle_{L^2}\\
III &=& \langle \operatorname{div} u, (u \cdot \nabla) \operatorname{div} u\rangle_{L^2}
\end{eqnarray*}
We now estimate the terms on the right-hand side of \eqref{L2_derivative} seperately for each fixed $0 \leq t \leq T$. To estimate the term $III$ we integrate by parts
\begin{eqnarray*}
 \langle \operatorname{div}u, (u \cdot \nabla) \operatorname{div} u \rangle_{L^2} &=& -\langle \sum_{k=1}^n \partial_k (u_k \operatorname{div} u), \operatorname{div}u \rangle_{L^2}\\
 &=& -\langle (\operatorname{div}u)^2,\operatorname{div}u \rangle_{L^2} - \langle \operatorname{div}u, (u \cdot \nabla) \operatorname{div}u \rangle_{L^2}.
\end{eqnarray*}
Thus we get
\begin{equation}\label{integ_by_parts}
 \langle \operatorname{div}u, (u \cdot \nabla)\operatorname{div}u \rangle_{L^2} = -\frac{1}{2} \langle (\operatorname{div}u)^2, \operatorname{div}u \rangle_{L^2}.
\end{equation}
Using the imbedding $H^{s-1}(\R^n) \hookrightarrow C_0^0(\R^n)$ it follows from \eqref{integ_by_parts} that 
\[
 |(\operatorname{div} u, (u \cdot \nabla) \operatorname{div} u)_{L^2}| \leq \frac{1}{2} ||\operatorname{div}u||_{L^\infty} ||\operatorname{div} u||^2_{L^2}.
\]
To estimate the term $I$ use the $L^2$-symmetry of $\chi(D)$, established in Lemma \ref{properties_chi}, to get 
\begin{multline*}
 \langle \operatorname{div}u, \chi(D) (u \cdot \nabla)\operatorname{div} u \rangle_{L^2} = \langle \chi(D) \operatorname{div}u , (u \cdot \nabla)\operatorname{div}u \rangle_{L^2}\\
 = -\langle \sum_{k=1}^n \partial_k (u_k \chi(D) \operatorname{div} u), \operatorname{div} u \rangle_{L^2}\\
 = -\langle (\operatorname{div}u) \chi(D) \operatorname{div} u, \operatorname{div} u \rangle_{L^2} 
 - \langle (u \cdot \nabla) \chi(D) \operatorname{div}u, \operatorname{div}u \rangle_{L^2}.
\end{multline*}
Hence 
\[
 2I + II = -\langle \chi(D)\operatorname{div} u,(\operatorname{div}u)^2\rangle_{L^2} - 2 \langle (u \cdot \nabla) \chi(D) \operatorname{div} u,\operatorname{div} u\rangle_{L^2}
\]
or
\begin{eqnarray*}
|2I+II| &\leq & ||\chi(D) \operatorname{div} u||_{L^2} ||\operatorname{div}u||_{L^\infty} ||\operatorname{div} u||_{L^2} + 2 ||(u \cdot \nabla) \chi(D) \operatorname{div} u||_{L^2} ||\operatorname{div} u||_{L^2}\\
&\leq & ||\operatorname{div} u||_{L^\infty} ||\operatorname{div} u||_{L^2}^2 + 2 ||(u \cdot \nabla) \chi(D) \operatorname{div} u||_{L^2} ||\operatorname{div} u||_{L^2}.
\end{eqnarray*}
Using Lemma \ref{properties_chi} and the imbedding $H^s(\R^n) \hookrightarrow C_0^0(\R^n)$ once more leads to
\[
 ||(u \cdot \nabla )\chi(D) \operatorname{div} u||_{L^2} \leq \sqrt{2} ||u||_{L^\infty} ||\operatorname{div}u||_{L^2}.
\]
Summarizing the above inequalities, we conclude that for any $0 \leq t \leq T$
\[
 \partial_t ||\operatorname{div}u||_{L^2}^2 \leq 3\big(||\operatorname{div}u||_{L^\infty} + ||u||_{L^\infty}\big) ||\operatorname{div}u||_{L^2}^2.
\]
Using the imbedding $H^s(\R^n) \hookrightarrow C_0^1(\R^n)$ we conclude that
\[
 C:= \sup_{0 \leq t \leq T} \big(||\operatorname{div} u(t)||_{L^\infty} + ||u(t)||_{L^\infty}\big) < \infty.
\]
As $\operatorname{div}u(0)=0$ we then get by Gronwall's inequality (see e.g. \cite{majda}) that $\operatorname{div}u(t)=0$ for all $t \in [0,T]$.
\end{proof}

We combine the results proved in Lemma \ref{lemma_dont_miss} and Lemma \ref{lemma_closedness} in the following

\begin{Prop}\label{prop_alternative}
For any solution $u \in C^0\big([0,T];H^s(\R^n;\R^n)\big)$ with $u(0)=u_0 \in H^s_\sigma(\R^n;\R^n)$ of \eqref{RE_int} the pair $\big(u,-B(u))$ satisfies (S1)-(S3). Conversely any pair $(u,p)$ satisfying (S1)-(S3) gives rise to a solution of \eqref{RE_int}.
\end{Prop}

\begin{Rem}\label{rem_alternative}
In view of Proposition \ref{prop_alternative}, in the sequel, we study equation \eqref{RE} rather than \eqref{E}.
\end{Rem}

\chapter{Lagrangian formulation}\label{section_lagrangian_formulation}

\section{Geodesic equation on $\Ds^s(\R^n)$}\label{section_geodesic_equation}

In this section we express the alternative formulation of the Euler equation \eqref{RE} in Lagrangian coordinates due to Arnold \cite{arnold}. It turns out that in such coordinates \eqref{RE} takes the form of a geodesic equation on $\Ds^s(\R^n)$, $s > n/2 +1$. This will be of use in the proof of Theorem \ref{th_lagrangian}. Throughout this section we assume that $n \geq 2$ and $s > n/2+1$.\\
To motivate Arnold's approach let us make the following formal calculations. Let $u$ be a solution of \eqref{RE}. Then the corresponding flow $\varphi$, given by $\partial_t \varphi=u \circ \varphi$, satisfies
\begin{equation}\label{derivation_geodesic}
 \partial_t^2 \varphi = \partial_t (u \circ \varphi) = \partial_t u \circ \varphi +\left[ (u \cdot \nabla)u \right] \circ \varphi = R_\varphi \left(\nabla B(\partial_t \varphi \circ \varphi^{-1})\right).
\end{equation}
Recall that the geodesics on a Riemannian manifold of dimension $d$ satisfy in local coordinates the following second-order ODE (see e.g. \cite{lang})
\[
 \ddot x_i = - \sum_{j,k=1}^d \Gamma_{jk}^i(x) \dot x_k \dot x_j, \quad i=1,\ldots,d
\]
where $\Gamma_{jk}^i$ are the Christoffel symbols. Note that the right-hand side of the latter equation is a quadratic form in $\dot x_k$, $1 \leq k \leq d$. In analogy, we write \eqref{derivation_geodesic} as
\begin{equation}\label{geodesic_eq}
\partial_t^2 \varphi = \Gamma_\varphi(\partial_t \varphi,\partial_t \varphi)
\end{equation}
where
\begin{equation}\label{christoffel_map}
\Gamma_\varphi(v,v) = R_\varphi \left( \nabla B(v \circ \varphi^{-1})\right).
\end{equation}

\noindent
The goal of this section is to show that $\Gamma$, defined by \eqref{christoffel_map} on appropriate Sobolev spaces, defines a real analytic map in $\varphi$ and $v$. For the notion of analyticity in real Banach spaces, we refer the reader to Appendix \ref{appendix_analyticity}. \\
We start with the definition of the following subspace of $L^2(\R^n)$ which will be very important in the sequel,
\[
 H_\Xi^\infty(\R^n) = \big\{ f \in L^2(\R^n) \; \big| \; \operatorname{supp} \hat f \subseteq \Xi \big\}
\]
where $\Xi=\{ \xi \in \R^n \;|\; |\xi| \leq 1 \}$ is the closed unit ball in $\R^n$. The following lemma describes some properties of $H_\Xi^\infty(\R^n)$.

\begin{Lemma}\label{lemma_hb_space}
The space $H_\Xi^\infty(\R^n)$ is a closed subspace of $L^2(\R^n)$. It consists of entire functions. Furthermore we have
\[
 H_\Xi^\infty(\R^n) \subseteq \cap_{s \geq 0} H^s(\R^n)
\]
with
\begin{equation}\label{hb_estimate}
 ||f||_s \leq 2^{s/2} ||f||_{L^2}  
\end{equation}
for all $f \in H_\Xi^\infty(\R^n)$.
\end{Lemma}

\begin{proof}
Let $(f_k)_{k \geq 1} \subseteq H_\Xi^\infty(\R^n)$ with $f_k \to f$ in $L^2(\R^n)$. Then by Plancherel's theorem, $\hat f_k \to \hat f$ in $L^2(\R^n)$. Hence for any compact subset $A \subseteq \R^n \setminus \Xi$ we have
\[
 0 = \int_A |\hat f_k| \to \int_A |\hat f| \quad \mbox{as }  k \to \infty.
\]
Hence $\int_A |\hat f| = 0$, showing that $f \in H_\Xi^\infty(\R^n)$. By the inversion formula for the Fourier transform we have
\[
 f(x) = \frac{1}{(2\pi)^{n/2}} \int_\Xi e^{ix \cdot \xi} \hat f(\xi) \,d\xi.
\]
As the domain of integration $\Xi$ is compact, $f$ is well-defined for any $x \in \mathbb C$ and is an entire function. The inequality \eqref{hb_estimate} follows from
\begin{equation*}
 ||f||_s^2 = \int_\Xi (1+|\xi|^2)^s |\hat f(\xi)|^2 \,d\xi \leq 2^s \int_\Xi |\hat f(\xi)|^2 \,d\xi = 2^s ||f||_{L^2}^2.
\end{equation*}
\end{proof}

\begin{Rem}\label{rem_vector_H_K}
Lemma \ref{lemma_hb_space} shows in particular that $\big(H_\Xi^\infty(\R^n),\langle \cdot,\cdot \rangle_{L^2}\big)$ is a Hilbert space. In the sequel we will also use the vector valued analog of $H_\Xi^\infty(\R^n)$
\[
 H_\Xi^\infty(\R^n;\R^n):=\big\{ f=(f_1,\ldots,f_n) \;\big|\; f_k \in H_\Xi^\infty(\R^n),\quad 1\leq k \leq n \big\}
\]
with norm $||f||_{L^2}=\big(||f_1||_{L^2}^2 + \cdots + ||f_n||_{L^2}^2\big)^{1/2}$.
\end{Rem}

\noindent
We have the following regularity property for the composition map

\begin{Prop}\label{prop_analytic_composition}
The composition map 
\[
 H_\Xi^\infty(\R^n) \times \Ds^s(\R^n) \to H^s(\R^n),\quad (f,\varphi) \mapsto f \circ \varphi
\]
is real analytic with radius of convergence $R=\infty$.
\end{Prop}

Recall from Section \ref{functional_setting} that the differential structure of $\Ds^s(\R^n)$ is given by identifying it with the open set $\Ds^s(\R^n) - \operatorname{id} \subseteq H^s(\R^n;\R^n)$.

\begin{proof}
 Let $f \in H_\Xi^\infty(\R^n)$. From Lemma \ref{lemma_hb_space} we know that $f$ is an entire function and hence admits a power series expansion for any $x,y \in \R^n$
\[
 f(x+y)=\sum_{|\alpha| \geq 0} \frac{1}{\alpha !} \partial^\alpha f(x) y^\alpha
\]
where we use the multi-index notation, i.e. for a multi-index $\alpha=(\alpha_1,\ldots,\alpha_n) \in \mathbb Z^n_{\geq 0}$, 
\[
 \partial^\alpha f(x) = \partial_1^{\alpha_1} \cdots \partial_n^{\alpha_n} f(x),\quad \alpha!=\alpha_1! \cdots \alpha_n !
\]
and for $y=(y_1,\ldots,y_n) \in \R^n$, $y^\alpha= y_1^{\alpha_1} \cdots y_n^{\alpha_n}$. For the derivative $\partial^\alpha f$ we have from \eqref{hb_estimate}
\[ 
 ||\partial^\alpha f||_s \leq ||f||_{s+|\alpha|} \leq 2^{(s+|\alpha|)/2} ||f||_{L^2}.
\]
Writing $\varphi = \operatorname{id} + g$, $g=(g_1,\ldots,g_n) \in H^s(\R^n;\R^n)$ we have with the notation $g^\alpha(x)=g_1^{\alpha_1}(x) \cdots g_n^{\alpha_n}(x)$, pointwise for all $x \in \R^n$
\[
 f\big(\varphi(x)\big) = f\big(x+g(x)\big) = \sum_{|\alpha| \geq 0} \frac{1}{\alpha!} \partial^\alpha f(x) g^\alpha(x)
\]
or formally as an identity in $\mathcal L\big(H_\Xi^\infty(\R^n),H^s(\R^n)\big)$
\begin{equation}\label{power_series}
 f \mapsto f \circ \varphi \equiv f \mapsto \sum_{k \geq 0} Q_k(g)(f)
\end{equation}
where $Q_k(g)$ is a linear differential operator of order $k$ whose coefficients are homogeneous polynomials in the components of $g \in H^s(\R^n;\R^n)$, $Q_k(g)=\sum_{|\alpha|=k} \frac{1}{\alpha!} g^\alpha \partial^\alpha$, acting on functions $f \in H_\Xi^\infty(\R^n)$ as
\[
 Q_k(g)(f) = \sum_{|\alpha|=k} \frac{1}{\alpha!} g^\alpha \partial^\alpha f.
\]
Note that $Q_k(g):H_\Xi^\infty(\R^n) \to H^s(\R^n)$ is a bounded linear map. Indeed we have by the Banach algebra property of $H^s(\R^n)$ for any multi-index $\alpha$ with $|\alpha|=k$
\[
 ||g^\alpha \partial^\alpha f||_s \leq C^k ||g||_s^k ||\partial^\alpha f||_s \leq C^k ||g||_s^k ||f||_{s+k}
\]
and hence by \eqref{hb_estimate}
\[
 ||Q_k(g)(f)||_s \leq \left(\sum_{|\alpha|=k} \frac{1}{\alpha!}\right) C^k 2^{(s+k)/2} ||f||_{L^2} ||g||_s^k  = \frac{1}{k!} n^k C^k 2^{(2+k)/2} ||f||_{L^2} ||g||_s^k.
\]
Here we used that by the multinomial theorem,
\[
 \sum_{|\alpha|=k} \frac{k!}{\alpha!} = (1+\ldots+1)^k = n^k.
\]
Recall from Appendix \ref{appendix_analyticity}, \eqref{radius_condition},
\[
 ||Q_k|| = \sup_{\mbox{\scriptsize $\begin{array}{c} ||g||_s \leq 1 \\ f \in H_\Xi^\infty(\R^n)\\ ||f||_{L^2} \leq 1 \end{array}$}} ||Q_k(g)(f)||_s.
\]
Altogether we have proved that $||Q_k|| \leq \frac{1}{k!} n^k C^k 2^{(2+k)/2}$ leading to 
\[
 \sup_{k \geq 0} ||Q_k|| r^k < \infty
\]
for all $r > 0$. Therefore the series \eqref{power_series} has convergence radius $R=\infty$. Now the pointwise limit $f \circ \varphi$ and the $H^s$-limit $\sum_{k \geq 0} Q_k(g)(f)$ must coincide. Thus we see that
\[
 \Phi:\Ds^s(\R^n) \to \mathcal L\big(H_\Xi^\infty(\R^n),H^s(\R^n)\big),\quad \varphi \mapsto \big[f \mapsto f \circ \varphi = \left(\sum_{k \geq 0} Q_k(g)\right) (f)\big]
\]
is real analytic. Again we identify here $\varphi$ with $g=\varphi - \operatorname{id}$. Finally this shows also that
\[
 H_\Xi^\infty(\R^n) \times \Ds^s(\R^n),\quad (f,\varphi) \mapsto \Phi(\varphi)(f) 
\]
is analytic as evaluation is an analytic operation.
\end{proof}

\begin{Rem}\label{rem_radius_infty}
The proof shows that the composition map in Proposition \ref{prop_analytic_composition}, which apriori is merely defined for $g=\varphi - \operatorname{id}$ in $\Ds^s(\R^n)-\operatorname{id} \subseteq H^s(\R^n;\R^n)$, can be extended analytically to all of $H^s(\R^n;\R^n)$ and that the resulting map is real analytic with convergence radius $R=\infty$. 
\end{Rem}

Recall from Section \ref{section_eulerian_formulation} that $\chi(D)$ is the Fourier multiplier operator with Fourier multiplier $\chi(\xi)$ given by the characteristic function of the unit ball $\Xi$ in $\R^n$.

\begin{Coro}\label{coro_composition_inverse}
For any $0 \leq s' \leq s$, the map
\[
 \Ds^s(\R^n) \to \mathcal L\big(H^{s'}(\R^n),H_\Xi^\infty(\R^n)\big),\quad \varphi \mapsto \left[ f \mapsto \chi(D) (f \circ \varphi^{-1})\right]
\]
is real analytic with radius of convergence $R=\infty$.
\end{Coro}

\begin{proof}
 First we consider
\begin{equation}\label{joint_composition}
 H^{s'}(\R^n) \times \Ds^s(\R^n) \to H_\Xi^\infty(\R^n),\quad (f,\varphi) \mapsto \chi(D)(f \circ \varphi^{-1})
\end{equation}
and show that it is weakly analytic -- see Remark \ref{rem_weakly_analytic}. Choose any $g \in H_\Xi^\infty(\R^n)$ and consider
\begin{equation}\label{analytic_pairing}
 \langle \chi(D) (f \circ \varphi^{-1}),g \rangle_{L^2}
\end{equation}
Note that $\chi(D)g=g$. As by Lemma \ref{properties_chi},
\[
 \langle \chi(D)(f\circ \varphi^{-1}),g \rangle_{L^2} = \langle f \circ \varphi^{-1}, \chi(D) g \rangle_{L^2}
\]
it then follows after a change of variable of integration $y=\varphi^{-1}(x)$ that
\begin{equation}\label{weak_integral}
 \int_{\R^n} f\circ \varphi^{-1} \cdot g \,dx = \int_{\R^n}f \cdot g \circ \varphi \cdot \det(d_y \varphi)\, dy.
\end{equation}
By Proposition \ref{prop_analytic_composition} it follows that
\[
 \Ds^s(\R^n) \to H^{s'}(\R^n),\quad \varphi \mapsto g \circ \varphi
\]
is real analytic with convergence radius $R=\infty$. In addition $\Ds^s(\R^n) \to H^{s-1}(\R^n)$, $\varphi \mapsto \det(d_x \varphi) -1$ is also real analytic with radius of convergence $R=\infty$, since it is a polynomial. Altogether one then concludes that the expression on the right-hand side of \eqref{weak_integral} is real analytic in $(f,\varphi)$ with radius of convergence $R=\infty$. As $g$ was arbitrary, we conclude from Proposition \ref{prop_weak_analytic} that \eqref{joint_composition} is real analytic with radius of convergence $R=\infty$. As the map \eqref{joint_composition} is linear in $f$ the claim of the corollary follows from Lemma \ref{lemma_linear}. 
\end{proof}

\noindent
Here and in the following for any real Hilbert spaces $X$ and $Y$ we denote by $L^2(X;Y)$ the space of continuous bilinear forms on $X$ with values in $Y$. For $B$ defined in \eqref{B_formula} let $\tilde B(f,g)=\frac{1}{4}\big(B(f+g)-B(f-g)\big)$. Note that $\tilde B$ is in $L^2\big(H^s(\R^n;\R^n);H^{s+1}(\R^n)\big)$. Similarly, denote by $\tilde B_1$, $\tilde B_2$ the symmetric bilinear forms corresponding to $B_1$ resp. $B_2$. Note that $\tilde B=\tilde B_1 + \tilde B_2$. Now we state the main proposition of this section.

\begin{Prop}\label{prop_analytic_spray}
The map
\begin{eqnarray*}
 \Gamma: \Ds^s(\R^n) &\to& L^2\big(H^s(\R^n;\R^n);H^s(\R^n;\R^n)\big)\\
 \varphi &\mapsto& \left[ (f,g) \mapsto R_\varphi \nabla \tilde B (f \circ \varphi^{-1},g \circ \varphi^{-1}) \right]
\end{eqnarray*}
is real analytic.
\end{Prop}

\noindent
We split the proof into two lemmas.

\begin{Lemma}\label{lemma_b1_analytic}
The map
\begin{eqnarray*}
 \Gamma_1:\Ds^s(\R^n) &\to& L^2\big(H^s(\R^n;\R^n);H^s(\R^n;\R^n)\big) \\
 \varphi &\mapsto& \left[(f,g) \mapsto R_\varphi \nabla \tilde B_1(f \circ \varphi^{-1},g \circ \varphi^{-1}) \right]
\end{eqnarray*}
is real analytic.
\end{Lemma}

\begin{proof}
 Note that for $f,g \in H^s(\R^n;\R^n)$
 \[
  \tilde B_1(f,g) = \sum_{i,k=1}^n \chi(D) \Delta^{-1} \partial_i \partial_k (f_i g_k)=\sum_{i,k=1}^n \Delta^{-1} \partial_i \partial_k \chi(D) (f_i g_k)
 \]
where $\Delta^{-1} \partial_i \partial_k$ is the Fourier multiplier operator with symbol $|\xi|^{-2} \xi_i \xi_k$. It follows from Corollary \ref{coro_composition_inverse} that
\[
 \Ds^s(\R^n) \to L^2\big(H^s(\R^n;\R^n);H_\Xi^\infty(\R^n)\big),\quad \varphi \mapsto \left[ (f,g) \mapsto \chi(D) R_\varphi^{-1} (f_i g_k) \right]
\]
is real analytic for any $1 \leq i,k \leq n$. Since 
\[
 \nabla \Delta^{-1} \partial_i \partial_k: H_\Xi^\infty(\R^n) \to H_\Xi^\infty(\R^n;\R^n)
\]
is a continuous linear map one sees that
\[
 \Ds^s(\R^n) \to L^2\big(H^s(\R^n;\R^n);H_\Xi^\infty(\R^n;\R^n)\big), \quad \varphi \mapsto \left[ (f,g) \mapsto \nabla \tilde B_1(f \circ \varphi^{-1},g \circ \varphi^{-1})\right]
\]
is real analytic. Finally we get by Proposition \ref{prop_analytic_composition} that
\[
 \Ds^s(\R^n) \to L^2\big(H^s(\R^n;\R^n);H^s(\R^n;\R^n)\big), \quad \varphi \mapsto \left[ (f,g) \mapsto R_\varphi \nabla \tilde B_1(f \circ \varphi^{-1},g \circ \varphi^{-1})\right]
\]
is real analytic as claimed.
\end{proof}

\begin{Lemma}\label{lemma_b2_analytic}
The map
\begin{eqnarray*}
\Gamma_2:\Ds^s(\R^n) &\to& L^2\big(H^s(\R^n;\R^n);H^s(\R^n;\R^n)\big) \\
\varphi &\mapsto & \left[ (f,g) \mapsto R_\varphi \nabla \tilde B_2(f \circ \varphi^{-1},g \circ \varphi^{-1}) \right]
\end{eqnarray*}
is real analytic.
\end{Lemma}

\begin{proof}
Note that for $f,g \in H^s(\R^n;\R^n)$
\[
 \tilde B_2(f,g) = \sum_{i,k=1}^n \Delta^{-1} \big(1-\chi(D)\big) (\partial_i f_k \partial_k g_i).
\]
It turns out to be convenient to write $\Delta^{-1}\big(1-\chi(D)\big)$ as $-\chi(D) + \left(\chi(D)+\Delta^{-1}\big(1-\chi(D)\big)\right)$. One can show that $\chi(D)+\Delta^{-1}\big(1-\chi(D)\big)$ maps $H^{s-2}(\R^n)$ to $H^s(\R^n)$ -- cf \eqref{two_regular}. Then
\[
 A:=\chi(D) + \Delta \big(1-\chi(D)\big)
\]
is its inverse and hence $\Delta^{-1}\big(1-\chi(D)\big)=-\chi(D)+A^{-1}$
First we show that
\[
\Ds^s(\R^n) \to \mathcal L\big(H^s(\R^n),H^{s-2}(\R^n)\big),\quad \varphi \mapsto \left[ f \mapsto R_\varphi A (f \circ \varphi^{-1}) \right]
\]
is real analytic. To this end note that by Corollary \ref{coro_composition_inverse} and Proposition \ref{prop_analytic_composition}
\[
 \Ds^s(\R^n) \to \mathcal L\big(H^s(\R^n),H^{s-2}(\R^n)\big),\quad \varphi \mapsto \left[ f \mapsto R_\varphi \chi(D) (f \circ \varphi^{-1}) \right]
\]
is real analytic. To handle the second summand in $A$ we write 
\[
 R_\varphi \Delta \big(1-\chi(D)\big) R_\varphi^{-1} = (R_\varphi \Delta R_\varphi^{-1}) \left(R_\varphi \big(1-\chi(D)\big) R_\varphi^{-1}\right).
\]
As $R_\varphi \big(1-\chi(D)\big) R_\varphi^{-1} = \operatorname{Id}-R_\varphi \chi(D) R_\varphi^{-1}$ one concludes from above that
\[
 \Ds^s(\R^n) \to \mathcal L\big(H^s(\R^n),H^s(\R^n)\big),\quad \varphi \mapsto \left[ f \mapsto R_\varphi \big(1-\chi(D)\big) (f \circ \varphi^{-1})\right]
\]
is real analytic. Next we show that
\[
 \Ds^s(\R^n) \to \mathcal L\big(H^s(\R^n),H^{s-2}(\R^n)\big),\quad \varphi \mapsto \left[ f \mapsto R_\varphi \Delta (f \circ \varphi^{-1})\right]
\]
is also real analytic. Indeed we have for $1 \leq k \leq n$
\[
 R_\varphi \partial_k (f \circ \varphi^{-1}) = R_\varphi \sum_{j=1}^n \partial_j f \circ \varphi^{-1} \cdot C_{jk} \circ \varphi^{-1} = \sum_{j=1}^n \partial_j f \cdot C_{jk} 
\]
where $C$ is the inverse of the Jacobian $d\varphi$, $C=(d\varphi)^{-1}$. Note that $C$ is of the form
\[
 C=\frac{1}{\det(d\varphi)} (c_{st})_{1 \leq s,t \leq n}
\]
where $c_{st}$ are polynomial expressions in terms of the entries in $(d\varphi)$. As $H^{s-1}(\R^n)$ is a Banach algebra, $c_{st}$ are analytic expressions in $\varphi$. By Lemma \ref{analytic_det} division by $\det(d\varphi)$ is an analytic operation. Thus by Lemma \ref{lemma_multiplication},
\[
 \Ds^s(\R^n) \to \mathcal L\big(H^{s'}(\R^n),H^{s'-1}(\R^n)\big),\quad \varphi \mapsto \left[ f \mapsto R_\varphi \partial_k (f \circ \varphi^{-1})\right]
\]
is real analytic for any $1 \leq s' \leq s$. Writing for any $1 \leq \ell,k \leq n$, $R_\varphi \partial_\ell \partial_k (f \circ \varphi^{-1})$ as $(R_\varphi \partial_\ell R_\varphi^{-1}) \big(R_\varphi \partial_k (f \circ \varphi^{-1})\big)$ one iterates the argument above to conclude that
\begin{equation}\label{conjugation}
 \Ds^s(\R^n) \to \mathcal L\big(H^s(\R^n),H^{s-2}(\R^n)\big),\quad \varphi \mapsto \left[ f \mapsto R_\varphi \partial_\ell \partial_k (f \circ \varphi^{-1}) \right]
\end{equation}
is real analytic. Thus we see that
\[
 \Ds^s(\R^n) \to \mathcal L\big(H^s(\R^n),H^{s-2}(\R^n)\big),\quad \varphi \mapsto [f \mapsto R_\varphi \Delta (f \circ \varphi^{-1})]
\]
is real analytic. Summarizing the results so far, we have shown that $\varphi \mapsto R_\varphi A R_\varphi^{-1}$ is real analytic as a map from $\Ds^s(\R^n)$ to $\mathcal L\big(H^s(\R^n),H^{s-2}(\R^n)\big)$. Using Neumann series (see e.g. \cite{dieudonne}) one sees that the inversion operator
\[
 \operatorname{inv}:GL\big(H^s(\R^n),H^{s-2}(\R^n)\big) \to GL\big(H^{s-2}(\R^n),H^s(\R^n)\big),\quad T \mapsto T^{-1}
\]
is also real analytic. Here for any Hilbert spaces $X,Y$, $GL(X,Y)$ denotes the open subset of $\mathcal L(X,Y)$ of all invertible continuous linear operators from $X$ to $Y$. Using that $R_\varphi A^{-1} R_\varphi^{-1} = \operatorname{inv}(R_\varphi A R_\varphi^{-1})$ it then follows that
\[
 \Ds^s(\R^n) \to \mathcal L\big(H^{s-2}(\R^n),H^s(\R^n)\big),\quad \varphi \mapsto [f \mapsto R_\varphi A^{-1} R_\varphi^{-1}(f)]
\]
is real analytic. Recall that $\Delta^{-1}\big(1-\chi(D)\big)=-\chi(D) + A^{-1}$ and hence in view of Corollary \ref{coro_composition_inverse} and Proposition \ref{prop_analytic_composition} it follows that 
\begin{equation}\label{analytic_pseudo_inverse}
\Ds^s(\R^n) \to \mathcal L\big(H^{s-2}(\R^n),H^s(\R^n)\big),\quad \varphi \mapsto [f \mapsto R_\varphi \Delta^{-1} \big(1-\chi(D)\big) R_\varphi^{-1}(f)]
\end{equation}
is real analytic. Now we are ready to prove the statement of the lemma. Since $\nabla$ and $\Delta^{-1}\big(1-\chi(D)\big)$ commute, we can write $\nabla \tilde B_2$ as
\[
 \nabla \tilde B_2(f,g) = \sum_{i,k=1}^n \Delta^{-1} \big(1-\chi(D)\big) \nabla (\partial_i f_k \partial_k g_i).
\]
Write $R_\varphi \nabla \tilde B_2(f\circ \varphi^{-1},g \circ \varphi^{-1})$ in the form
\begin{equation}\label{b2_rewritten}
\sum_{i,k=1}^n  \left( R_\varphi \Delta^{-1} \big(1-\chi(D)\big) R_\varphi^{-1} \right) \circ \left(R_\varphi \nabla \big(\partial_i (f_k \circ \varphi^{-1}) \partial_k (g_i \circ \varphi^{-1})\big)\right).
\end{equation}
By \eqref{conjugation} we have for any $1 \leq i,k \leq n$ that
\begin{eqnarray*}
\Ds^s(\R^n) &\to& L^2\big(H^s(\R^n;\R^n);H^{s-2}(\R^n;\R^n)\big),\\
\varphi &\mapsto& \left[ (f,g) \mapsto R_\varphi \nabla \big(\partial_i (f_k \circ \varphi^{-1}) \partial_k (g_i \circ \varphi^{-1})\big)\right]
\end{eqnarray*}
is real analytic. Combining \eqref{analytic_pseudo_inverse} and \eqref{b2_rewritten} the lemma follows.
\end{proof}

\begin{proof}[Proof of Proposition \ref{prop_analytic_spray}]
As $\Gamma=\Gamma_1 + \Gamma_2$ the claim of the proposition follows from Lemma \ref{lemma_b1_analytic} and Lemma \ref{lemma_b2_analytic}.
\end{proof}

\section{The exponential map}\label{section_exponential_map}

The goal of this section is to prove Theorem \ref{th_lagrangian} with $V$ given by $V(\varphi,v)=\big(v,\Gamma_\varphi(v,v)\big)$ where $\Gamma$ is the Christoffel map given by \eqref{christoffel_map}. Throughout this section we assume that $n \geq 2$ and $s > n/2+1$. Let us recall the geodesic equation from \eqref{geodesic_eq}
\begin{equation}\label{geo_eq}
 \partial_t^2 \varphi = \Gamma_\varphi(\partial_t \varphi,\partial_t \varphi).
\end{equation}
To show Theorem \ref{th_lagrangian} we want to write the solutions of \eqref{geo_eq} as a flow. To this end we write the geodesic equation \eqref{geo_eq} as a first order ODE in $\Ds^s(\R^n) \times H^s(\R^n;\R^n)$
\begin{equation}\label{first_order}
\partial_t \left(\begin{array}{c}\varphi \\ v \end{array}\right) = \left(\begin{array}{c} v \\ \Gamma_\varphi(v,v) \end{array}\right)
\end{equation}
and consider as initial values $\varphi(0)=\operatorname{id} \in \Ds^s(\R^n)$ and $v(0)=v_0 \in H^s(\R^n;\R^n)$. By Proposition \ref{prop_analytic_spray}, the vector field $(\varphi,v) \mapsto \big(v,\Gamma_\varphi(v,v)\big)$ is real analytic. It then follows from the existence and uniqueness theorem of ODE's (cf Proposition \ref{prop_analytic_flow}) that there exists a time $\tau > 0$ and a ball $B_\delta(0) \subseteq H^s(\R^n;\R^n)$ with radius $\delta > 0$ and center $0$ such that \eqref{first_order} has a unique solution on $[0,\tau]$ for all initial values $v_0 \in B_\delta(0)$. For \eqref{first_order} we have the following scaling property: If $(\varphi,v):[0,\tau] \to \Ds^s(\R^n) \times H^s(\R^n;\R^n)$ is a solution to \eqref{first_order} with $(\varphi(0),v(0))=(\operatorname{id},v_0)$ then for all $\lambda > 0$
\begin{equation}\label{scaled_solution}
\big(\varphi(\lambda\cdot),\lambda v(\lambda \cdot)\big):[0,\tau/\lambda] \to \Ds^s(\R^n) \times H^s(\R^n;\R^n)
\end{equation}
is also a solution to \eqref{first_order} with initial condition $(\operatorname{id},\lambda v_0)$. Thus we see that the solutions to \eqref{first_order} exist on $[0,1]$ for all initial values $(\operatorname{id},v_0)$ with $v_0 \in B_{\delta \cdot \tau}(0)$. Let
\begin{equation}\label{uexp_def}
 U^s_{\exp} \subseteq H^s(\R^n;\R^n)
\end{equation}
denote the set of $v_0$ such that the solution to $\eqref{first_order}$ with initial value $\varphi(0)=\operatorname{id}, v(0)=v_0$ exists longer than time $1$. By Proposition \ref{prop_analytic_flow} $U^s_{\exp}$ is open. Moreover from the scaling property \eqref{scaled_solution} we see that $U^s_{\exp}$ is star-shaped in $H^s(\R^n;\R^n)$ with respect to $0$. Now define the flow map, referred to as exponential map,
\begin{equation}\label{def_exp}
 \exp:U^s_{\exp} \to \Ds^s(\R^n),\quad v_0 \mapsto \varphi(1;v_0)
\end{equation}
where here $\big(\varphi(t;v_0),v(t;v_0)\big)$ denotes the solution to \eqref{first_order} with initial values $(\operatorname{id},v_0)$. From Proposition \ref{prop_analytic_spray} and Proposition \ref{prop_analytic_flow} we immediately get
\begin{Prop}\label{exp_analytic}
The exponential map
\[
 \exp:U^s_{\exp} \to \Ds^s(\R^n)
\]
is real analytic.
\end{Prop}

\noindent
The following lemma describes how the map $\exp$ describes the solutions of \eqref{geo_eq}.

\begin{Lemma}\label{lemma_exp}
Consider the second order initial value problem 
\begin{equation}\label{want_to_solve}
 \partial_t^2 \varphi = \Gamma_\varphi(\partial_t \varphi,\partial_t \varphi),\qquad \varphi(0)=\operatorname{id},\;\partial_t \varphi(0)=u_0 \in H^s(\R^n;\R^n).
\end{equation}
Then its solution on the maximal interval of existence is given by
\begin{equation}\label{exp_solution}
 \psi(t):=\exp(t u_0) \quad \mbox{for} \quad 0 \leq t < T^\ast
\end{equation}
where $T^\ast=\sup \{t \,| \, t u_0 \in U^s_{\exp} \}$. 
\end{Lemma}

\begin{proof}
Take $\psi$ as given in \eqref{exp_solution}. By the definition of $\exp$ we have
\[
 \psi(t) = \varphi(1;tu_0)
\]
where $\varphi(1;w_0)$, $w_0 \in U^s_{\exp}$, denotes the time $1$ value of $\varphi$ solving \eqref{geo_eq} with $\varphi(0)=\operatorname{id}$ and $\partial_t \varphi(0)=w_0$. By the scaling property \eqref{scaled_solution} we have
\begin{equation}\label{exp_rescaled}
 \psi(t) = \varphi(1;tu_0) = \varphi(t;u_0)
\end{equation}
for all $0 \leq t < T^\ast$. Thus by \eqref{first_order}
\[
 \partial_t^2 \psi(t)=\Gamma_{\psi(t)}\big(\partial_t \psi(t),\partial_t \psi(t)\big)
\]
for all $0 \leq t < T^\ast$ with initial values $\psi(0)=\operatorname{id}$ and $\partial_t \psi(0)=u_0$. By the uniqueness statement of Proposition \ref{prop_analytic_flow} we get that $\psi$ agrees with the solution \eqref{want_to_solve} on $[0,T^\ast)$. From the definition of $T^\ast$ it follows that it is the maximal time of existence for \eqref{want_to_solve}.
\end{proof}

An immediate consequence of \eqref{exp_rescaled} is

\begin{Lemma}\label{d_exp}
The derivative of $\exp$ at $0 \in H^s(\R^n;\R^n)$ is the identity map, 
\[
 d_0 \exp (v) = v,\quad \forall v \in H^s(\R^n;\R^n).
\]
\end{Lemma}

\begin{proof}
By the definition of the derivative
\[
 d_0 \exp(v) = \left. \partial_\varepsilon \right|_{\varepsilon=0} \exp(\varepsilon v).
\]
From \eqref{exp_rescaled} it follows that
\[
 \left. \partial_\varepsilon \right|_{\varepsilon=0} \exp(\varepsilon v)=v
\]
showing the claim.
\end{proof}

Now we can give the proof of Theorem \ref{th_lagrangian}.

\begin{proof}[Proof of Theorem \ref{th_lagrangian}]
(i) Note that the vector field $V$ can be expressed in terms of $\Gamma$, $V(\varphi,v)=\big(v,\Gamma_\varphi(v,v)\big)$. Item (i) then follows from Proposition \ref{prop_analytic_spray}.\\
\noindent
(ii) Take $u_0 \in H^s(\R^n;\R^n)$ and consider
\[
 \varphi(t):=\exp(tu_0), \quad 0 \leq t < T^\ast
\]
where $T^\ast=\sup \{t \, | \, t u_0 \in U^s_{\exp}\}$. We claim that $u(t):=(\partial_t \varphi)\big(t,\varphi^{-1}(t)\big)$ is a solution to \eqref{RE}. By Proposition \ref{prop_analytic_flow}, $\varphi$ is in the space $C^\infty\big([0,T^\ast);\Ds^s(\R^n)\big)$ introduced at the end of Section \ref{functional_setting}. Take any $T \in (0,T^\ast)$. By the continuity of the composition and the inversion map we know that $u \in C^0\big([0,T];H^s(\R^n;\R^n)\big)$.  By the Sobolev imbedding \eqref{sobolev_imbedding}, $\varphi, \partial_t \varphi \in C^1\big([0,T] \times \R^n;\R^n)$. It then follows that $u \in C^1([0,T] \times \R^n;\R^n)$. Indeed from the implicit function theorem $\varphi^{-1}(t,x)$, as a solution of the equation $\varphi(t,y)=x$ in $y$, lies in $C^1([0,T] \times \R^n;\R^n)$. Thus by the chain rule,
\[
 u=\partial_t \varphi\big(t,\varphi^{-1}(t)\big) \in C^1([0,T] \times \R^n;\R^n).
\]
Hence pointwise at any $(t,x) \in [0,T] \times \R^n$
\[
 \partial_t^2 \varphi = \partial_t u \circ \varphi + \left[ (u \cdot \nabla) u \right] \circ \varphi = \Gamma_\varphi(\partial_t \varphi,\partial_t \varphi)
\]
or, by the definition \eqref{christoffel_map} of $\Gamma$,
\[
 R_\varphi \big( \partial_t u + (u \cdot \nabla) u \big) = R_\varphi \nabla B(u)
\]
showing that we have pointwise
\[
 \partial_t u + (u \cdot \nabla) u =\nabla B(u).
\]
By the fundamental lemma of calculus it then follows that for any $x \in \R^n$ (dropping the $x$ in the argument)
\begin{equation}\label{pointwise_identity}
u(t) = u_0 + \int_0^t \nabla B(u(\tau)) - (u(\tau) \cdot \nabla) u(\tau) \,d\tau.
\end{equation}
Note that as $u$ is in $C^0\big([0,T];H^s(\R^n;\R^n)\big)$ it follows from \eqref{B_formula} that the integrand lies in $C^0\big([0,T];H^{s-1}(\R^n;\R^n)\big)$. Thus by the definition \eqref{RE_int}, $u$ is a solution to \eqref{RE}, which proves the claim. By Proposition \ref{prop_alternative} we then get that $u$ is a solution to \eqref{E} in the case $u_0 \in H^s_\sigma(\R^n;\R^n)$.
\end{proof}

\chapter{Applications}\label{section_applications}

\section{Local well-posedness of the Euler equation}\label{section_local_wellposedness}

As a first application of the Lagrangian formulation of Chapter \ref{section_lagrangian_formulation} we present an alternative proof of the local well-posedness result of equation \eqref{E} due to Kato \cite{kato}, stated in Theorem \ref{th_kato}. In view of Proposition \ref{prop_alternative} it follows from the following more general result.

\begin{Th}\label{th_general}
 Let $s > n/2+1$ and $n \geq 2$. Then the initial value problem \eqref{RE} is locally well-posed in $H^s(\R^n;\R^n)$, i.e. for any $w \in H^s(\R^n;\R^n)$, there is a neighborhood $W \subseteq H^s(\R^n;\R^n)$ of $w$ and $T > 0$ such that for any $u_0 \in W$, the equation \eqref{RE} has a unique solution $u(\cdot;u_0) \in C^0\big([0,T];H^s(\R^n;\R^n)\big)$ and moreover
\[
 W \to C^0\big([0,T];H^s(\R^n;\R^n)\big),\quad u_0 \mapsto u(\cdot;u_0)
\]
is continuous.
\end{Th}

\begin{proof}
Let $w \in H^s(\R^n;\R^n)$ be an arbitrary initial value. By Theorem \ref{th_lagrangian}(i), proved in Section \ref{section_exponential_map}, there is a $T>0$ and a neighborhood $W \subseteq H^s(\R^n;\R^n)$ of $w$ such that for any $u_0 \in W$, there exists a solution $\varphi(\cdot;u_0)$ of \eqref{E_ivp} on $[0,T]$. By the proof of Theorem \ref{th_lagrangian}(ii)
\[
 u(t;u_0):=\partial_t \varphi(t;u_0) \circ \varphi^{-1}(t;u_0)
\]
is a solution to \eqref{RE}. Moreover we know by the continuity of the group operations in $\Ds^s(\R^n)$ that
\[
 U \to C^0\big([0,T];H^s(\R^n;\R^n)\big),\quad u_0 \mapsto u(\cdot;u_0)
\]
is continuous. To prove local well-posedness of \eqref{RE} it thus remains to show the uniqueness of solutions of \eqref{RE}. Assume that we have for some $T > 0$ a solution $u \in C^0\big([0,T];H^s(\R^n;\R^n)\big)$ to problem \eqref{RE} with $u(0)=u_0 \in H^s(\R^n;\R^n)$. By Proposition \ref{prop_integration} there exists a unique flow $\varphi \in C^1\big([0,T];\Ds^s(\R^n)\big)$ solving
\begin{equation}\label{phi_ode}
 \partial_t \varphi = u \circ \varphi;\quad \varphi(0)=\operatorname{id}.
\end{equation}
Note that in particular $\varphi \in C^1([0,T] \times \R^n;\R^n)$. We claim that $\varphi$ is a geodesic, i.e. it solves \eqref{geo_eq}. As $u$ is a solution to \eqref{RE} we get by the Sobolev imbedding \eqref{sobolev_imbedding} 
\[
 u \in C^1([0,T] \times \R^n;\R^n).
\]
As $\partial_t \varphi = u \circ \varphi$ we also have $\partial_t \varphi \in C^1([0,T] \times \R^n;\R^n)$. Taking the $t$-derivative in \eqref{phi_ode} we get pointwise
\[
 \partial_t^2 \varphi = \partial_t u \circ \varphi + [ (u \cdot \nabla) u ] \circ \varphi = R_\varphi [ \partial_t u + (u \cdot \nabla) u]
\]
Using that $u$ solves \eqref{RE}, one gets $\partial_t^2 \varphi = R_\varphi \nabla B(u)$ or by \eqref{christoffel_map}
\begin{equation}\label{get_geodesic}
 \partial_t^2 \varphi = \Gamma_\varphi(\partial_t \varphi,\partial_t \varphi).
\end{equation}
From the fundamental lemma of calculus we thus have pointwise
\[
 \partial_t \varphi = \operatorname{id} + \int_0^t \Gamma_\varphi(\partial_t \varphi,\partial_t \varphi) \,d\tau, \quad t \in [0,T].
\]
The integrand is in $C^0\big([0,T];H^s(\R^n;\R^n)\big)$. This shows that \eqref{get_geodesic} is actually an identity in $H^s$. Hence $\varphi$ is a geodesic with $\varphi(0)=\operatorname{id}$ and $\partial_t \varphi(0)=u_0$. Now assume that $u_1$ and $u_2$ are two solutions on the interval $[0,T]$ with $u_1(0)=u_2(0)=u_0$. Denote by $\varphi_1$ and $\varphi_2$ the corresponding flows introduced above. By the uniqueness theorem for ODE's one concludes that $\varphi_1(t)=\varphi_2(t)$ for any $0 \leq t \leq T$. As
\[
 u_i = \left(\partial_t \varphi_i\right) \circ \varphi_i^{-1}, \quad i=1,2
\]
it thus follows that $u_1(t)=u_2(t)$ for all $t \in [0,T]$. 
\end{proof}

\noindent
By Lemma \ref{lemma_exp} and by the proof of Theorem \ref{th_general} we get

\begin{Coro}\label{flow_exp}
Let $u$ be a solution of equation \eqref{RE} on $[0,1]$ with $u(0)=u_0 \in H^s(\R^n;\R^n)$. Then $u_0 \in U^s_{\exp}$ and moreover the flow of $u$ on $[0,1]$ is given by
\[
 \varphi(t)=\exp(tu_0), \qquad 0 \leq t \leq 1
\]
where $\exp$ is the exponential map and $U^s_{\exp}$ its domain defined as in \eqref{def_exp}.
\end{Coro}

\begin{proof}[Proof of Theorem \ref{th_kato}]
The local-wellposedness of \eqref{E} follows from Proposition \ref{prop_alternative} and Theorem \ref{th_general}. 
\end{proof}

\section{Proof of Theorem \ref{th_not_uniform} and Theorem \ref{th_not_differentiable}}\label{section_proof_main}

In this section we will prove Theorem \ref{th_not_uniform} and \ref{th_not_differentiable}.\\ \\
First introduce the following important quantity.
\begin{Def}\label{def_vorticity}
Let $u=(u_1,\ldots,u_n) \in C^1(\R^n;\R^n)$. We define the vorticity $\Omega \in C^0(\R^n;\R^{n \times n})$ of $u$ as the following $n \times n$ skew-symmetric matrix with coefficients given by
\begin{equation}\label{vorticity_formula}
 \Omega_{ij}(u)=\partial_j u_i - \partial_i u_j.
\end{equation}
\end{Def}
Note that $\Omega = du - du^\top$, where $du^\top$ denotes the transposed matrix of the Jacobian $du$. For more details on the vorticity see e.g. \cite{chemin,majda}. By the Biot-Savart law (see Lemma \ref{lemma_biot_savart}) we can recover the vector field from its vorticity. Moreover we have the following conservation law (see also \cite{arnold}).

\begin{Prop}\label{prop_conserved}
Let $s > n/2+1$ with $n \geq 2$ and $u \in C^0\big([0,T];H^s(\R^n;\R^n)\big)$ be a solution to \eqref{RE}, $\varphi \in C^\infty\big([0,T];\Ds^s(\R^n)\big)$ the corresponding flow and $\Omega(t):=\Omega\big(u(t)\big)$ the corresponding vorticity. Then
\begin{equation}\label{omega_regular}
 \Omega \text{ is in } C^0\big([0,T];H^{s-1}(\R^n;\R^{n \times n})\big) \cap C^1\big([0,T];H^{s-2}(\R^n;\R^{n \times n})\big)
\end{equation}
and
\begin{equation}\label{conserved}
d\varphi(t)^\top \cdot \Omega(t) \circ \varphi(t) \cdot d\varphi(t) = \Omega(0)
\end{equation}
for all $0 \leq t \leq T$.
\end{Prop} 

\begin{proof}
As by assumption $u$ is a solution of \eqref{RE} in $C^0\big([0,T];H^s(\R^n;\R^n)\big)$, $\partial_t u$ is in $C^0\big([0,T];H^{s-1}(\R^n;\R^n)\big)$ and hence we have \eqref{omega_regular}. Note that the claimed identity \eqref{conserved} holds for $t=0$. Hence it holds for any $t$ if the $t$-derivative of the left-hand side vanishes. We begin by some auxiliary computations. Consider the system of equations \eqref{RE} 
\begin{equation}\label{euler_components}
\partial_t u_i + \sum_{k=1}^n u_k \partial_k u_i = \partial_i B(u)
\end{equation}
for $1 \leq i \leq n$ with $B$ as in \eqref{B_formula}. Taking the partial derivative with respect to $x_j$ in \eqref{euler_components} we get in view of $s > n/2+1$ and Lemma \ref{lemma_multiplication} in $L^2$
\[
\partial_t \partial_j u_i + \sum_{k=1}^n \partial_j u_k \partial_k u_i + \sum_{k=1}^n u_k \partial_k \partial_j u_i =  \partial_j \partial_i B(u)
\]
for all $1 \leq i,j \leq n$. Subtracting 
\[
 \partial_t \partial_i u_j + \sum_{k=1}^n \partial_i u_k \partial_k u_j + \sum_{k=1}^n u_k \partial_k \partial_i u_j =  \partial_i \partial_j B(u)
\]
from the identity above, we get the following matrix equation in $L^2$
\[
 \partial_t \Omega + du \cdot du - du^\top \cdot du^\top + (u \cdot \nabla) \Omega = 0 
\]
or
\begin{equation}\label{vorticity_equation}
 \partial_t \Omega + \Omega \cdot du + du^\top \cdot \Omega + (u \cdot \nabla) \Omega = 0
\end{equation}
where the differential operator $u \cdot \nabla$ acts componentwise. In order to compute the $t$-derivative of the left-hand side of \eqref{conserved} we now compute the one of $d\varphi$. Taking the derivative of both sides of $\partial_t \varphi=u \circ \varphi$ one gets
\begin{equation}\label{dphi_t}
 \partial_t d\varphi = du \circ \varphi \cdot d\varphi.
\end{equation}
To compute the $t$-derivative of $\Omega \circ \varphi$ we approximate $\Omega$ by a sequence $(\Omega_k)_{k \geq 1}$ in $C^1\big([0,T];H^s(\R^n;\R^{n \times n})\big)$ with 
\[
 \Omega = \lim_{k \to \infty} \Omega_k \text{ in } C^0\big([0,T];H^{s-1}(\R^n;\R^{n \times n})\big) \cap C^1\big([0,T];H^{s-2}(\R^n;\R^{n \times n})\big). 
\]
By the Sobolev imbedding \eqref{sobolev_imbedding} $\Omega_k \in C^1([0,T] \times \R^n;\R^{n \times n})$ for any $k \geq 1$. For the $t$-derivate of $\Omega_k \circ \varphi$ we get pointwise at any $(t,x) \in [0,T] \times \R^n$,
\[
 \partial_t (\Omega_k \circ \varphi) = (\partial_t \Omega_k) \circ \varphi + \big((u \cdot \nabla) \Omega_k\big) \circ \varphi.
\]
This identity holds also in $L^2$. Letting $k \to \infty$ we get by Lemma \ref{lemma_multiplication} and Lemma \ref{lemma_continuous_right_translation} the following identity in $L^2$
\begin{equation}\label{omega_phi_t}
 \partial_t (\Omega \circ \varphi) = (\partial_t \Omega) \circ \varphi + \big((u \cdot \nabla) \Omega \big) \circ \varphi.
\end{equation}
By \eqref{dphi_t}-\eqref{omega_phi_t} we have
\begin{multline*}
 \partial_t \left(d\varphi^\top \cdot \Omega \circ \varphi \cdot d\varphi\right) = d\varphi^\top \cdot du^\top \circ \varphi \cdot \Omega \circ \varphi \cdot d\varphi + d\varphi^\top \cdot \left(\partial_t \Omega\right) \circ \varphi \cdot d\varphi \\
+ d\varphi^\top \cdot \left((u \cdot \nabla) \Omega \right) \circ \varphi \cdot d\varphi + d\varphi^\top \cdot \Omega \circ \varphi \cdot du \circ \varphi \cdot d\varphi
\end{multline*}
which is equal to
\[
 d\varphi^\top \cdot R_\varphi \left(\partial_t \Omega + du^\top \cdot \Omega + \Omega \cdot du + (u \cdot \nabla) \Omega\right) \cdot d\varphi.
\]
By \eqref{vorticity_equation} the latter expression vanishes and hence \eqref{conserved} is proved. 
\end{proof}

The following corollary of Proposition \ref{prop_conserved} will be one of the key ingredients in the proof of Theorem \ref{th_not_uniform}. 

\begin{Coro}\label{coro_compact_vorticity}
Under the assumptions of Proposition \ref{prop_conserved}, the vorticity at time $t$ is given by
\begin{equation}\label{conserved_different}
 \Omega(t)=R_{\varphi(t)}^{-1}\left([d\varphi(t)^\top]^{-1} \Omega(0) [d\varphi(t)]^{-1}\right).
\end{equation}
In particular, the support of $\Omega(t,\cdot)$ remains compact if the one of $\Omega(0,\cdot)$ is.
\end{Coro}

We need to estimate the right-hand side of \eqref{conserved_different}.

\begin{Lemma}\label{vorticity_estimate}
Let $s > n/2+1$ with $n \geq 2$ and $\varphi_\bullet \in \Ds^s(\R^n)$. Then there is $C > 0$ and there is a neighborhood $U \subseteq \Ds^s(\R^n)$ of $\varphi_\bullet$ such that
\[
 \frac{1}{C} ||f||_{s-1} \leq ||R_\varphi^{-1} \left([d\varphi^\top]^{-1} f [d\varphi]^{-1}\right)||_{s-1} \leq C ||f||_{s-1}
\]
for all $f \in H^{s-1}(\R^n;\R^{n \times n})$ and for all $\varphi \in U$. 
\end{Lemma}

\begin{proof}
By \cite{composition}, the maps
\begin{eqnarray*}
 \phi:\Ds^s(\R^n) \times H^{s-1}(\R^n;\R^{n \times n}) &\to& H^{s-1}(\R^n;\R^{n \times n})\\
 (\varphi,f) &\mapsto& R_\varphi^{-1}\left([d\varphi^\top]^{-1} f [d\varphi]^{-1}\right)
\end{eqnarray*}
and
\begin{eqnarray*}
\psi:\Ds^s(\R^n) \times H^{s-1}(\R^n;\R^{n \times n}) &\to& H^{s-1}(\R^n;\R^{n \times n}) \\
(\varphi,g) &\mapsto & d\varphi^\top (R_\varphi g) d\varphi
\end{eqnarray*}
are continuous. Note that we have
\begin{equation}\label{phi_inverse}
 \psi\big(\varphi,\phi(\varphi,f)\big)=f
\end{equation}
for all $\varphi \in \Ds^s(\R^n)$ and for all $f \in H^{s-1}(\R^n;\R^{n \times n})$. As $\phi$ is continuous with $\phi(\varphi_\bullet,0)=0$ there exist $\rho_1 > 0$ and a neighborhood $U_1 \subseteq \Ds^s(\R^n)$ of $\varphi_\bullet$ with 
\[
 ||\phi(\varphi,f)||_{s-1} \leq 1
\]
for all $\varphi \in U_1$ and $f \in B_{\rho_1}(0)$, where $B_{\rho_1}(0)$ denotes the ball in $H^{s-1}(\R^n;\R^{n \times n})$ of radius $\rho_1$ centered at $0$. By the linearity of $\phi(\varphi,f)$ with respect to $f$
\[
 ||\phi(\varphi,f)||_{s-1} \leq \frac{1}{\rho_1} ||f||_{s-1}
\]
for all $\varphi \in U_1$ and for all $f \in H^{s-1}(\R^n;\R^{n \times n})$. The same arguments lead to a similar estimate for $\psi$, i.e. there exist $\rho_2 > 0$ and a neighborhood $U_2 \subseteq \Ds^s(\R^n)$ of $\varphi_\bullet$ such that
\[
 ||\psi(\varphi,g)||_{s-1} \leq \frac{1}{\rho_2} ||g||_{s-1}
\]
for all $\varphi \in U_2$ and for all $g \in H^{s-1}(\R^n;\R^{n \times n})$. Thus we get from \eqref{phi_inverse}
\[
 ||f||_{s-1} = ||\psi\big(\varphi,\phi(\varphi,f)\big)||_{s-1} \leq \frac{1}{\rho_2} ||\phi(\varphi,f)||_{s-1}
\]
for all $\varphi \in U_2$ and for all $f \in H^{s-1}(\R^n;\R^{n \times n})$. Then the choice $U=U_1 \cap U_2$ and $C=\max\{1/\rho_1,1/\rho_2\}$ shows the claim. 
\end{proof}

A key ingredient into the proof of Theorem \ref{th_not_uniform} is Proposition \ref{prop_nonuniform} below, which treats the special case $T=1$ of Theorem \ref{th_not_uniform}. We denote by $U \equiv U_1 \subseteq H_\sigma^s(\R^n;\R^n)$ the domain of definition of $\phi\equiv E_1$ of \eqref{ET_def},
\begin{equation}\label{time_one_map}
 \phi:U \to H_\sigma^s(\R^n;\R^n),\quad u_0 \mapsto u(1;u_0)
\end{equation}
where $u(1;u_0)$ denotes the time 1 value of the solution $u$ of \eqref{RE}. Note that by \eqref{def_exp} and Theorem \eqref{th_lagrangian}(ii), $U^s_{\exp} \cap H^s_\sigma(\R^n;\R^n) \subseteq U$ and that by Corollary \ref{flow_exp}, $U \subseteq U^s_{\exp} \cap H^s_\sigma(\R^n;\R^n)$. Altogether we have
\begin{equation}\label{u_eq_uexp}
 U=U^s_{\exp} \cap H^s_\sigma(\R^n;\R^n).
\end{equation}

\begin{Prop}\label{prop_nonuniform}
Assume that $s >n/2+1$ with $n \geq 2$. Then the map $\phi$ is at no point of $U$ locally uniformly continuous.
\end{Prop}

\noindent
Before proving the proposition we need some lemmas.

\begin{Lemma}\label{lemma_almost_identity}
Let $\varphi \in C^0\big([0,1];\Ds^s(\R^n)\big)$. Then for any $\varepsilon > 0$ there is $R > 0$ such that
\[
 |\varphi(t,y) - y| < \varepsilon \quad \mbox{and} \quad |d\varphi(t,y) - I_n| < \varepsilon
\]
for any $t \in [0,1]$ and for any $y \in \R^n$ with $|y| \geq R$. Here $|\cdot|$ denotes the euclidean norm and $I_n$ the $n \times n$ identity matrix. 
\end{Lemma}

\begin{proof}[Proof of Lemma \ref{lemma_almost_identity}]
Note that by the Sobolev imbedding \eqref{sobolev_imbedding}, there exists $C > 0$ such that for any $f \in H^s(\R^n;\R^n)$,
\begin{equation}\label{imbedding_constant}
||f||_{C^1} \leq C ||f||_s.
\end{equation} 
As for any element $\varphi \in \Ds^s(\R^n)$, $\varphi(y)-y$ is in $H^s(\R^n;\R^n)$, it then follows that for any $t_0 \in [0,1]$ there is $R_{t_0} > 0$ such that
\[
 |\varphi(t_0,y) - y| < \varepsilon/2 \quad \mbox{and} \quad |d\varphi(t_0,y) - I_n| < \varepsilon/2
\]
for all $|y| \geq R_{t_0}$. Choose $\delta > 0$ in such a way that for any $t$ in $(t_0-\delta,t_0+\delta) \cap [0,1]$,
\[
 ||\varphi(t)-\varphi(t_0)||_s < \frac{\varepsilon}{2C}
\]
where $C > 0$ is the imbedding constant in \eqref{imbedding_constant}. Then 
\[
 |\varphi(t,y)-y| < \varepsilon \quad \mbox{and} \quad |d\varphi(t,y) - I_n| < \varepsilon
\]
for any $t \in (t_0-\delta,t_0+\delta) \cap [0,1]$ and $|y| \geq R_{t_0}$. Since we can cover $[0,1]$ with finitely many of such intervals we get the claim.
\end{proof}

\begin{Lemma}\label{lemma_dexp} 
Let $U$ be the domain of $\phi$ of \eqref{time_one_map} and $u_0 \in U \cap C^\infty_c(\R^n;\R^n)$. Consider the restriction of the differential of $\exp$ at $u_0$ to $H^s_\sigma(\R^n;\R^n)$, 
\[
 d_{u_0}\exp:H_\sigma^s(\R^n;\R^n) \to H^s(\R^n;\R^n),\quad v_0 \mapsto \left. \partial_\varepsilon \right|_{\varepsilon=0} \exp(u_0 + \varepsilon v_0).
\]
Then there exists $m > 0$ with the following property: For any $R > 0$ there exists $v \in C^\infty_{\sigma,c}(\R^n;\R^n)$ with $|\big(d_{u_0}\exp(v)\big)(x^\ast)| \geq m$, $||v||_s=1$ and support in the ball $B_1(x^\ast)=\{ x \in \R^n \;|\; |x-x^\ast| < 1 \}$ for some $x^\ast \in \R^n$ with $|x^\ast| \geq R$.
\end{Lemma}

\begin{proof}[Proof of Lemma \ref{lemma_dexp}]
Take $w \in C^\infty_{\sigma,c}(\R^n;\R^n)$ with support in the ball $B_1(0)$ and with the properties $||w||_s=1$ and $w(0) \neq 0$. Choose $\delta > 0$ so that $u_0 + \varepsilon w(\cdot - x^\ast) \in U$ for any $|\varepsilon| \leq \delta$ and any $x^\ast \in \R^n$. Now define $w_{x^\ast}:=w(\cdot - x^\ast)$ where $x^\ast$ will be conveniently chosen at the end of the proof. For $\varepsilon \in [-\delta,\delta]$ we denote by $u^{(\varepsilon)}$ the solution of \eqref{RE} with initial data $u_0 + \varepsilon w_{x^\ast}$ and by $\varphi^{(\varepsilon)}$ the corresponding flow. By Corollary \ref{flow_exp}, for $(\varepsilon,t) \in [-\delta,\delta] \times [0,1]$, $\varphi^{(\varepsilon)}(t)$ is given by
\begin{equation}\label{phi_exp}
 \varphi^{(\varepsilon)}(t)=\exp\big(t(u_0+\varepsilon w_{x^\ast})\big).
\end{equation}
Thus by Proposition \ref{exp_analytic}
\begin{equation}\label{exp_c1}
\varphi^{(\cdot)}(\cdot):[-\delta,\delta] \times [0,1] \to \Ds^s(\R^n)
\end{equation}
is $C^1$. Hence, denoting by $I_n$ the $n \times n$-identity matrix,
\begin{equation}\label{dexp_c1}
 d \varphi^{(\cdot)}(\cdot)-I_n:[-\delta,\delta] \times [0,1] \to H^{s-1}(\R^n;\R^{n \times n})
\end{equation}
is $C^1$ as well. By Corollary \ref{coro_compact_vorticity}, the vorticity $\Omega\big(u^{\varepsilon}(t)\big)$ has compact support. Hence by Lemma \ref{lemma_biot_savart} (Biot-Savart law) we have for all $x \in \R^n$ and $t \in [0,1]$, $|\varepsilon| \leq \delta$
\[
 u^{(\varepsilon)}(t,x) = \frac{1}{\omega_n} \int_{\R^n} \Omega\big(u^{(\varepsilon)}(t)\big)(y) \frac{x-y}{|x-y|^n} \,dy.
\]
Now using the conservation law \eqref{conserved} we get
\small
\[
 u^{(\varepsilon)}(t,x) = \frac{1}{\omega_n} \int_{\R^n} R_{\varphi^{(\varepsilon)}(t,y)}^{-1} \Big([d\varphi^{(\varepsilon)}(t,\cdot)^\top]^{-1} \big(\Omega(u_0)+\varepsilon \Omega(w_{x^\ast})\big)(\cdot) [d\varphi^{(\varepsilon)}(t,\cdot)]^{-1}\Big) \frac{x-y}{|x-y|^n}\,dy.
\]
\normalsize
Using that $\varphi^{(\varepsilon)}(t)$ is volume-preserving we have
\small
\[
 u^{(\varepsilon)}(t,x)=\frac{1}{\omega_n} \int_{\R^n} [d\varphi^{(\varepsilon)}(t,y)^\top]^{-1} \big(\Omega(u_0)+\varepsilon \Omega(w_{x^\ast})\big)(y) [d\varphi^{(\varepsilon)}(t,y)]^{-1} \frac{x-\varphi^{(\varepsilon)}(t,y)}{|x-\varphi^{(\varepsilon)}(t,y)|^n} \,dy.
\]
\normalsize
The relation $\varphi^{(\varepsilon)}(1) = \operatorname{id} + \int_0^1 u^{(\varepsilon)}(t) \circ \varphi^{(\varepsilon)}(t) \,dt$ then leads to 
\begin{equation}\label{varphi_representation}
\varphi^{(\varepsilon)}(1,x)=x+I^{(\varepsilon)}(x)
\end{equation}
where $I^{(\varepsilon)}(x)$ is given by
\small
\[
  \frac{1}{\omega_n} \int_0^1 \int_{\R^n} [d\varphi^{(\varepsilon)}(t,y)^\top]^{-1} \big(\Omega(u_0)+\varepsilon \Omega(w_{x^\ast})\big)(y) [d\varphi^{(\varepsilon)}(t,y)]^{-1} \frac{\varphi^{(\varepsilon)}(t,x)-\varphi^{(\varepsilon)}(t,y)}{|\varphi^{(\varepsilon)}(t,x)-\varphi^{(\varepsilon)}(t,y)|^n} \,dy dt.
\]
\normalsize
Write $I^{(\varepsilon)}(x)$ as the sum $I_1^{(\varepsilon)}(x) + \varepsilon I_2^{(\varepsilon)}(x)$, where $I_1^{(\varepsilon)}(x)$ is defined by
\small
\[
\frac{1}{\omega_n} \int_0^1 \int_{y \in \operatorname{supp}u_0} [d\varphi^{(\varepsilon)}(t,y)^\top]^{-1} \Omega(u_0)(y) [d\varphi^{(\varepsilon)}(t,y)]^{-1} \frac{\varphi^{(\varepsilon)}(t,x)-\varphi^{(\varepsilon)}(t,y)}{|\varphi^{(\varepsilon)}(t,x)-\varphi^{(\varepsilon)}(t,y)|^n} \,dy dt
\]
\normalsize
and
\small
\[
I_2^{(\varepsilon)}(x):= \frac{1}{\omega_n} \int_0^1 \int_{\R^n} [d\varphi^{(\varepsilon)}(t,y)^\top]^{-1} \Omega(w_{x^\ast})(y) [d\varphi^{(\varepsilon)}(t,y)]^{-1} \frac{\varphi^{(\varepsilon)}(t,x)-\varphi^{(\varepsilon)}(t,y)}{|\varphi^{(\varepsilon)}(t,x)-\varphi^{(\varepsilon)}(t,y)|^n} \,dy dt.
\]
\normalsize
Next we need to get an expression for 
\[
 \big(d_{u_0}\exp(w_{x^\ast})\big)(x)=\left. \partial_\varepsilon \right|_{\varepsilon=0} \varphi^{(\varepsilon)}(1,x).
\]
This will be accomplished by taking the $\varepsilon$-derivative of the right-hand side of \eqref{varphi_representation}. But first we have to make some preparations.\\
Consider the curves $\varphi:=\left. \varphi^{(\varepsilon)} \right|_{\varepsilon=0}:[0,1] \to \Ds^s(\R^n)$ and $\varphi^{-1}:[0,1] \to \Ds^s(\R^n)$. As by Proposition \ref{exp_analytic}, $U \to C^0\big([0,1];\Ds^s(\R^n)\big)$, $v_0 \to [t \mapsto \exp(t v_0)]$ is continuous, there exist $Q > 0$ and $\delta' > 0$ so that for any $\tilde u_0$ in the $\delta'$-ball $B_{\delta'}(u_0)$ in $H^s_\sigma(\R^n;\R^n)$ centered at $u_0$ one has
\[
 \max_{0 \leq t \leq 1} ||\exp(t \tilde u_0)-\operatorname{id}||_s \quad \mbox{and} \quad \max_{0 \leq t \leq 1} ||(\exp(t \tilde u_0)^{-1}-\operatorname{id}||_s < Q.
\]
As $||\varepsilon w_{x^\ast}||_s=\varepsilon$ for any $x^\ast \in \R^n$, it follows that $u_0 + \varepsilon w_{x^\ast} \in B_\varepsilon(u_0)$ for all $|\varepsilon| \leq \delta$ and hence, by choosing $\delta$ smaller if necessary, so that $\delta \leq \delta'$, one has
\begin{equation}\label{uniform_bound}
\sup_{\mbox{\scriptsize $\begin{array}{c} 0 \leq t \leq 1 \\ x^\ast \in \R^n \\ |\varepsilon| \leq \delta \end{array}$}} ||\varphi^{(\varepsilon)}(t)-\operatorname{id}||_s < Q,\qquad \sup_{\mbox{\scriptsize $\begin{array}{c} 0 \leq t \leq 1 \\ x^\ast \in \R^n \\ |\varepsilon| \leq \delta \end{array}$}} ||\big(\varphi^{(\varepsilon)}(t)\big)^{-1}-\operatorname{id}||_s < Q.
\end{equation}
By \eqref{uniform_bound} and the Sobolev imbedding \eqref{sobolev_imbedding} there exists a constant $M > 0$ with
\begin{equation}\label{uniform_M_bound}
 |\varphi^{(\varepsilon)}(t,x)-x|,\;|d\varphi^{(\varepsilon)}(t,x)| < M
\end{equation}
for any $(\varepsilon,t) \in [-\delta,\delta]\times [0,1]$, $x \in \R^n$ and $x^\ast \in \R^n$. For $N \geq 1$ choose $R_N > 0$ so large that for any $z \in \R^n$ with $|z| \geq R_N$ 
\[
 d(z,\operatorname{supp}u_0):=\inf_{y \in \operatorname{supp} u_0} |z - y| > (N+2) M.
\]
For any $y \in \operatorname{supp} u_0$ and any $x^\ast \in \R^n$ with $|x^\ast| \geq R_N$ we then get
\begin{multline}
 |\varphi^{(\varepsilon)}(t,x^\ast)-\varphi^{(\varepsilon)}(t,y)| = |(\varphi^{(\varepsilon)}(t,x^\ast)-x^\ast) + (x^\ast - y) + (y-\varphi^{(\varepsilon)}(t,y))| \\
 \label{large_denominator}
 \geq |x^\ast - y| - |\varphi^{(\varepsilon)}(t,x^\ast) - x^\ast| - |y - \varphi^{(\varepsilon)}(t,y)| \geq N M.
\end{multline}
It then follows from \eqref{exp_c1}-\eqref{dexp_c1} that by the Leibniz rule, $\left. \partial_\varepsilon \right|_{\varepsilon=0} I_1^{(\varepsilon)}$ can be computed by differentiating the integrand of $I_1^{(\varepsilon)}$ with respect to $\varepsilon$. Before doing this introduce for $0 \leq t \leq 1$ and $y \in \R^n$
\[
 J(t,y):=\left. \partial_\varepsilon \right|_{\varepsilon=0} \varphi^{(\varepsilon)}(t,y)=\big(d_{tu_0}\exp(t w_{x^\ast})\big)(y).
\]
In particular we then have by \eqref{dexp_c1}
\begin{equation}\label{dexp_derivative}
 \left. \partial_\varepsilon \right|_{\varepsilon = 0} d \varphi^{(\varepsilon)}(t,y) = dJ(t,y)
\end{equation}
and thus by the standard formula for the derivative of the inverse of a matrix
\begin{equation}\label{dexp_derivative_inv}
 \left. \partial_\varepsilon \right|_{\varepsilon=0} \big(d\varphi^{(\varepsilon)}(t,y)\big)^{-1} = -\big(d\varphi(t,y)\big)^{-1} dJ(t,y) \big(d\varphi(t,y)\big)^{-1}.
\end{equation}
Note that for any $0 \leq t \leq 1$,
\[
 ||J(t)||_s \leq ||d_{tu_0}\exp|| \, ||t w_{x^\ast}||_s \leq \max_{0 \leq t \leq 1} ||d_{tu_0}\exp||
\]
where $||d_{tu_0}\exp||$ is the operator norm of $d_{t u_0}\exp:H^s_\sigma(\R^n;\R^n) \to H^s(\R^n;\R^n)$ and hence by the Sobolev imbedding \eqref{sobolev_imbedding}, $J(t,x)$ and $dJ(t,x)$ are uniformly bounded with respect to $0 \leq t \leq 1$, $x \in \R^n$ and $x^\ast \in \R^n$. For convenience let $\Delta_{x,y} \varphi(t)=\varphi(t,x)-\varphi(t,y)$. Now taking the derivative of the integrand of $I_1^{(\varepsilon)}(x^\ast)$ with respect to $\varepsilon$ we get by the product rule 
\[
J_1(x^\ast):=\left. \partial_\varepsilon \right|_{\varepsilon=0} I_1^{(\varepsilon)}(x^\ast)=J_{1,1}(x^\ast) + J_{1,2}(x^\ast) + J_{1,3}(x^\ast)
\]
where $J_{1,1}(x^\ast), J_{1,2}(x^\ast)$ and $J_{1,3}(x^\ast)$ are given by, respectively,
\small
\begin{multline*}
\frac{1}{\omega_n} \int_0^1 \int_{\operatorname{supp}u_0} -[d\varphi(t,y)^\top]^{-1} dJ(t,y)^\top [d\varphi(t,y)^\top]^{-1} \Omega(u_0)(y) [d\varphi(t,y)]^{-1} \frac{\Delta_{x^\ast,y} \varphi(t)}{|\Delta_{x^\ast,y} \varphi(t)|^n}\,dy\,dt,
\end{multline*}
\begin{multline*}
\frac{1}{\omega_n} \int_0^1 \int_{\operatorname{supp}u_0} -[d\varphi(t,y)^\top]^{-1} \Omega(u_0)(y) [d\varphi(t,y)]^{-1} dJ(t,y) [d\varphi(t,y)]^{-1} \frac{\Delta_{x^\ast,y} \varphi(t)}{|\Delta_{x^\ast,y} \varphi(t)|^n}\,dy\,dt,
\end{multline*}
\begin{multline*}
\frac{1}{\omega_n} \int_0^1 \int_{\operatorname{supp}u_0} [d\varphi(t,y)^\top]^{-1} \Omega(u_0)(y) [d\varphi(t,y)]^{-1} \cdot \\
\left\{ \frac{\Delta_{x^\ast,y}J(t)}{|\Delta_{x^\ast,y} \varphi(t)|^n} - n \frac{\Delta_{x^\ast,y}\varphi(t) \cdot \Delta_{x^\ast,y} J(t) \Delta_{x^\ast,y}\varphi(t)}{|\Delta_{x^\ast,y} \varphi(t)|^{n+2}}\right\} \,dy\,dt
\end{multline*}
\normalsize
Note that the domain of integration $[0,1]\times \operatorname{supp}u_0$ of the integrals $J_{1,1}(x^\ast),J_{1,2}(x^\ast)$ and $J_{1,3}(x^\ast)$ is compact. The integrands are uniformly bounded, independet of the choice of $x^\ast$ with $|x^\ast| \geq R_N$. Moreover from \eqref{large_denominator} we see that the denominators go to infinity as $|x^\ast| \to \infty$. Thus for any $\rho > 0$ there exists $R'_{\rho} \geq R_N$ such that 
\begin{equation}\label{j1_small}
|J_1(x^\ast)| < \rho \quad \forall x^\ast \in \R^n \mbox{ with } |x^\ast| \geq R'_\rho. 
\end{equation}
Now consider $I_2^{(\varepsilon)}(x^\ast)$. As a consequence of \eqref{uniform_bound} and the Sobolev imbedding \eqref{sobolev_imbedding},
\[
 \sup_{\mbox{\scriptsize $\begin{array}{c} 0 \leq t \leq 1 \\ x^\ast \in \R^n \\ |\varepsilon| \leq \delta \end{array}$}} ||\varphi^{(\varepsilon)}(t)-\operatorname{id}||_{C^1} \quad \mbox{and} \quad \sup_{\mbox{\scriptsize $\begin{array}{c} 0 \leq t \leq 1 \\ x^\ast \in \R^n \\ |\varepsilon| \leq \delta \end{array}$}} ||\big(\varphi^{(\varepsilon)}(t)\big)^{-1}-\operatorname{id}||_{C^1}
\]
are finite. Hence there is $L > 0$, independent of $x^\ast$, with
\begin{equation}\label{lipschitz_up_down}
 \frac{1}{L} |x-y| \leq |\varphi^{(\varepsilon)}(t,x)-\varphi^{(\varepsilon)}(t,y)| \leq L |x-y| \qquad \forall x,y \in \R^n
\end{equation}
for any $(\varepsilon,t) \in [-\delta,\delta]\times [0,1]$. When combined with \eqref{uniform_M_bound} we get that the integrand of $I_2^{(\varepsilon)}(x^\ast)$ can be estimated uniformly for $|\varepsilon| \leq \delta$ and $0 \leq t \leq 1$ by
\begin{equation}\label{dominated}
 C \frac{L |x^\ast-y|}{(\frac{1}{L}|x^\ast-y|)^n} = C L^{n+1} \frac{1}{|x^\ast-y|^{n-1}} 
\end{equation}
for some constant $C > 0$. As the latter bound is independent of $\varepsilon$ one can take the limit $\varepsilon \to 0$ under the integral to get
\begin{eqnarray*}
 J_2(x^\ast)&:=&\left. \partial_\varepsilon \right|_{\varepsilon=0} (\varepsilon I_2^{(\varepsilon)}(x^\ast)) \\
&=& \frac{1}{\omega_n} \int_0^1 \int_{B_1(x^\ast)} [d\varphi(t,y)^\top]^{-1} \Omega(w_{x^\ast})(y) [d\varphi(t,y)]^{-1} \frac{\Delta_{x^\ast,y} \varphi(t)}{|\Delta_{x^\ast,y} \varphi(t)|^n} \,dy\,dt.
\end{eqnarray*}
Now the key idea is to show that the latter expression differs by a small error from
\[
 \frac{1}{\omega_n} \int_0^1 \int_{B_1(x^\ast)} \Omega(w_{x^\ast})(y) \frac{x^\ast-y}{|x^\ast-y|^n} \, dy\,dt
\]
which by the Biot-Savart law equals $w_{x^\ast}(x^\ast)=w(0)$ and hence does not vanish. To prove that the difference of $J_2(x^\ast)$ with the latter integral is indeed small, write $B_1(x^\ast)$ as the union of $B_1(x^\ast) \setminus B_\theta(x^\ast)$ and $B_\theta(x^\ast)$ with $0 < \theta < 1$ to be chosen at the end of the proof and write $J_2(x^\ast)$ as a sum of the corresponding integrals $J_2(x^\ast)=J_{2,1}^{(\theta)}(x^\ast)+J_{2,2}^{(\theta)}(x^\ast)$. First note that by the Sobolev imbedding \eqref{sobolev_imbedding} and the condition $s > n/2+1$, for any $y \in \R^n$ and for any $x^\ast \in \R^n$
\begin{equation}\label{bounded_vorticity}
 |\Omega(w_{x^\ast})(y)| \leq C ||w_{x^\ast}||_s = C.
\end{equation}
Using \eqref{dominated}-\eqref{bounded_vorticity} we get
\begin{equation}\label{small_error}
 |J_{2,2}^{(\theta)}(x^\ast)| \leq C' \int_0^1 \int_{B_\theta(x^\ast)} \frac{1}{|x^\ast-y|^{n-1}} \,dy\,dt \leq C'' \theta
\end{equation}
for some constants $C',C''$ independent of $x^\ast$. Note that the denominator of the integrand of $J_{2,1}^{(\theta)}(x^\ast)$ is bounded away from $0$. Indeed, by \eqref{lipschitz_up_down}
\[
 |\Delta_{x^\ast,y} \varphi(t)| = |\varphi(t,x^\ast)-\varphi(t,y)| \geq \frac{1}{L} |x^\ast-y| \geq \frac{\theta}{L}
\]
for all $y \in B_1(x^\ast) \setminus B_{\theta}(x^\ast)$. So by Lemma \ref{lemma_almost_identity} and \eqref{bounded_vorticity}, for any fixed $0 < \theta < 1$, and any $\rho > 0$ there exists a constant $R^{(\theta)}_\rho > 0$ such that for any $x^\ast \in \R^n$ with $|x^\ast| \geq R^{(\theta)}_\rho$
\begin{equation}\label{j21_diff_small}
 \left| J_{2,1}^{(\theta)}(x^\ast) - \frac{1}{\omega_n} \int_{B_1(x^\ast)\setminus B_{\theta}(x^\ast)} \Omega(w_{x^\ast})(y) \frac{x^\ast-y}{|x^\ast-y|^n} \,dy\right| < \rho.
\end{equation}
We now choose $x^\ast$ and $0 < \theta < 1$ according to our needs. Write 
\[
 w_{x^\ast}(x^\ast)=\frac{1}{\omega_n} \int_{B_1(x^\ast)} \Omega(w_{x^\ast})(y) \frac{x^\ast-y}{|x^\ast-y|^n} \,dy= w_1^{(\theta)}(x^\ast) + w_2^{(\theta)}(x^\ast)
\]
where
\[
 w_1^{(\theta)}(x^\ast) = \frac{1}{\omega_n} \int_{B_1(x^\ast) \setminus B_\theta(x^\ast)} \Omega(w_{x^\ast})(y) \frac{x^\ast-y}{|x^\ast-y|^n} \,dy
\]
and
\[
 w_2^{(\theta)}(x^\ast) = \frac{1}{\omega_n} \int_{B_\theta(x^\ast)} \Omega(w_{x^\ast})(y) \frac{x^\ast-y}{|x^\ast-y|^n} \,dy.
\]
First choose $0 < \theta < 1$ in such a way that we have for any choice of $x^\ast \in \R^n$
\begin{equation}\label{theta_choice}
 |J_{2,2}^{(\theta)}(x^\ast)| < a/8 \quad \mbox{and} \quad |w_2^{(\theta)}(x^\ast)| < a/8
\end{equation}
where $a=|w(0)|$. Due to \eqref{small_error} this is possible. Then for any $R > 0$ choose $x^\ast \in \R^n$ with $|x^\ast| \geq \max(R'_{a/8},R^{(\theta)}_{a/8},R)$ but otherwise arbitrary and let $v:=w_{x^\ast}$. Then by \eqref{j1_small},\eqref{j21_diff_small} and \eqref{theta_choice}
\begin{multline*}
 |\big(d_{u_0}\exp(v)\big)(x^\ast)-w_{x^\ast}(x^\ast)| = |J_1(x^\ast) + J_2(x^\ast) - w_{x^\ast}(x^\ast)| \\
 =|J_1(x^\ast) + J_{2,1}^{(\theta)}(x^\ast) + J_{2,2}^{(\theta)}(x^\ast) - w_1^{(\theta)}(x^\ast) - w_2^{(\theta)}(x^\ast)| \\
\leq |J_1(x^\ast)| + |J_{2,1}^{(\theta)}(x^\ast)-w_1^{(\theta)}(x^\ast)| + |J_{2,2}^{(\theta)}(x^\ast)| + |w_2^{(\theta)}(x^\ast)| \leq a/8 + a/8 + a/8 + a/8 = 
a/2.
\end{multline*}
Thus we see that $|\big(d_{u_0}\exp(v)\big)(x^\ast)| \geq a/2$ showing the claim with the choice $m=a/2$. 
\end{proof}

\noindent
Now we can prove Proposition \ref{prop_nonuniform}.

\begin{proof}[Proof of Proposition \ref{prop_nonuniform}]
It suffices to show that for any $u_0$ in the domain $U \subseteq H^s_\sigma(\R^n;\R^n)$ of $\phi$ there exists $R_\ast > 0$ with $B_{R_\ast}(u_0) \subseteq U$ so that $\phi$ is not uniformly continuous on $B_R(u_0)$ for any $0 < R \leq R_\ast$. As  $s > n/2+1$, $H^s(\R^n;\R^n) \hookrightarrow C^1_0(\R^n;\R^n)$. We denote by $C > 0$ the constant of this imbedding
\begin{equation}\label{sobolev_constant}
||f||_{C^1} \leq C ||f||_s.
\end{equation}
By the continuity of the exponential map (Proposition \ref{exp_analytic}), there exists $R_0 > 0$ so that $B_{R_0}(u_0) \subseteq U$ and for any $\varphi,\psi \in \exp\big(B_{R_0}(u_0)\big)$
\[
 ||\varphi - \psi||_s < \frac{1}{C}.
\]
Hence by \eqref{sobolev_constant} there is a constant $L > 0$ so that for any $\varphi,\psi \in \exp\big(B_{R_0}(u_0)\big)$
\begin{equation}\label{less_one}
 |\varphi(x)-\psi(x)| < 1 \quad \mbox{and} \quad |\varphi(x)-\varphi(y)| < L |x-y|, \quad \forall x,y \in \R^n.
\end{equation}
By the smoothness of the exponential map (Proposition \ref{exp_analytic}) and Taylor's theorem, for any $v,v+h$ in an arbitrary convex subset $V \subseteq U$,
\[
 \exp(v+h)=\exp(v)+d_v\exp(h) + \frac{1}{2} \int_0^1 (1-t) d^2_{v+th}\exp(h,h)\,dt.
\]
By choosing $0 < R_1 \leq R_0$, smaller if necessary, we can ensure that for some $C_1 > 0$
\begin{equation}\label{remainder}
||\exp(v+h)-\exp(v) - d_v \exp(h)||_s \leq C_1 ||h||_s^2, \qquad \forall v \in B_{R_1}(u_0), h \in B_{R_1}(0)
\end{equation}
As $v \mapsto d_v \exp$ is continuous we get for some $0 < R_2 \leq R_1$
\begin{equation}\label{near_differential}
||d_{v_1} \exp(h) - d_{v_2}\exp(h)||_s \leq \frac{m}{4C} ||h||_s, \qquad \forall v_1,v_2 \in B_{R_2}(u_0), h \in H^s_\sigma(\R^n;\R^n)
\end{equation}
where $m > 0$ is the constant in the statement of Lemma \ref{lemma_dexp} and $C > 0$ given by \eqref{sobolev_constant}. Finally by choosing $0 < R_3 \leq R_2$, sufficiently small, Lemma \ref{vorticity_estimate} implies that there exists $C_2 > 0$ so that
\begin{equation}\label{vort_estimate}
 \frac{1}{C_2} ||f||_{s-1} \leq ||R_\varphi^{-1}\left([d\varphi^\top]^{-1} f [d\varphi]^{-1}\right)||_{s-1} \leq C_2 ||f||_{s-1}
\end{equation}
for any $f \in H^{s-1}(\R^n;\R^{n\times n})$ and any $\varphi \in \exp\big(B_{R_3}(u_0)\big)$. Now set $R_\ast=R_3$ and take any $0 < R \leq R_\ast$. By the density of $C^\infty_{\sigma,c}(\R^n;\R^n)$ in $H^s_\sigma(\R^n;\R^n)$ (Lemma \ref{lemma_dense}), there exists $\bar u_0 \in C^\infty_{\sigma,c}(\R^n;\R^n) \cap B_{R/4}(u_0)$. Let $\varphi_\bullet:=\exp(\bar u_0)$ and introduce $K:=\operatorname{supp} \bar u_0$ and
\[
 K'=\{ y \in \R^n \, | \, \operatorname{dist}\big(y,\varphi_\bullet(K)\big) \leq 1 \}
\]
where $\operatorname{dist}\big(y,\varphi_\bullet(K)\big) = \inf_{x \in K} |y - \varphi_\bullet(x)|$ is the distance of $y$ to the set $\varphi_\bullet(K)$. By \eqref{less_one} we see that $K'$ has the property
\begin{equation}\label{in_k_strich}
 \varphi(K) \subseteq K', \qquad \forall \varphi \in \exp\big(B_R(u_0)\big)
\end{equation}
Note that $\lim_{|x| \to \infty} |\varphi_\bullet(x)|=\infty$. By Lemma \ref{lemma_dexp} we then can choose $x^\ast \in \R^n \setminus K'$ and $v \in C^\infty_{\sigma,c}(\R^n;\R^n)$ with $||v||_s=1$ in such a way that
\begin{equation}\label{v_property}
\operatorname{dist}\big(\varphi_\bullet(x^\ast),K'\big) > L + 1 \quad \mbox{and} \quad |\big(d_{\bar u_0} \exp(v)\big)(x^\ast)| \geq m.
\end{equation}
We set $M:=|\big(d_{\bar u_0}\exp(v)\big)(x^\ast)|$ and define $v_k=\frac{R}{4k}v$, $k \geq 1$. As $||v||_s=1$
\begin{equation}\label{v_k_in_ball}
 ||v_k||_s = \frac{R}{4k} < R/3.
\end{equation}
By the definition of $v_k$ we have $|\big(d_{\bar u_0}\exp(v_k)\big)(x^\ast)|=\delta_k:=M\frac{R}{4k}$. By \eqref{less_one} for any $k \geq 1$ there is 
\begin{equation}\label{rho_k_def}
0 < \rho_k < \min(\delta_k/4,1)=\min(\frac{MR}{16k},1)
\end{equation}
such that
\begin{equation}\label{disjoint_support}
 \varphi\big(B_{\rho_k}(x^\ast)\big) \subseteq B_{\delta_k/4}\big(\varphi(x^\ast)\big) \quad \forall \varphi \in \exp\big(B_R(u_0)\big).
\end{equation}
Now choose for each $k \geq 1$, a $w_k \in C^\infty_{\sigma,c}(\R^n;\R^n)$ with 
\begin{equation}\label{w_k_support}
 \operatorname{supp} w_k \subseteq B_{\rho_k}(x^\ast) \quad \mbox{and} \quad ||w_k||_s=R/4
\end{equation}
and define for $k \geq 1$ the pair of initial values
\[
 u_{0,k}=\bar u_0 + w_k \quad \mbox{and} \quad \tilde u_{0,k}=u_{0,k} + v_k.
\]
By our choices $(u_{0,k})_{k \geq 1},(\tilde u_{0,k})_{k \geq 1} \subseteq B_R(u_0)$ and $||u_{0,k} - \tilde u_{0,k}||_s = ||v_k||_s \to 0$ as $k \to \infty$. Denote the diffeomorphims corresponding to $u_{0,k},\tilde u_{0,k}$ by $\varphi_k,\tilde \varphi_k \in \Ds^s(\R^n)$, 
\[
 \varphi_k = \exp(u_{0,k}) \quad \mbox{and} \quad \tilde \varphi_k =\exp(\tilde u_{0,k})
\]
and the solutions of \eqref{RE} corresponding to the initial values $u_{0,k},\tilde u_{0,k}$ by $u_k,\tilde u_k:[0,1] \to H_\sigma^s(\R^n;\R^n)$. The corresponding vorticities at time $t=0$, $\Omega_{0,k}$ and $\tilde \Omega_{0,k}$, and $t=1$, $\Omega_{1,k}$ and $\tilde \Omega_{1,k}$, are then given by 
\begin{equation}\label{vorticities}
\begin{array}{ccccc}
 \Omega_{0,k} & = &\Omega(u_{0,k})&=&\Omega(\bar u_0)+\Omega(w_k) \\
 \tilde \Omega_{0,k} &= &\Omega_{0,k} + \Omega(v_k)&=&\Omega(\bar u_0)+ \Omega(w_k + v_k)
\end{array}
\end{equation}
and
\[
 \Omega_{1,k}=\Omega\big(u_k(1)\big);\quad \tilde \Omega_{1,k} = \Omega\big(\tilde u_k(1)\big).
\]
Note that we have for some $C' > 0$
\begin{equation}\label{u_omega_ineq}
||\phi(u_{0,k})-\phi(\tilde u_{0,k})||_s = ||u_k(1) - \tilde u_k(1)||_s \geq \frac{1}{C'} ||\Omega_{1,k}-\tilde \Omega_{1,k}||_{s-1}.
\end{equation}
We aim at estimating $||\Omega_{1,k}-\tilde \Omega_{1,k}||_{s-1}$ from below. By the conservation law \eqref{conserved} we have
\begin{equation}\label{vorticity_reexpressed}
 \Omega_{1,k}=R_{\varphi_k}^{-1} \left([d\varphi_k^\top]^{-1} \Omega_{0,k} [d\varphi_k]^{-1}\right) \quad \mbox{and} \quad
 \tilde \Omega_{1,k} = R_{\tilde \varphi_k}^{-1} \left([d\tilde \varphi_k^\top]^{-1} \tilde \Omega_{0,k} [d\tilde \varphi_k]^{-1}\right).
\end{equation}
By \eqref{v_property} the distance of $\varphi_\bullet(x^\ast)$ to $K'$ is bigger than $L+1$ and hence by \eqref{less_one} 
\[
\operatorname{dist}\big(\varphi(x^\ast),K') > L \quad  \mbox{for any } \varphi \in \exp\big(B_R(u_0)\big).
\]
On the other hand by \eqref{less_one} and $\rho_k < 1$ one has 
\[
 |\varphi(x^\ast)-\varphi(x)| \leq L |x^\ast - x| \leq L \quad \forall x \in \operatorname{supp} w_k.
\]
Combining the two latter displayed inequalities one concludes that
\begin{equation}\label{support_contained1}
 \varphi\big(\operatorname{supp}(w_k)\big) \cap K' =\emptyset, \quad \forall \varphi \in \exp\big(B_R(u_0)\big).
\end{equation}
As $\operatorname{supp}(w_k + v_k) \subseteq B_1(x^\ast)$ the same argument gives
\begin{equation}\label{support_contained2}
 \varphi\big(\operatorname{supp}(w_k + v_k)\big) \cap K' = \emptyset, \quad \forall \varphi \in \exp\big(B_R(u_0)\big).
\end{equation}
By \eqref{in_k_strich},
\[
 \operatorname{supp} R_{\varphi_k}^{-1}\left((d\varphi_k^\top)^{-1} \Omega(\bar u_0) (d\varphi_k)^\top \right) \subseteq K'
\]
and
\[
 \operatorname{supp} R_{\tilde \varphi_k}^{-1}\big((d \tilde \varphi_k^\top)^{-1} \Omega(\bar u_0) (d\tilde \varphi_k)^\top \big) \subseteq K'.
\]
From \eqref{support_contained1}-\eqref{support_contained2},
\[
  \varphi_k\big(\operatorname{supp} \Omega(w_k)\big) \subseteq \R^n\setminus K'
\quad \mbox{and} \quad 
 \tilde \varphi_k \big(\operatorname{supp} \Omega(w_k+v_k)\big) \subseteq \R^n \setminus K'.
\]
By \eqref{vorticities}-\eqref{vorticity_reexpressed} it then follows that
\begin{multline}\label{from_below}
||\Omega_{1,k}-\tilde \Omega_{1,k}||_{s-1} = ||R_{\varphi_k}^{-1}\left((d\varphi_k^\top)^{-1} \Omega(\bar u_0) (d\varphi_k)^{-1}\right)- R_{\tilde \varphi_k}^{-1}\left((d\tilde \varphi_k^\top)^{-1} \Omega(\bar u_0) (d\tilde \varphi_k)^{-1}\right)||_{s-1} \\
+ ||R_{\varphi_k}^{-1} \left((d\varphi_k^\top)^{-1} \Omega(w_k) (d\varphi_k)^{-1}\right) - R_{\tilde \varphi_k}^{-1} \left((d\tilde \varphi_k^\top)^{-1} \Omega(w_k+v_k) (d\tilde \varphi_k)^{-1}\right)||_{s-1} \\
\geq ||R_{\varphi_k}^{-1} \left((d\varphi_k^\top)^{-1} \Omega(w_k) (d\varphi_k)^{-1}\right) - R_{\tilde \varphi_k}^{-1} \left((d\tilde \varphi_k^\top)^{-1} \Omega(w_k+v_k) (d\tilde \varphi_k)^{-1}\right)||_{s-1}
\end{multline}
We claim that, for large $k$,
\begin{equation}\label{show_disjoint_support}
\varphi_k\big(\operatorname{supp}(w_k)\big) \cap \tilde \varphi_k\big(\operatorname{supp}(w_k)\big) = \emptyset.
\end{equation}
Indeed by the Taylor formula 
\[
 \tilde \varphi_k - \varphi_k = \exp(\bar u_0 + w_k + v_k) - \exp(\bar u_0 + w_k) =
 d_{\bar u_0 + w_k}\exp(v_k) + \mathcal R_k
\]
where $\mathcal R_k$ is the remainder term. Thus we can write
\begin{equation}\label{diffeo_difference}
\tilde \varphi_k - \varphi_k = d_{\bar u_0} \exp(v_k) + \left( d_{\bar u_0+ w_k}\exp(v_k) - d_{\bar u_0} \exp(v_k)\right) + \mathcal R_k.
\end{equation}
We want to estimate $\tilde \varphi(x^\ast)-\varphi(x^\ast)$ by estimating the three terms on the right-hand side of the latter identity individually. By the Sobolev imbedding \eqref{sobolev_constant} and \eqref{remainder} we get the following estimate for $\mathcal R_k(x^\ast) \in \R^n$
\[
 |\mathcal R_k(x^\ast)| \leq C ||\mathcal R_k||_s \leq C C_1 ||v_k||_s^2 = C C_1 \frac{R^2}{16k^2}.
\]
For $k$ sufficiently large it then follows that
\[
 |\mathcal R_k(x^\ast)| < \frac{\delta_k}{4}.
\]
Furthermore, using \eqref{sobolev_constant} and \eqref{near_differential}, together with $m \leq M$ (cf \eqref{v_property})
\begin{multline*}
 \big| \big(d_{\bar u_0 + w_k}\exp(v_k)\big)(x^\ast) - \big(d_{\bar u_0}\exp(v_k)\big)(x^\ast) \big| \\
\leq C ||d_{\bar u_0 + w_k}\exp(v_k) - d_{\bar u_0}\exp(v_k)||_s  \leq \frac{m}{4} ||v_k||_s \leq \frac{M R}{16k} = \frac{\delta_k}{4}.
\end{multline*}
Finally, for the first term on the right-hand side of \eqref{diffeo_difference} one has by definition,
\[
 \big|d_{\bar u_0} \exp(v_k)(x^\ast)\big| = \delta_k.
\]
Combining the estimates above, \eqref{diffeo_difference} yields for $k$ large enough
\[
 |\tilde \varphi_k(x^\ast) - \varphi_k(x^\ast)| > \frac{\delta_k}{2}.
\]
By \eqref{disjoint_support} we get for large $k$
\[
 \varphi_k\big(B_{\rho_k}(x^\ast)\big) \cap \tilde \varphi_k\big(B_{\rho_k}(x^\ast)\big) = \emptyset
\]
showing \eqref{show_disjoint_support}. It leads by the triangle inequality to the estimate
\begin{multline}\label{limsup1}
 ||R_{\varphi_k}^{-1} \left([d\varphi_k^\top]^{-1} \Omega(w_k) [d\varphi_k]^{-1} \right) - R_{\tilde \varphi_k}^{-1} \left([d\tilde \varphi_k]^{-1} \Omega(w_k + v_k) [d\tilde \varphi_k]^{-1} \right)||_{s-1} \\
 \geq ||R_{\varphi_k}^{-1} \left([d\varphi_k^\top]^{-1} \Omega(w_k) [d\varphi_k]^{-1}\right)||_{s-1} + ||R_{\tilde \varphi_k}^{-1} \left([d\tilde\varphi_k^\top]^{-1} \Omega(w_k) [d\tilde\varphi_k]^{-1}\right)||_{s-1}\\
 - ||R_{\tilde \varphi_k}^{-1} \left([d\tilde\varphi_k^\top]^{-1} \Omega(v_k) [d\tilde\varphi_k]^{-1}\right)||_{s-1}.
\end{multline}
The latter term we can be estimated using \eqref{vort_estimate} by
\begin{equation}\label{limsup2}
 ||R_{\tilde \varphi_k}^{-1} \left([d\tilde \varphi_k^\top]^{-1} \Omega(v_k) [d\tilde \varphi_k]^{-1}\right)||_{s-1} \leq C_2 ||\Omega(v_k)||_{s-1} \leq C_2 C' ||v_k||_s
\end{equation}
which by \eqref{v_k_in_ball} goes to $0$ for $k \to \infty$. For the first two terms on the right-hand side of the inequality \eqref{limsup1} we have again by \eqref{vort_estimate}
\begin{equation}\label{limsup3}
||R_{\varphi_k}^{-1}\left([d\varphi_k^\top]^{-1} \Omega(w_k) [d\varphi_k]^{-1}\right)||_{s-1} \geq \frac{1}{C_2} ||\Omega(w_k)||_{s-1}
\end{equation}
and
\begin{equation}\label{limsup4}
 ||R_{\tilde \varphi_k}^{-1}\left([d\tilde \varphi_k^\top]^{-1} \Omega(w_k) [d\tilde \varphi_k]^{-1}\right)||_{s-1} \geq \frac{1}{C_2} ||\Omega(w_k)||_{s-1}.
\end{equation}
Combining \eqref{limsup1}-\eqref{limsup4}, the inequality \eqref{from_below} then leads to
\[
 \limsup_{k \geq 1} ||\Omega_{1,k}-\tilde \Omega_{1,k}||_{s-1} \geq \limsup_{k \geq 1} \frac{2}{C_2} ||\Omega(w_k)||_{s-1}.
\]
We will get the result by showing that $\limsup_{k \geq 1} ||\Omega(w_k)||_{s-1}$ is bounded away from $0$. In $H^s(\R^n;\R^n)$ the following norm
\[
 |||f|||_s := ||f||_{L^2} + ||df||_{s-1}
\]
is equivalent to the norm $||\cdot||_s$. In particular there exists $C_3 > 0$ so that for any $f \in H^s(\R^n;\R^n)$
\begin{equation}\label{norms_equiv}
\frac{1}{C_3} ||f||_s \leq |||f|||_s \leq C_3 ||f||_s.
\end{equation}
By \eqref{norms_equiv} we thus get $|||w_k|||_s \geq \frac{1}{C_3} \frac{R}{4}$ for all $k \geq 1$. By \eqref{rho_k_def} and \eqref{w_k_support}
\begin{multline}\label{l2_zero}
 ||w_k||_{L^2} \leq ||w_k||_{L^\infty} \operatorname{vol}\big(B_{\rho_k}(x^\ast)\big) \leq C ||w_k||_s \operatorname{vol}\big(B_{\rho_k}(x^\ast)\big) \\
  \leq C \frac{R}{4} \operatorname{vol}\big(B_1(0)\big) \left(\frac{MR}{16k}\right)^n.
\end{multline}
Hence $||w_k||_{L^2}$ goes to $0$ for $k \to \infty$ implying that
\[
 \limsup_{k \geq 1} ||d w_k||_{s-1} \geq \frac{1}{C_3} \frac{R}{4}.
\]
By Lemma \ref{lemma_gradient_velocity} 
\[
 \limsup_{k \geq 1} ||\Omega(w_k)||_{s-1} \geq \limsup_{k \geq 1} \frac{1}{C_4} ||d w_k||_{s-1} \geq \frac{1}{C_3 C_4} \frac{R}{4}
\]
for some constant $C_4 >0$. By \eqref{u_omega_ineq} we then conclude
\begin{equation}\label{indep_of_R}
 \limsup_{k \geq 1} ||\phi(u_{0,k})-\phi(\tilde u_{0,k})||_s \geq \limsup_{k \geq 1} \frac{1}{C'} ||\Omega_{1,k}-\tilde \Omega_{1,k}||_{s-1} \geq \frac{1}{4C_3 C_4} R
\end{equation} 
whereas $||u_{0,k} - \tilde u_{0,k}||_s \to 0$. As $(u_{0,k}),(\tilde u_{0,k})$ are in $B_R(u_0)$ this shows that $\phi$ is not uniformly continuous on $B_R(u_0)$.
\end{proof}

\begin{proof}[Proof of Theorem \ref{th_not_uniform}]
Recall that we want to prove that for $T>0$ the time $T$ solution map of \eqref{E}
\[
 E_T:U_T \to H^s_\sigma(\R^n;\R^n),\quad u_0 \to u(T;u_0)
\]
is nowhere locally uniformly continuous. By the scaling property described in \eqref{scaling} we have for any $u_0 \in H^s_\sigma(\R^n,\R^n)$
\begin{equation}\label{E_T}
E_T(u_0) = T \phi(T u_0).
\end{equation}
Thus by Proposition \ref{prop_nonuniform} we get that $E_T$ is also nowhere locally uniformly continuous.
\end{proof}

\noindent
Finally we can give the proof of Theorem \ref{th_not_differentiable}

\begin{proof}[Proof of Theorem \ref{th_not_differentiable}]
By \eqref{E_T} it suffices to consider the case $T=1$, i.e. to prove that for any given $s > n/2+1$ and $U\equiv U_1$ as in \eqref{time_one_map}
\[
 \phi:U \to H^s_\sigma(\R^n;\R^n),\quad u_0 \mapsto u(1;u_0)
\]
is nowhere differentiable. The key ingredient is inequality \eqref{indep_of_R}. Let us reformulate it in a convenient way. Let $w \in U$. Then by the last part of the proof of Proposition \ref{prop_nonuniform} there are $R_\ast,C_\ast > 0$ with $B_{R_\ast}(w) \subseteq U$ satisfying the following property: for any $0 < R \leq R_\ast$ there are sequences $(u_{0,k})_{k \geq 1},(\tilde u_{0,k})_{k \geq 1} \subseteq B_{R}(w)$ with
\begin{equation}\label{zero_seq}
\lim_{k \to \infty} ||u_{0,k}-\tilde u_{0,k}||_s =0
\end{equation}
and 
\begin{equation}\label{below_by_R}
||\phi(u_{0,k})-\phi(\tilde u_{0,k})||_s \geq C_\ast R,\qquad \forall k \geq 1.
\end{equation}
Assume now that $\phi$ is differentiable in $w$. For any $h \in H^s_\sigma(\R^n;\R^n)$ with $w+h \in B_{R_\ast}(w)$
\begin{equation}\label{diff_remainder}
\mathcal R(w,h):= \phi(w+h)-\phi(w)+d_w \phi(h).
\end{equation}
By the definition of differentiability there is $0 < R \leq R_\ast$ with
\begin{equation}\label{remainder_estimate}
||\mathcal R(w,h)||_s \leq \frac{C_\ast}{4} ||h||_s
\end{equation}
for any $h \in H^s_\sigma(\R^n;\R^n)$ with $||h||_s \leq R$. Take sequences $(u_{0,k})_{k \geq 1},(\tilde u_{0,k})_{k \geq 1} \subseteq B_R(w)$ satisfying \eqref{zero_seq}-\eqref{below_by_R}. We then get by \eqref{diff_remainder}
\[
 \phi(u_{0,k})=\phi(w+(u_{0,k}-w)) = \phi(w) + d_w \phi(u_{0,k}-w) + \mathcal R(w,u_{0,k}-w)
\]
and a similar expression for $\phi(\tilde u_{0,k})$. Hence
\[
 \phi(u_{0,k})-\phi(\tilde u_{0,k}) = d_w \phi (u_{0,k}-\tilde u_{0,k}) +\mathcal R(w,u_{0,k}-w) - \mathcal R(w,\tilde u_{0,k}-w).
\]
and thus by \eqref{zero_seq}, $||d_w \phi (u_{0,k}-\tilde u_{0,k})||_s \underset{k \to \infty}{\to} 0$ yielding
\begin{multline*}
 \limsup_{k \geq 1} ||\phi(u_{0,k})-\phi(\tilde u_{0,k})||_s \\
\leq \limsup_{k \geq 1} ||\mathcal R(w,u_{0,k}-w)||_s + \limsup_{k \geq 1} ||\mathcal R(w,\tilde u_{0,k}-w)||_s \leq \frac{C_\ast}{2} \bar R
\end{multline*}
where the last inequality follows from \eqref{remainder_estimate}. This is a contradiction to \eqref{below_by_R}. Hence $\phi$ is not differentiable in $w$. As $w$ was arbitrary the claim follows.
\end{proof}

\section{The submanifold $\Ds^s_\mu(\R^n)$}\label{section_submanifold}

Throughout this section we assume as usual $s > n/2+1$ with $n \geq 2$. We will prove Theorem \ref{th_submanifold} saying that $\Ds^s_\mu(\R^n)$ is a closed analytic submanifold of $\Ds^s(\R^n)$. The most natural way to prove this statement is to consider the analytic map
\begin{equation}\label{map_F}
 \varphi \mapsto \left[F(\varphi):\R^n \to \R,\quad x \mapsto \det(d_x \varphi) - 1\right]
\end{equation}
for $\varphi$ in $\Ds^s(\R^n)$ -- see the proof of the corresponding result for $\Ds^s(M)$, $M$ a compact manifold, of Ebin and Marsden \cite{ebin}. We clearly have $\Ds^s_\mu(\R^n)=F^{-1}(0)$. Using the Banach algebra property of $H^{s-1}(\R^n)$ one shows that $F$ takes values in $H^{s-1}(\R^n)$ and is analytic. In particular $\Ds_\mu^s(\R^n)$ is a closed subset of $\Ds^s(\R^n)$ and it remains to show that $0 \in H^{s-1}(\R^n)$ is a regular value of $F$. The differential of $F$ at $\operatorname{id} \in \Ds^s_\mu(\R^n)$ is given by 
\[
 d_{\operatorname{id}}F:H^s(\R^n;\R^n) \to H^{s-1}(\R^n),\quad f \mapsto \operatorname{div} f
\]
which is however not surjective.

\begin{Lemma}\label{lemma_not_surjective}
The map
\[
 \operatorname{div}:H^s(\R^n;\R^n) \to H^{s-1}(\R^n),\quad f \mapsto \operatorname{div} f
\]
is not surjective.
\end{Lemma}

\begin{proof}
Assume that $\operatorname{div}$ is surjective. As $H^s_\sigma(\R^n;\R^n)$ is by definition the null space of $\operatorname{div}$, the map
\[
 \Psi:H^s_\sigma(\R^n;\R^n)^\perp \to H^{s-1}(\R^n),\quad f \mapsto \operatorname{div} f
\]
is then a bijection. Here $H^s_\sigma(\R^n;\R^n)^\perp$ is the orthogonal complement of $H^s_\sigma(\R^n;\R^n)$ in $H^s(\R^n;\R^n)$ with respect to the inner product $\langle \cdot,\cdot \rangle_s$. By the open mapping theorem $\Psi$ has a continuous inverse denoted by $\Phi$
\[
 \Phi:H^{s-1}(\R^n) \to H^s_\sigma(\R^n;\R^n)^\perp.
\]
In particular it means that there is a constant $C > 0$ so that
\begin{equation}\label{continuous_inverse}
||\Phi(w)||_s \leq C ||w||_{s-1}, \qquad \forall w \in H^{s-1}(\R^n).
\end{equation}
We then get for any $w \in H^{s-1}(\R^n)$ by integration by parts
\[
 \langle  w, w \rangle_{s-2} = \langle \Psi \Phi(w), w \rangle_{s-2} = -\sum_{j=1}^n \langle \Phi_j(w),\partial_j w \rangle_{s-2}.
\]
Applying Cauchy-Schwarz we get for any $w \in H^{s-1}(\R^n)$
\begin{eqnarray}
\nonumber
 ||w||_{s-2}^2 &\leq& ||\Phi(w)||_{s-2} ||\nabla w||_{s-2} \\
\label{wrong_inequality}
&\leq& C ||w||_{s-1} ||\nabla w||_{s-2}
\end{eqnarray}
where we used \eqref{continuous_inverse}. We claim that the inequality \eqref{wrong_inequality} cannot hold. To see it take an element $w \in H^{s-1}(\R^n)$ with $||w||_{L^2}=1$ whose Fourier transform $\hat w$ is supported in the unit ball, $\operatorname{supp} \hat w \subseteq B_1(0)$. Define for $k \geq 1$ and $x \in \R^n$
\[
w_k(x) = w\left(\frac{x}{k}\right).
\]
Using that $\widehat{w_k}(\xi)=k^n \hat w(k \xi)$, one has
\begin{eqnarray*}
 ||w_k||_{s-2}^2 &=& \int_{\R^n} (1+|\xi|^2)^{s-2} k^{2n} |\hat w(k\xi)|^2 \,d\xi \\
\noalign{\noindent and by the change of variable $\eta:=k\xi$}
 &=& \int_{|\eta| \leq 1} (1+|\frac{\eta}{k}|^2)^{s-2} k^n |\hat w(\eta)|^2 \,d\eta \geq k^n ||w||^2_{L^2}=k^n.
\end{eqnarray*}
Analogously we have
\[
 ||w_k||_{s-1}^2 = \int_{|\eta| \leq 1} (1+|\frac{\eta}{k}|^2)^{s-1} k^n |\hat w(\eta)|^2 \,d\eta \leq 2^{s-1} k^n ||w||_{L^2}^2=2^{s-1}k^n.
\]
Similarly we have for $1 \leq j \leq n$
\begin{eqnarray*}
 ||\partial_j w_k||_{s-2}^2 &=& \int_{\R^n} (1+|\xi|^2)^{s-2} \xi_j^2 k^{2n} |\hat w(k \xi)|^2 \,d\xi \\
&=& \int_{|\eta|\leq 1} (1+|\frac{\eta}{k}|^2)^{s-2} \frac{\eta_j^2}{k^2} k^n |\hat w(\eta)|^2 \,d\eta\\
&\leq& 2^{s-2} k^{n-2} ||w||_{L^2}^2=2^{s-2}k^{n-2}.
\end{eqnarray*}
So 
\[
 ||w_k||_{s-1} ||\nabla w||_{s-2} \leq (2^{s-1}  k^n)^{1/2} \cdot (2^{s-2}k^{n-2})^{1/2} =2^{s-3/2} k^{n-1} \mbox{ and } ||w||_{s-2}^2 \geq k^n.
\]
Thus for $k$ large the inequality \eqref{wrong_inequality} cannot hold. This shows that the assumption that $\operatorname{div}$ is surjetive is wrong. 
\end{proof}

Lemma \ref{lemma_not_surjective} shows that $0$ is not a regular value of $F$. To prove Theorem \ref{th_submanifold} we therefore have to argue differently then Ebin and Marsden in \cite{ebin}. The key idea is to use the exponential map as a parametrization of $\Ds^s_\mu(\R^n)$. We want to show that near $\operatorname{id} \in \Ds^s_\mu(\R^n)$ there exists a neighborhood $V$ of $0$ in $H^s_\sigma(\R^n;\R^n)$ which $\exp$ maps bijectively onto a neighborhood of $\operatorname{id}$ in $\Ds^s_\mu(\R^n)$. Recall from \eqref{uexp_def} that $U^s_{\exp} \subseteq H^s(\R^n;\R^n)$ denotes the domain of the exponential map. 

\begin{Prop}\label{prop_local_manifold}
There is a neighborhood $\tilde U \subseteq U^s_{\exp}$ of $0$ such that
\[
 \exp\big(\tilde U \cap H^s_\sigma(\R^n;\R^n)\big) = \exp(\tilde U) \cap \Ds^s_\mu(\R^n).
\]
Moreover $\left. \exp \right|_{\tilde U}$ is an analytic diffeomorphism onto its image.
\end{Prop}

First we have to make some preparations for the proof of Proposition \ref{prop_local_manifold}. In the following we denote as usual by $u(t;u_0)$ the solution of \eqref{RE} at time $t$ with initial value $u_0$. By the proof of Theorem \ref{th_lagrangian}(ii), $\; u(t;u_0)$ is well-defined on $[0,1] \times U^s_{\exp}$. We state without proof the following quite well-known lemma (see e.g. \cite{majda})

\begin{Lemma}\label{det_derivative}
Let $u:[0,1] \times \R^n \to \R^n$ be a $C^1$-vector field admitting a flow $\varphi:[0,1] \times \R^n \to \R^n$, i.e. a $C^1$-map satisfying
\[
 \partial_t \varphi(t,x)=u(t,\varphi(t,x)) \quad \mbox{and} \quad \varphi(0,x)=x
\]
for any $(t,x) \in [0,1] \times \R^n$. Then for all $(t,x) \in [0,1] \times \R^n$
\[
 \partial_t \det\big(d_x \varphi(t,x)\big) = (\operatorname{div} u)\big(t,\varphi(t,x)\big) \cdot \det\big(d_x\varphi(t,x)\big)
\]
or, in integrated form,
\begin{equation}\label{det_formula}
\det\big(d_x\varphi(t,x)\big) = e^{\int_0^t (\operatorname{div} u)(\tau,\varphi(\tau,x)) \,d\tau}.
\end{equation}
\end{Lemma}

As a consequence of Lemma \ref{det_derivative}, Lemma \ref{lemma_exp} and Theorem \ref{th_lagrangian}(ii) we get

\begin{Coro}\label{exp_restricted}
The exponential map $\exp$ maps the divergence free vector fields into the volume-preserving diffeomorphisms, i.e. 
\[
 \exp\big(U^s_{\exp} \cap H^s_\sigma(\R^n;\R^n)\big) \subseteq \Ds^s_\mu(\R^n).
\]
\end{Coro}

\begin{Lemma}\label{lemma_small_u}
For any $\varepsilon > 0$ there is a neighborhood $\tilde U \subseteq U^s_{\exp}$ of $0$ such that
\[
 ||u(t;u_0)||_s < \varepsilon
\]
for all $t \in [0,1]$ and for all $u_0 \in \tilde U$.
\end{Lemma}

\begin{proof}[Proof of Lemma \ref{lemma_small_u}]
Since by the proof of Theorem \ref{th_lagrangian}(ii), given at the end of Section \ref{section_exponential_map},
\[
 [0,1] \times U^s_{\exp} \to H^s(\R^n;\R^n), \quad (t,u_0) \mapsto u(t;u_0)
\]
is continuous and $u(t;0)=0$ for all $t \in [0,1]$, the claim follows by the compactness of $[0,1]$.
\end{proof}

\begin{Lemma}\label{lemma_linear_u}
There is a neighborhood $\tilde U \subseteq U^s_{\exp}$ of $0$ and a constant $C>0$ such that we have for any $0 \leq t \leq 1$ and any $u_0 \in \tilde U$
\[
 ||\operatorname{div} u(t;u_0)||_{s-1} \leq C ||\operatorname{div} u_0||_{s-1}.
\]
\end{Lemma}

\begin{proof}[Proof of Lemma \ref{lemma_linear_u}]
Choose $\tilde U$ to be a small ball around $0$ contained in $U^s_{\exp}$ so that on the one hand by Lemma \ref{lemma_small_u}
\begin{equation}\label{assumption1}
 ||u(t;u_0)||_s \leq 1, \quad \forall 0 \leq t \leq 1,\; \forall u_0 \in \tilde U
\end{equation}
and on the other hand, by Lemma \ref{lemma_linear_growth}, for some $C_1 >0$, for any $\varphi \in \exp(\tilde U)$
\begin{equation}\label{assumption2}
 ||R_\varphi f||_{s-1} \leq C_1 ||f||_{s-1} \mbox{ and } ||R_\varphi^{-1} f||_{s-1} \leq C_1 ||f||_{s-1},\; \forall f \in H^{s-1}(\R^n).
\end{equation}
Denote by $\varphi(\cdot;u_0)$ the flow corresponding to $u(\cdot;u_0)$. By the chain rule for $s$ sufficiently large one has
\begin{equation}\label{t_derivative_div}
 \partial_t \big((\operatorname{div} u)\circ \varphi\big) = R_\varphi\left( \partial_t \operatorname{div}u + (u \cdot \nabla)\operatorname{div}u\right).
\end{equation}
Approximate $u(\cdot;u_0)$ by $(u_k)_{k \geq 1} \subseteq C^1\big([0,1];H^{s+1}(\R^n;\R^n))$ in the norm of the space
\[
C^0\big([0,1];H^s(\R^n;\R^n)\big) \cap C^1\big([0,1];H^{s-1}(\R^n;\R^n)\big).
\]
Then for any $k \geq 1$, one has in $C^0\big([0,1];H^{s-1}(\R^n;\R^n)\big)$
\[
 \partial_t \big((\operatorname{div} u_k)\circ \varphi\big) = R_\varphi\left( \partial_t \operatorname{div}u_k + (u \cdot \nabla)\operatorname{div}u_k\right)
\]
In particular the identity holds in $C^0\big([0,1];H^{s-2}(\R^n;\R^n)\big)$. Letting $k \to \infty$ on both sides of the latter identity leads to
\[
 \partial_t \big((\operatorname{div} u)\circ \varphi\big) = R_\varphi\left( \partial_t \operatorname{div}u + (u \cdot \nabla)\operatorname{div}u\right).
\]
Substituting formula \eqref{u_derivative} for $\partial_t \operatorname{div} u$ one gets
\begin{equation}\label{t_derivative_explicit}
 \partial_t \big((\operatorname{div} u) \circ \varphi\big) = R_\varphi \left( \chi(D) \big( 2(u \cdot \nabla) \operatorname{div} u + (\operatorname{div} u)^2 \big) \right).
\end{equation}
Integrating \eqref{t_derivative_explicit} with respect to $t$ yields
\[
 \operatorname{div} u(t) = R_{\varphi(t)}^{-1} \left(\operatorname{div} u_0 + \int_0^t R_{\varphi(\tau)} \left(\chi(D)\Big(2(u(\tau) \cdot \nabla) \operatorname{div} u(\tau) + (\operatorname{div} u(\tau))^2\Big)\right) \,d\tau \right).
\]
Using \eqref{assumption2} we get
\begin{multline}
\label{divu_triangle}
 ||\operatorname{div} u(t)||_{s-1} \leq C_1 ||\operatorname{div} u_0||_{s-1} \\
 + C_1^2 \int_0^t ||\chi(D)\big(2(u(\tau) \cdot \nabla) \operatorname{div} u(\tau)\big)||_{s-1}  + ||\chi(D)(\operatorname{div} u(\tau))^2||_{s-1} \,d\tau
\end{multline}
For the first expression under the integral sign we have by Lemma \ref{properties_chi} for all $0 \leq \tau \leq 1$
\[
 ||\chi(D)\big(2(u(\tau) \cdot \nabla) \operatorname{div} u(\tau)\big)||_{s-1} \leq 2\sqrt{2} ||\big(u(\tau) \cdot \nabla\big) \operatorname{div} u(\tau)||_{s-2}.
\]
Lemma \ref{lemma_multiplication} implies that there exists a constant $C_2 > 0$ such that for all $0 \leq \tau \leq 1$
\[
 \||\big(u(\tau) \cdot \nabla\big) \operatorname{div} u(\tau)||_{s-2} \leq C_2 ||u(\tau)||_s ||\operatorname{div} u(\tau)||_{s-1}.
\]
Combined with \eqref{assumption1} we thus have proved that
\[
  ||\chi(D)\big(2(u(\tau) \cdot \nabla) \operatorname{div} u(\tau)\big)||_{s-1} \leq 2 \sqrt{2} C_2 ||\operatorname{div} u(\tau)||_{s-1}.
\]
For the second expression in the integrand in \eqref{divu_triangle} the Banach algebra property of $H^{s-1}(\R^n)$ says that there exists an absolute constant $C_3 > 0$ so that
\[
 ||\chi(D) \big(\operatorname{div} u(\tau)\big)^2||_{s-1} \leq ||\big(\operatorname{div}u(\tau)\big)^2||_{s-1}\leq C_3 ||\operatorname{div} u(\tau)||_{s-1}^2,\; \forall 0 \leq \tau \leq 1,
\]
or using that $||\operatorname{div} u(\tau)||_{s-1} \leq ||u(\tau)||_s \leq 1$ one concludes that
\[
 ||\chi(D) \big(\operatorname{div} u(\tau)\big)^2||_{s-1} \leq C_3 ||\operatorname{div} u(\tau)||_{s-1},\; \forall 0 \leq \tau \leq 1.
\]
Substituting the obtained inequalities into \eqref{divu_triangle} there is an absolute constant $C_4 > 0$ such that 
\[
 ||\operatorname{div} u(t) ||_{s-1} \leq C_1 ||\operatorname{div} u_0||_{s-1} + C_4 \int_0^t ||\operatorname{div} u(\tau)||_{s-1} \,d\tau,\; \forall 0 \leq t \leq 1.
\]
By Gronwall's inequality we then have for any $0 \leq t \leq 1$
\[
 ||\operatorname{div} u(t)||_{s-1} \leq C_1 ||\operatorname{div} u_0||_{s-1} (1+e^{C_4 t}).
\] 
By choosing $\tilde U$ as described above and $C=C_1 (1+e^{C_4})$ we get the claim.
\end{proof}

\noindent
Now we can prove Proposition \ref{prop_local_manifold}.

\begin{proof}[Proof of Proposition \ref{prop_local_manifold}]
By Corollary \ref{exp_restricted}, 
\[
 \exp\big(U^s_{\exp} \cap H^s_\sigma(\R^n;\R^n)\big) \subseteq \Ds^s_\mu(\R^n).
\]
Lemma \ref{d_exp}, together with the inverse function theorem implies that there exists a neighborhood $U' \subseteq U^s_{\exp}$ of $0$ so that 
\[
 \exp:U' \to \Ds^s(\R^n)
\]
is a diffeomorphism onto its image. In particular
\[
 \exp:U' \cap H^s_\sigma(\R^n;\R^n) \to \Ds^s_\mu(\R^n)
\]
is $1-1$. It remains to show that there exists a neighborhood $\tilde U \subseteq U'$ of $0$ so that $\exp(\tilde U) \cap \Ds^s_\mu(\R^n)$ is contained in $\exp\big(\tilde U \cap H^s_\sigma(\R^n;\R^n)\big)$. Arguing by contraposition we show that there exists a neighborhood $\tilde U$ so that any $u_0 \in \tilde U$ with $\exp(u_0) \not\in \Ds^s_\mu(\R^n)$ is an element in $H^s(\R^n;\R^n)\setminus H^s_\sigma(\R^n;\R^n)$. By the formula \eqref{det_formula}, the condition $\exp(u_0) \not\in \Ds^s_\mu(\R^n)$, $u_0 \in U^s_{\exp}$, means for the corresponding solution $u(t)\equiv u(t;u_0)$ and the corresponding flow $\varphi(t)\equiv \varphi(t;u_0)$
\begin{equation}\label{verify_condition}
 \int_0^1 (\operatorname{div} u) (t,\varphi(t,x)) \,dt \neq 0 \mbox{ for some } x \in \R^n.
\end{equation}
In a first step we want to express $\int_0^1 (\operatorname{div} u(t)) \circ \varphi(t)\,dt$ in a convenient way. Integrating \eqref{t_derivative_explicit} gives
\[
 (\operatorname{div} u(t)) \circ \varphi(t) = \operatorname{div} u_0 + \int_0^t R_{\varphi(\tau)} \left(\chi(D)\big(2 (u(\tau) \cdot \nabla) \operatorname{div} u(\tau) + (\operatorname{div} u(\tau))^2 \big) \right) \,d\tau.
\] 
Integrating again we arrive at
\begin{multline}
 \label{integral_div}
 \int_0^1 (\operatorname{div} u(t)) \circ \varphi(t) \,dt = \operatorname{div} u_0 \\
 + \int_0^1 \int_0^t R_{\varphi(\tau)} \left(\chi(D)\big(2 (u(\tau) \cdot \nabla) \operatorname{div} u(\tau) + (\operatorname{div} u(\tau))^2 \big) \right) \,d\tau dt.
\end{multline}
The aim is to bound the $H^{s-1}$-norm of the left-hand side of the latter identity away from $0$. By Lemma \ref{lemma_linear_growth} there exists a ball $\tilde U \subseteq U'$, with $U'$ as above, centered at $0$ and $C_1 > 0$ such that for any $f \in H^{s-1}(\R^n;\R^n)$
\begin{equation}\label{composition_linear}
 ||R_\psi f||_{s-1} \leq C_1 ||f||_{s-1},\quad \forall\psi \in \exp(\tilde U).
\end{equation}
Thus we get for any $u_0 \in \tilde U$
\begin{multline*}
\big|\big| \int_0^1 \int_0^t R_{\varphi(\tau)} \left(\chi(D)\big(2 (u(\tau) \cdot \nabla)\operatorname{div} u(\tau) + (\operatorname{div} u(\tau))^2 \big) \right) \,d\tau dt \big|\big|_{s-1} \\ 
\leq \int_0^1 \int_0^t ||R_{\varphi(\tau)} \left(\chi(D)\big(2 (u(\tau) \cdot \nabla)\operatorname{div} u(\tau) + (\operatorname{div} u(\tau))^2 \big) \right)||_{s-1} \,d\tau dt\\
\leq C_1 \int_0^1 \int_0^t ||\chi(D)\big(2 (u(\tau) \cdot \nabla)\operatorname{div} u(\tau) + (\operatorname{div} u(\tau))^2 \big)||_{s-1} \,d\tau dt
\end{multline*}
where in the last inequality we used \eqref{composition_linear}. By Lemma \ref{properties_chi} there is an absolute constant $C_2>0$ such that for any $0 \leq \tau \leq 1$
\begin{multline*}
 ||\chi(D)\big(2 (u(\tau) \cdot \nabla)\operatorname{div} u(\tau) + (\operatorname{div} u(\tau))^2 \big)||_{s-1} \\
 \leq C_2 \big(||(u(\tau) \cdot \nabla) \operatorname{div} u(\tau)||_{s-2} + ||(\operatorname{div} u(\tau))^2||_{s-2} \big).
\end{multline*} 
By Lemma \ref{lemma_multiplication} there exists $C_3 >0$ such that for any $0 \leq \tau \leq 1$
\[
 || (\operatorname{div} u(\tau))^2||_{s-2} \leq C_3 ||u(\tau)||_s ||\operatorname{div} u(\tau)||_{s-1}
\]
and
\[
 ||2 (u(\tau) \cdot \nabla)\operatorname{div} u(\tau)||_{s-2} \leq C_3 ||u(\tau)||_s ||\operatorname{div} u(\tau)||_{s-1}.
\]
From the last two inequalities we conclude that there is an absolute constant $C_4 > 0$ such that for any $0 \leq \tau \leq 1$ and for any $u_0 \in \tilde U$
\[
 ||\chi(D)\big(2 (u(\tau) \cdot \nabla)\operatorname{div} u(\tau) + (\operatorname{div} u(\tau))^2 \big)||_{s-1} \leq C_4 ||u(\tau)||_s ||\operatorname{div} u(\tau)||_{s-1}.
\]
By Lemma \ref{lemma_small_u} -- Lemma \ref{lemma_linear_u} and after shrinking $\tilde U$, if necessary, we get for any $0 \leq \tau \leq 1$ and for any $u_0 \in \tilde U$
\[
 ||\chi(D)\big(2 (u(\tau) \cdot \nabla)\operatorname{div} u(\tau) + (\operatorname{div} u(\tau))^2 \big)||_{s-1} \leq \frac{1}{2} ||\operatorname{div} u_0||_{s-1}.
\]
Thus we get from \eqref{integral_div} for any $u_0 \in \tilde U$
\begin{equation}\label{div_ineq}
 ||\int_0^1 (\operatorname{div} u(t)) \circ \varphi(t) \,dt||_{s-1} \geq \frac{1}{2} ||\operatorname{div} u_0||_{s-1}.
\end{equation}
In particular we see from \eqref{div_ineq}, that for any $u_0 \in \tilde U$ with $\operatorname{div} u_0 \neq 0$ the statement \eqref{verify_condition} holds.
\end{proof}

\noindent
Now we can prove Theorem \ref{th_submanifold}

\begin{proof}[Proof of Theorem \ref{th_submanifold}]
It is to show that the property described in Definition \ref{def_submanifold} holds for $\Ds^s_\mu(\R^n)$. Let $\tilde U \subseteq H^s(\R^n;\R^n)$ be as in the statement of Proposition \ref{prop_local_manifold}. Then
\[
 \exp\big(\tilde U \cap H^s_\sigma(\R^n;\R^n)\big) = \exp(\tilde U) \cap \Ds^s_\mu(\R^n)
\]
and hence $\exp(\tilde U) \cap \Ds^s_\mu(\R^n)$ is a submanifold of $\exp(\tilde U)$ with
\[
 \left. \exp \right|_{\tilde U \cap H^s_\sigma(\R^n;\R^n)}
\]
being a parametrization. To show the analog conclusion for an arbitrary $\psi \in \Ds^s_\mu(\R^n)$ instead of $\operatorname{id} \in \Ds^s_\mu(\R^n)$ we use the group structure of $\Ds^s(\R^n)$ and $\Ds^s_\mu(\R^n)$. We claim that for any $\psi \in \Ds^s(\R^n)$
\[
 R_\psi:\Ds^s(\R^n) \to \Ds^s(\R^n), \quad \varphi \mapsto \varphi \circ \psi
\]
is real analytic. Indeed using the identification of $\Ds^s(\R^n)$ with $\Ds^s(\R^n) - \operatorname{id} \subseteq H^s(\R^n;\R^n)$, one has with $g=\psi - \operatorname{id}$,
\[
 R_\psi:f \mapsto g + f \circ \psi
\]
which is affine and hence real analytic. The map $R_\psi$ is invertible with inverse $R_\psi^{-1}$. Now as $\Ds^s_\mu(\R^n) \subseteq \Ds^s(\R^n)$ is a subgroup one has for any $\psi \in \Ds^s(\R^n)$
\[
 R_\psi\big(\exp(\tilde U \cap H^s_\sigma(\R^n;\R^n))\big) = R_\psi\big(\exp(\tilde U)\big) \cap \Ds^s_\mu(\R^n).
\]
Note that $R_\psi\big(\exp(\tilde U)\big)$ is a neighborhood of $\psi$ in $\Ds^s(\R^n)$. Hence 
\[
 \left. R_\psi \circ \exp \right|_{\tilde U \cap H^s_\sigma(\R^n;\R^n)}
\]
is a real analytic parametrization of $R_\psi\big(\exp(\tilde U)\big) \cap \Ds^s_\mu(\R^n)$. As $\psi \in \Ds^s_\mu(\R^n)$ is arbitrary we get by Definition \ref{def_submanifold} that $\Ds^s_\mu(\R^n)$ is a real analytic submanifold of $\Ds^s(\R^n)$. 
\end{proof}

By Theorem \ref{th_submanifold} we get a differential structure for $\Ds^s_\mu(\R^n)$. An immediate corollary is the following one.

\begin{Coro}\label{coro_exp_restricted}
The exponential map restricts to a real analytic map
\[
 \exp:U^s_{\exp} \cap H^s_\sigma(\R^n;\R^n) \to \Ds^s_\mu(\R^n).
\]
Moreover it is a diffeomorphism around $0$.
\end{Coro}

\begin{Rem}
The tangent space of $\Ds^s_\mu(\R^n)$ at $\operatorname{id} \in \Ds^s_\mu(\R^n)$, as a subspace of $T_{\operatorname{id}}\Ds^s(\R^n)\equiv H^s(\R^n;\R^n)$, is given by
\[
T_{\operatorname{id}}\Ds^s_\mu(\R^n) = H^s_\sigma(\R^n;\R^n).
\]
Indeed the tangent space at $\operatorname{id} \in \Ds^s_\mu(\R^n)$ is by Corollary \ref{coro_exp_restricted} spanned by the vectors
\[
 \left. \partial_\varepsilon \right|_{\varepsilon=0} \exp(\varepsilon v)=v
\]
for $v \in H^s_\sigma(\R^n;\R^n)$. The tangent space at an arbitrary $\psi \in \Ds^s_\mu(\R^n)$ is the right translate of $H^s_\sigma(\R^n;\R^n)$ by $\psi$, i.e. $\tilde v$ is in $T_\psi \Ds^s_\mu(\R^n)$ iff it is of the form
\[
 \tilde v = v \circ \psi
\]
for some $v \in H^s_\sigma(\R^n;\R^n)$.
\end{Rem}
\clearpage
\appendix

\chapter{Analyticity in real Banach spaces}\label{appendix_analyticity}

The references for this section are \cite{mujica, analyticity}. For differential calculus in Banach spaces see e.g. \cite{dieudonne}. In the following $X$, $Y$, $Z$ will denote {\em real} Banach spaces with the corresponding norms $||\cdot||_X$, $||\cdot||_Y$, $||\cdot||_Z$. We denote by $L^k(X;Y)$ the space of continuous $k$-linear forms on $X \times \ldots \times X$ ($k$-times) with values in $Y$. For any symmetric $\tilde Q \in L^k(X;Y)$ denote by $Q$ the restriction of $\tilde Q$ onto the diagonal. $Q$ is referred to as the homogeneous polynomial associated to $\tilde Q$. For a sequence of symmetric $k$-linear forms $(\tilde Q_k)_{k \geq 0}$, $\tilde Q_k \in L^k(X;Y)$, with the corresponding homogeneouos polynomials $(Q_k)_{k \geq 0}$ consider the power series around $x_0 \in X$
\begin{equation}\label{formal_power}
f(x) = \sum_{k \geq 0} Q_k (x-x_0):=\sum_{k \geq 0} \tilde Q_k(x-x_0,\ldots,x-x_0).
\end{equation}
Following \cite{mujica, analyticity} we call the convergence radius of the power series
\[
 \sum_{k \geq 0} ||Q_k|| t^k, \quad t \in \R
\]
the radius (of convergence) of the series given in \eqref{formal_power}, where we denote by $||Q_k||$ the norm of the homogeneous polynomial $Q_k$, i.e.
\begin{equation}\label{radius_condition}
 ||Q_k|| := \sup_{||x||_X \leq 1} ||Q_k (x)||_Y.
\end{equation}
Thus by the Cauchy-Hadamard formula (see e.g \cite{mujica}) the radius $R$ of the series \eqref{formal_power} is given by
\begin{equation}\label{hadamard_formula}
 1/R = \limsup_{k \to \infty} ||Q_k||^{1/k}.
\end{equation}
We will use this in the following form: If the power series \eqref{formal_power} has radius $R > 0$ we then have
\begin{equation}\label{alternative_description}
\sup_{k \geq 0} ||Q_k|| r^k < \infty
\end{equation}
for any $0 \leq r < R$. On the other hand, if \eqref{alternative_description} holds for any $0 \leq r < R$ then the power series has (at least) radius $R$.\\
Now to the notion of real analyticity.

\begin{Def}\label{def_analytic}
We say that $f:U \subseteq X \to Y$ is real analytic in the open set $U$ if for all $x_0 \in U$ the map $f$ can be represented in a ball around $x_0$ of radius $r > 0$ as a power series of the form \eqref{formal_power} with radius $R \geq r$, i.e. we have
\[
 f(x) = \sum_{k \geq 0} Q_k (x-x_0), \quad ||x-x_0||_X < r.
\]
\end{Def}

As is shown in \cite{analyticity} a power series of the form \eqref{formal_power} with radius $R > 0$ defines a real analytic map in the ball $||x-x_0||_X < R$. There it is also shown that a real analytic map is $C^\infty$ and that composition of real analytic maps is again real analytic. These properties allow the notion of submanifold and the corresponding notion of real analytic maps in the category of real analytic objects. We will use the following form of the definition of a submanifold.

\begin{Def}\label{def_submanifold}
Let $X$ be a real Hilbert space and $U \subseteq X$ a non-empty open subset. We say that $M \subseteq U$,$M \neq \emptyset$, is a real analytic submanifold of $U$ if there is some closed subspace $V \subseteq X$ such that for all $m \in M$ there is some neighborhood $W \subseteq X$ of $m$, an open neighborhood $G$ of $0$ and a real analytic diffeomorphism $\Phi$ ($\Phi^{-1}$ is also real analytic) 
\[
 \Phi:G \to W
\]
such that we have
\[
 \Phi(G \cap V) = W \cap M.
\]
One calls $\left. \Phi \right|_{G \cap V}$ a parametrization of $W \cap M$.
\end{Def}

An existence and uniqueness theorem for analytic ODE's can be found in \cite{dieudonne}. Actually in \cite{dieudonne} they just discuss the situation for complex Banach spaces. But by complexification one immediately gets the analog result for real Banach spaces which reads as

\begin{Prop}\label{prop_analytic_flow}
Let $V:O \subseteq X \to X$ be a real analytic map (vector field) on the open set $O$. For every $w \in O$ there is a $T >0$ and $\delta > 0$ with $B_\delta(w) \subseteq O$ such that for any $u_0 \in B_\delta(w)$ the initial value problem
\begin{equation}\label{analytic_IVP}
 \dot \gamma(t) = V\big(\gamma(t)\big); \quad \gamma(0)=u_0
\end{equation}
has a unique solution $\gamma(t)=\Psi(t,u_0)$. Moreover the flow
\[
 \Psi : (-T,T) \times B_\delta(u_0) \to O
\]
is real analytic.
\end{Prop}

The following criterion (see also \cite{alex2} for a more general result) is used in Chapter \ref{section_lagrangian_formulation} to prove that a given map is real analytic.

\begin{Prop}\label{prop_weak_analytic}
Let $\big(X,\langle \cdot,\cdot \rangle_X\big)$ and $\big(Y,\langle \cdot, \cdot \rangle_Y\big)$ be real Hilbert spaces. Let $\phi:U \subseteq X \to Y$ be a map on the open subset $U \subseteq X$. Assume that the following property holds for some $x_0 \in U$ and $R>0$: For any $y \in Y$ we have that the map 
\[
 \langle \phi(\cdot),y \rangle_Y : U \to \R
\]
admits a power series representation with radius $R$ around $x_0$. Then $\phi:U \to Y$ admits a power series representation of radius $R$ around $x_0$.
\end{Prop}

\begin{Rem}\label{rem_weakly_analytic}
This is somehow the version of ''weakly holomorphic implies holomorphic'' suitable for real Hilbert spaces. For complex Hilbert spaces the situation is much easier (see e.g. \cite{mujica}).
\end{Rem}

To prove this proposition we need the following lemma (see also \cite{alex} for a more general formulation). It is just Proposition \ref{prop_weak_analytic} for the case $X=\R$.

\begin{Lemma}\label{lemma_weak_curve}
Let $Y$ be as in Proposition \ref{prop_weak_analytic} and $\gamma:(-R,R) \to Y$ a curve such that for every $y \in Y$ the map
\[
 \langle \gamma(\cdot),y \rangle_Y :(-R,R) \to \R
\]
has a convergent power series with radius $R$ around $0$. Then $\gamma:(-R,R) \to Y$ admits a power series representation
\[
 \gamma(t) = \sum_{k \geq 0} \frac{1}{k!} a_k t^k, \quad(a_k)_{k \geq 0} \subseteq Y
\]
with radius $R$.
\end{Lemma}

\begin{proof}
For a function $f:(-R,R) \to \R$ we define for $k \geq 0$ the finite differences recursively by
\[
 \Delta_h f(t) = f(t+h)-f(t), \ldots, \Delta_h^{k+1} f(t) = \Delta_h^k f(t+h) - \Delta_h^k f(t).
\]
For a fixed $k \geq 0$ these expressions make sense for $t \in (-R,R)$ and $h$ small enough. These finite differences are defined in the same way for $Y$-valued $f$. Furthermore we have for smooth $f$
\begin{equation}\label{derivatives}
  \partial_t^k f(t) = \lim_{h \to 0} \frac{\Delta_h^k f(t)}{h^k}.
\end{equation}
By assumption we have for every  $y \in Y$
\[
 f^{(y)}(t):= \langle \phi(t),y \rangle = \sum_{k \geq 0} \frac{1}{k!} a_k^{(y)} t^k, \quad (a_k^{(y)})_{k \geq 0} \subseteq \R
\]
a power series expansion with radius $R$. By linearity we have
\[
 \frac{\Delta_h^k f^{(y)}(0)}{h^k} = \langle \frac{\Delta_h^k \phi(0)}{h^k},y \rangle_Y.
\]
From \eqref{derivatives} we get 
\[
 \frac{\Delta_h^k f^{(y)}(0)}{h^k} \to a_k^{(y)}
\]
as $h \to 0$. As this holds for every $y \in Y$, we get for some $a_k \in Y$
\[
 \frac{\Delta_h^k \phi(0)}{h^k} \rightharpoonup a_k
\]
i.e., it converges weakly to $a_k$ in $Y$. This $a_k$ has the property
\[
 a_k^{(y)} = \langle a_k,y \rangle_Y
\]
for all $y \in Y$. As $f^{(y)}(t)$ has convergence radius $R$ we have
\[
 \sup_{k \geq 0} \frac{1}{k!} |a_k^{(y)}| r^k < \infty \quad \mbox{ or equivalently } \quad \sup_{k \geq 0} \frac{1}{k!} |\langle a_k,y \rangle_Y| r^k < \infty 
\]
for all $y \in Y$ and $r < R$. By the uniform boundedness principle we then have
\[
 \sup_{k \geq 0} \frac{1}{k!} ||a_k||_Y r^k < \infty
\]
for all $r < R$. This means that $\tilde \phi:(-R,R) \to Y$ defined by
\[
 \tilde \phi(t) := \sum_{k \geq 0} \frac{1}{k!} a_k t^k
\]
is a power series with radius $R$. We have for any $y \in Y$ and $t \in (-R,R)$
\[
 \langle \tilde \phi(t),y \rangle_Y = \sum_{k \geq 0} \frac{1}{k!} \langle a_k, y \rangle_Y t^k = \langle \phi(t), y \rangle_Y
\]
which means $\tilde \phi(t) = \phi(t)$. This shows the lemma.
\end{proof}

\noindent
With the help of this lemma we can prove the proposition.

\begin{proof}[Proof of Proposition \ref{prop_weak_analytic}]
Without loss of generality we assume $x_0=0$. Let $v \in X \setminus \{0\}$. Consider the curve $t \mapsto \phi(t v)$. By assumption $\langle \phi(tv),y \rangle_Y$ has a convergent power series around $0$ with radius $R/||v||_X$. As $R$ does not depend on $y$ we can apply Lemma \ref{lemma_weak_curve} to $t \mapsto \phi(tv)$ and we get that it is a smooth curve. In particular
\[
 Q_k(v):= \left. \partial_t^k \right|_{t=0} \phi(tv)
\]
is well-defined. On the other hand we have by assumption, for any fixed $y \in Y$,
\[
 \langle \phi(v),y \rangle_Y = \sum_{k \geq 0} \frac{1}{k!} Q_k^{(y)}(v)
\]
for some $\R$-valued homogeneuos polynomial $Q_k^{(y)}$ of order $k$, $k \geq 0$. As we have
\[
 \left. \partial_t^k \right|_{t=0} \langle \phi(tv),y\rangle_Y = Q_k^{(y)}(v)
\]
we get for all $y \in Y$
\[
 \langle Q_k(v),y \rangle_Y = Q_k^{(y)}(v).
\]
By \cite{mujica} we know that a weakly continuouos polynomial is a continuous polynomial, i.e. $Q_k(v)$ is a homogeneuous $Y$-valued polynomial in $X$ of order $k$. As the power series with $Q_k^{(y)}$ has radius $R$, we have
\[
 \sup_{k \geq 0} \frac{1}{k!} ||Q_k^{(y)}|| r^k < \infty
\]
for all $y \in Y$ and $r < R$ where $||Q_k^{(y)}||$ is the norm of the $\R$-valued homogeneous polynomial $Q_k^{(y)}$. Again by the uniform boundedness principle we conclude that
\[
 \sup_{k \geq 0} \frac{1}{k!} ||Q_k|| r^k < \infty
\]
for all $r < R$, where here $||Q_k||$ is the norm of the $Y$-valued homogeneous polynomial $Q_k$. Therefore $\tilde \phi:B_R(0) \to Y$ defined by 
\[
 \tilde \phi(v) := \sum_{k \geq 0} \frac{1}{k!} Q_k(v)
\]
is a power series with radius $R$. Now we have for all $y \in Y$ and for all $v \in B_R(0)$
\[
 \langle \tilde \phi(v),y \rangle_Y = \sum_{k \geq 0} \frac{1}{k!} \langle Q_k(v),y \rangle_Y = \sum_{k \geq 0} \frac{1}{k!} Q_k^{(y)}(v) = \langle \phi(v),y \rangle_Y.
\]
Therefore $\tilde \phi(v) = \phi(v)$. Hence the claim.
\end{proof}

\noindent
Sometimes we have to deal with maps which are linear in one entry, i.e. maps of the form
\[
 \phi:(Y \times O) \subseteq Y \times X \to Z
\]
where $\phi(\cdot,x)$, $x \in O$, is linear in the first entry, i.e. $\phi(\cdot,x) \in L(Y;Z)$. For such maps we have the following lemma

\begin{Lemma}\label{lemma_linear}
Assume that
\[
 \phi : Y \times O \to Z
\]
is real analytic and linear in the first entry and has a power series expansion around $(0,x_0) \in Y \times X$ with radius $R$ where $O \subseteq X$ is open and $x_0 \in X$. Then
\begin{eqnarray*}
 \tilde \phi : O &\to& L(Y;Z) \\
 x &\mapsto& \big( y \mapsto \phi(y,x)\big)
\end{eqnarray*}
has a power series expansion with radius $R$ around $x_0$.
\end{Lemma}

\begin{proof}
Without loss of generality we assume $x_0=0$. We have for any fixed $y \in Y$ with $||y||<R$, by Taylor's theorem the following expansion around $x_0=0$
\begin{equation}\label{desired_expansion}
 \phi(y,x) = \sum_{k \geq 0} \frac{1}{k!} d_{2,(y,0)}^k \phi \, x^k
\end{equation}
where $d_2$ denotes the partial derivative in the second entry, i.e. 
\begin{equation}\label{formula1}
 d_{2,(y,0)}^k \phi \,x^k = d^k_{(y,0)} \phi \,(0,x)^k.
\end{equation}
Here we use for a Banach space $W$ and $w \in W$ the notation $w^k$ for $(w,\ldots,w) \in W \times \cdots \times W$ ($k$-times) and $d^k_p \phi \, w^k$ stands for the $k$'th order differential at the point $p$ evaluated in $w^k$. One has 
\begin{equation}\label{formula2}
d_{2,(y,0)}^k \phi \,x^k = \left. \partial_t^k \right|_{t=0} \phi(y,tx)
\end{equation}
We see from \eqref{formula2} that
\[
 y \to d_{2,(y,0)}^k \phi(y,0) \,x^k
\]
is linear. Recall that we have the canonical isomorphism $L^{k+1}(Y \times X\times \cdots \times X;Z) \simeq L^k(X \times \cdots \times X;L(Y;Z))$. Therefore we can look at $d_{2,(\cdot,0)}^k$ as a polynomial in $X$ with values in $L(Y;Z)$. Thus \eqref{desired_expansion} will be the desired expansion. But we have to estimate the corresponding norms to ensure that it has radius $R$. Take $0 < \delta < R$. For any fixed $y \in Y$ with $||y||_Y < \delta$ we have a power series for $x \mapsto \phi(y,x)$ with radius $R-\delta$ -- see \cite{analyticity}. Thus we have
\[
 \sup_{k \geq 0} \frac{1}{k!} \left(\sup_{||x||_X \leq 1} ||d^k_{(y,0)} \phi \,(0,x)^k||_Z \right)r^k < \infty
\]
for all $r < R-\delta$. By the linearity of $d^k_{(y,0)}$ in $y$ this extends to all $y \in Y$. Hence by the uniform boundedness principle 
\[
 \sup_{k \geq 0} \frac{1}{k!} \left(\sup_{||y||_Y \leq 1} \sup_{||x||_X \leq 1} ||d^k_{(y,0)} \phi \,(0,x)^k||_Z \right) r^k < \infty
\]
for all $r < R - \delta$. Thus the expansion \eqref{desired_expansion} has radius $R-\delta$. By letting $\delta \to 0$ we get the claim. 
\end{proof}

Finally we give an example of a real analytic operation which will be needed.

\begin{Lemma}\label{analytic_det}
Let $s > n/2+1$. The map
\[
 \phi:H^{s-1}(\R^n) \times \Ds^s(\R^n) \to H^{s-1}(\R^n),\quad (f,\varphi) \mapsto \frac{f}{\det(d\varphi)}
\]
is real analytic.
\end{Lemma}

\begin{proof}
Consider the map
\[
 \Psi:\Ds^s(\R^n) \to \mathcal L\big(H^{s-1}(\R^n),H^{s-1}(\R^n)\big),\quad \varphi \mapsto \left[ f \mapsto f \cdot \det(d\varphi)\right]
\]
which is real analytic. From \cite{composition} we know that $\phi$ is welldefined (and continuous). This means that $\Psi$ maps into the invertible linear maps. Since the inversion map 
$\operatorname{inv}:T \mapsto T^{-1}$ is real analytic (cf. Neumann series), we see that
\[
 \phi(f,\varphi)=\operatorname{inv}(\Psi) (f)
\]
is real analytic.
\end{proof}

\chapter{Integration of $H^s$-vector fields}\label{integration}

The goal of this section is to prove that we can integrate a $H^s$-vector field to a flow in $\Ds^s(\R^n)$. More precisely

\begin{Prop}\label{prop_integration}
Let $s > n/2+1$ and $T>0$. For a given $u \in C\big([0,T];H^s(\R^n;\R^n)\big)$ there is a unique $\varphi \in C^1\big([0,T];\Ds^s(\R^n)\big)$ solving
\[
 \partial_t \varphi = u \circ \varphi \mbox{ on } [0,T];\quad \varphi(0)=\operatorname{id} \in \Ds^s(\R^n).
\]
\end{Prop}

\begin{Rem}
Proposition \ref{prop_integration} was proved in \cite{fischer}. The idea there is the following. If we write $\varphi^{-1}=\operatorname{id}+f$, where $f \in C^0\big([0,T];H^s(\R^n;\R^n)\big)$, we get by differentiating $\varphi^{-1} \circ \varphi = \operatorname{id}$
\[
 \partial_t f \circ \varphi + (I_n + df) \circ \varphi \cdot \partial_t \varphi = 0 
\]
or the following transport equation for $f$
\[
 \partial_t f + u + df \cdot u=0.
\]
with coefficients in $H^s$. Now one can use the theory for linear symmetric hyperbolic systems developed in \cite{fischer} to solve this problem. But we will give a more ''dynamical systems''-proof.
\end{Rem}

The uniqueness part of the proposition is an easy task. Indeed by the Sobolev imbedding \eqref{sobolev_imbedding} we see that $u$ is a uniformly Lipschitz vector field $u:[0,T] \times \R^n \to \R^n$ with respect to the spatial variable, because we have
\[
 |u(t,x)-u(t,y)| \leq C ||u(t)||_s |x-y| \leq CM |x-y|
\]
where $M=\max_{0 \leq t \leq T} ||u(t)||_s$. Thus we have a unique flow $\tilde \varphi:[0,T] \times \R^n \to \R^n$.\\ \\
\noindent
Before proving the proposition we will make some preparation. Since the composition map is linear in the first entry we can get the following local linear growth estimate.

\begin{Lemma}\label{lemma_linear_growth}
Let $s > n/2+1$, $0 \leq s' \leq s$ and $\varphi_\bullet \in \Ds^s(\R^n)$ be given. Then there is a neighborhood $\mathcal G$ of $\varphi_\bullet$ in $\Ds^s(\R^n)$ and a $C > 0$ with
\[
 \frac{1}{C} ||f||_{s'} \leq ||f \circ \varphi||_{s'} \leq C ||f||_{s'}
\]
for all $f \in H^{s'}(\R^n)$ and $\varphi \in \mathcal G$.
\end{Lemma}

\begin{proof}
Consider the composition map
\[
 \mu:H^{s'}(\R^n) \times \Ds^s(\R^n) \to H^{s'}(\R^n),\quad (f,\varphi) \mapsto f \circ \varphi
\]
which by \cite{composition} is continuous. As we have $\mu(0,\varphi_\bullet)=0$ there exist, by the continuity of $\mu$, $R > 0$ and a neighborhood $\mathcal G$ of $\varphi_\bullet$ such that we have
\[
 ||f \circ \varphi||_{s'} \leq 1
\]
for all $f \in H^{s'}(\R^n)$ with $||f||_{s'} \leq R$ and for all $\varphi \in \mathcal G$. By linearity we thus get
\[
 ||f \circ \varphi||_{s'} \leq \frac{1}{R} ||f||_{s'}
\]
for all $f \in H^{s'}(\R^n)$ and for all $\varphi \in \mathcal G$. The same reasoning gives, by shrinking $R$ and $\mathcal G$ if necessary,
\[
 ||g \circ \varphi^{-1}||_{s'} \leq \frac{1}{R} ||g||_{s'}
\]
for all $g \in H^{s'}(\R^n;\R^n)$ and $\varphi \in \mathcal G$. Replacing $g$ by $f \circ \varphi$ we get the claim.
\end{proof}

\noindent
The following lemma is a special case of Proposition \ref{prop_integration}. The proof follows the one given in \cite{ebin}.

\begin{Lemma}\label{lemma_integration}
Assume $s > n/2+2$. Then the claim of Proposition \ref{prop_integration} holds.
\end{Lemma}

\begin{proof}
Note that for $s>n/2+2$ the space $\Ds^{s-1}(\R^n)$ is defined. In a neighborhood of $\operatorname{id} \in \Ds^{s-1}(\R^n)$, let's say
\[
 \mathcal G^{s-1}_\varepsilon:=\big\{ \varphi \in \Ds^{s-1}(\R^n) \; \big| \; ||\varphi - \operatorname{id}||_{s-1} < \varepsilon \big\}
\]
we have by Lemma \ref{lemma_linear_growth} for some constant $C > 0$ 
\[
 || f \circ \varphi||_{s-1} \leq C ||f||_{s-1}
\]
for all $f \in H^{s-1}(\R^n)$ and for all $\varphi \in \mathcal G^{s-1}_\varepsilon$. By shrinking $\varepsilon$ we can assume that $\operatorname{id} + g \in \Ds^{s-1}(\R^n)$ for all $g \in H^{s-1}(\R^n;\R^n)$ with $||g||_{s-1} < \varepsilon$. Now consider for the given $u \in C\big([0,T];H^s(\R^n;\R^n)\big)$ the map
\[
 V: [0,T] \times \Ds^{s-1}(\R^n) \to \Ds^{s-1}(\R^n),\quad (t,\varphi) \mapsto u(t) \circ \varphi.
\]
From \cite{composition} we know that $V$ is a time-dependent vector field on $\Ds^{s-1}(\R^n)$, which is continuous in the time variable and $C^1$ in the $\varphi$ variable. By the existence theory for ODE's (see e.g. \cite{dieudonne}) we know that there is some $\delta > 0$ and a $\psi \in C^1\big([0,\delta],\Ds^{s-1}(\R^n)\big)$ with
\[
 \partial_t \psi = u \circ \psi \mbox{ on } [0,\delta];\quad \psi(0)=\operatorname{id}.
\]
Assume now that we have for some $0 < \delta' \leq \delta$
\[
 ||\psi(t) - \operatorname{id}||_{s-1} < \varepsilon
\]
for $0 \leq t \leq \delta'$. Note that by continuity such a $\delta'$ exists. Recall that we have
\[
 \psi(t) = \operatorname{id} + \int_0^t u(\tau) \circ \psi(\tau) \,d\tau.
\]
for all $0 \leq t \leq \delta'$. Thus we get for any $t \in [0,\delta']$
\[
 ||\psi(t) - \operatorname{id}||_{s-1} \leq C \int_0^t ||u(\tau)||_{s-1} \,d\tau \leq CM\delta'
\]
where $M=\max_{0 \leq \tau \leq T} ||u(\tau)||_s$. In particular by choosing $\delta' \leq \varepsilon/(2CM)$ we get
\begin{equation}\label{stays_in_ball}
 ||\psi(t)-\operatorname{id}||_{s-1} \leq \varepsilon/2
\end{equation}
for $0 \leq t \leq \delta'$. As $C$ is fixed, this choice of $\delta'$ just depends on $M$ and not on the particular values of $u$. Thus we see that $\forall t_0 \in [0,T]$ the ODE
\[
 \partial_t \psi = u \circ \psi;\quad \psi(t_0)=\operatorname{id}
\]
has a solution on $[t_0,t_0 + \delta'] \cap [t_0,T]$ as for these values of $t$ the condition \eqref{stays_in_ball} is preserved. Now we proceed as follows: We solve
\[
 \partial_t \psi_1 = u \circ \psi_1;\quad \psi_1(0)=\operatorname{id}
\]
on $[0,\delta']$. Then we solve
\[
 \partial_t \psi_2 = u \circ \psi_2;\quad \psi_2(\delta')=\operatorname{id}
\]
on $[\delta',2\delta']$ (without loss we can assume $2\delta' \leq T$) and define $\varphi:[0,2\delta'] \to \Ds^{s-1}(\R^n)$ by
\[
 \varphi(t) = \begin{cases} \psi_1(t), \quad & t \in [0,\delta') \\ \psi_2(t) \circ \psi_1(\delta'), \quad & t \in [\delta',2\delta'] \end{cases}.
\]
From the definition it is clear that $\varphi \in C\big([0,2\delta'];D^{s-1}(\R^n)\big)$. From the properties of $\psi_1,\psi_2$ we have
\[
 \varphi(t) = \operatorname{id} + \int_0^t u(\tau) \circ \varphi(\tau) \,d\tau
\]
for all $t \in [0,2\delta']$. Indeed on $[0,\delta']$ this is clear. For $t \in [\delta',2\delta']$ we have
\[
 \psi_2(t) = \operatorname{id} + \int_{\delta'}^t u(\tau) \circ \psi_2(\tau) \,d\tau
\]
or 
\[
 \psi_2(t)-\operatorname{id} = \int_{\delta'}^t u(\tau) \circ \psi_2(\tau) \,d\tau.
\]
Applying the continuous linear operator $R_{\psi_1(\delta')}$ to this equation we get
\[
 \psi_2(t) \circ \psi_1(\delta') = \psi_1(\delta') + \int_{\delta'}^t u(\tau) \circ \psi_2(\tau) \circ \psi_1(\delta') \,d\tau
\]
which is by definition
\[
 \varphi(t) = \varphi(\delta') + \int_{\delta'}^t u(\tau) \circ \varphi(\tau) \,d\tau
\]
showing the claim. Iterating this procedure we can construct a solution 
\[
\varphi \in C^1\big([0,T];\Ds^{s-1}(\R^n)\big).
\]
Next we want to show that we have actually 
\[
\varphi \in C^1\big([0,T];\Ds^s(\R^n)\big).
\]
Writing $\varphi=\operatorname{id} + f$ where $f \in C^1\big([0,T];H^{s-1}(\R^n)\big)$ we get by taking the differential of $\partial_t \varphi = u \circ \varphi$
\begin{equation}\label{ode_df}
 \partial_t df = du \circ \varphi \cdot (I_n + df)
\end{equation}
where $df \in C^1\big([0,T];H^{s-2}(\R^n;\R^{n \times n})\big)$ denotes the Jacobian of $f$ and $I_n$ the $n \times n$-identity matrix. As we have by the results for the composition map given in \cite{composition}
\[
 du \circ \varphi \in C^0\big([0,T];H^{s-1}(\R^n;\R^{n \times n})\big)
\]
we can view \eqref{ode_df} as a inhomogenous linear ODE with coefficients in $H^{s-1}$. By uniqueness of solutions this means that $df$ lies actually in 
\[
C^1\big([0,T];H^{s-1}(\R^n;\R^{n \times n})\big). 
\]
This show that $\varphi \in C^1\big([0,T];\Ds^s(\R^n)\big)$. Hence the claim. 
\end{proof}

\noindent
To prove Proposition \ref{prop_integration} we will need the following well-known lemmas.

\begin{Lemma}\label{lemma_interpolation}
Let $f \in H^s(\R^n)$, $s \geq 0$. Then we have the following interpolation inequality for $0 \leq s' \leq s$ and $\lambda \in (0,1)$
\begin{equation}\label{interpolation_ineq}
||f||_{\lambda s' + (1-\lambda) s} \leq ||f||_{s'}^\lambda ||f||_s^{1-\lambda}.
\end{equation}
\end{Lemma}

\begin{proof}
We have by definition
\begin{eqnarray*}
||f||^2_{\lambda s' + (1-\lambda) s} &=& \int_{\R^n} (1+|\xi|^2)^{\lambda s' + (1-\lambda) s} |\hat f(\xi)|^2 d\xi \\
&=& \int_{\R^n} (1+|\xi|^2)^{\lambda s'} |\hat f(\xi)|^{2\lambda} (1+|\xi|^2)^{(1-\lambda)s} |\hat f(\xi)|^{2(1-\lambda)} d\xi \\
\noalign{\noindent and using the H\"older inequality}\\
&\leq& ||(1+|\xi|^2)^{\lambda s'} |\hat f(\xi)|^{2\lambda}||_{L^\frac{1}{\lambda}} ||(1+|\xi|^2)^{(1-\lambda)s} |\hat f(\xi)|^{2(1-\lambda)} ||_{L^\frac{1}{1-\lambda}} \\
&=& ||f||_{s'}^{2\lambda} ||f||_s^{2(1-\lambda)} 
\end{eqnarray*}
which shows the claim.
\end{proof}

\noindent
For approximating functions by regular ones we have

\begin{Lemma}\label{lemma_approximation}
Let $f \in H^s(\R^n)$, $s \geq 0$. Let $\chi_k(D)$, $k \geq 1$, be the Fourier multiplier with symbol $\chi_k$ given by 
\[
 \chi_k(\xi)=\begin{cases} 1, \quad &|\xi| \leq k \\ 0, \quad & |\xi| > k \end{cases}
\]
Then we have $\chi_k(D) f \in H^{s+1}(\R^n)$ and
\[
 \chi_k(D) f \to f \quad \mbox{ in } H^s(\R^n)
\]
as $k \to \infty$.
\end{Lemma}

\begin{proof}
That $\chi_k(D) f \in H^{s+1}(\R^n)$ follows from
\begin{multline*}
 ||\chi_k(D)f||_{s+1}^2 = \int_{|\xi| \leq k} (1+|\xi|^2)^{s+1} |\hat f(\xi)|^2 d\xi \leq (1+k^2)^{s+1} \int_{\R^n} |\hat f(\xi)|^2 d\xi < \infty.
\end{multline*}
One has actually $\chi_k(D)f \in H^\infty(\R^n)=\cap_{s \geq 0} H^s(\R^n)$, but this is not needed here. For the second claim we write
\[
 ||\chi_k(D) f - f||_s^2 = \int_{|\xi| > k} (1+|\xi|^2)^s |\hat f(\xi)|^2 d\xi.
\]
Now by Lebesgue's dominated convergence we get
\[
 \int_{|\xi| > k} (1+|\xi|^2)^s |\hat f(\xi)|^2 d\xi \to 0
\]
as $k \to \infty$. Hence the claim.
\end{proof}

\noindent
We even have that this convergence is uniform on compact curves.

\begin{Coro}\label{coro_uniform}
Let $u \in C^0\big([0,T];H^s(\R^n)\big)$ for some $T > 0$. Then $\chi_k(D) u \in C^0\big([0,T];H^{s+1}(\R^n)\big)$ and
\[
 \sup_{0 \leq t \leq T} ||\chi_k(D) u(t) - u(t)||_s \to 0
\]
as $k \to 0$.
\end{Coro}

\begin{proof}
We will prove a slightly stronger result. We will prove that for a compact set $K \subseteq H^s(\R^n)$ we have
\[
 \chi_k(D) f \to f \quad \mbox{ in } H^s(\R^n)
\]
as $k \to \infty$ uniformly in $f \in K$. First note that $||\chi_k(D) f||_s \leq ||f||_s$. Let $\varepsilon > 0$. As $K$ is compact we have a finite set of points (let's say $M$ points) $(f_m)_{1 \leq m \leq M} \subseteq H^s(\R^n)$ such that
\[
 K \subseteq \cup_{m=1}^M B_\varepsilon(f_m)
\] 
where
\[
 B_\varepsilon(f) = \big\{ g \in H^s(\R^n) \,\big|\, ||g-f||_s < \varepsilon \big\}
\]
is the $\varepsilon$-ball in $H^s(\R^n)$ around $f$ with radius $\varepsilon$. By Lemma \ref{lemma_approximation} there is a $N$ such that
\[
 ||\chi_k(D) f_m - f_m||_s < \varepsilon 
\]
for all $k \geq N$ and $1 \leq m \leq M$. For an arbitrary $f \in K$ take a $f_j$, $1 \leq j \leq M$, with $f \in B_\varepsilon(f_j)$. With this choice we have
\begin{multline*}
 ||\chi_k(D) f - f||_s \leq ||\chi_k(D)f - \chi_k(D) f_j||_s + ||\chi_k(D) f_j - f_j ||_s + ||f_j - f||_s < 3 \varepsilon
\end{multline*}
for all $k \geq N$. This proves the claim for the compact set $K$. Now as the image of the curve $u$ is compact we get the desired result.
\end{proof}

\noindent
We know that there is some $\varepsilon > 0$ such that for all $g \in H^s(\R^n;\R^n)$ with $||g||_s < \varepsilon$ we have $\operatorname{id}+g \in \Ds^s(\R^n)$. Denote this set by $\mathcal G_\varepsilon^s$, i.e.
\[
 \mathcal G_\varepsilon^s = \big\{ \varphi \in \Ds^s(\R^n) \, \big| \, ||\varphi - \operatorname{id}||_s < \varepsilon \big\}.
\] 
By Lemma \ref{lemma_linear_growth} we get (by shrinking $\varepsilon$ if necessary) for all $\varphi \in \mathcal G_\varepsilon^s$ 
\begin{equation}\label{uniform_estimate1}
 ||f \circ \varphi||_{s-1} \leq C ||f||_{s-1}, \quad \forall f \in H^{s-1}(\R^n;\R^n)
\end{equation}
and
\begin{equation}\label{uniform_estimate2}
 ||f \circ \varphi||_s \leq C ||f||_s, \quad \forall f \in H^s(\R^n;\R^n)
\end{equation}
for some $C>0$. We further assume by making $0 < \varepsilon < 1$ small enough that we have $\det(d_x \varphi) > \varepsilon$ for all $x \in \R^n$ and for all $\varphi \in \mathcal G_\varepsilon^s$. Because of the Sobolev imbedding \eqref{sobolev_imbedding} this is possible. Now with this choice of $\varepsilon$ resp. $\mathcal G_\varepsilon^s$ we prove the following Lipschitz type estimate.

\begin{Lemma}\label{lemma_lipschitz_type}
There is $\tilde C > 0$ such that for any $\varphi,\psi \in \mathcal G_\varepsilon^s$
\[
 ||f \circ \varphi - f \circ \psi||_{s-1} \leq \tilde C ||f||_s ||\varphi - \psi||_{s-1},\quad \forall f \in H^s(\R^n).
\]
\end{Lemma}

\begin{proof}
By the fundamental lemma of calculus we have pointwise
\begin{eqnarray}
\nonumber
 f \circ \varphi - f \circ \psi &=& \int_0^1 \partial_t \left( f\big(\psi + t(\varphi - \psi)\big) \right) dt \\
\label{fundamental_lemma}
&=& \int_0^1 \nabla f\big(\psi + t (\varphi - \psi)\big) (\varphi - \psi) dt
\end{eqnarray}
As $t \mapsto \psi +t (\varphi - \psi)$ is a continuous curve in $\mathcal G_\varepsilon^s$ we see that the integrand is a continuous curve $H^{s-1}(\R^n;\R^n)$. Indeed $\varphi - \psi \in H^s(\R^n;\R^n)$ and $H^{s-1}$ is a Banach algebra. Thus we see that \eqref{fundamental_lemma} is an identity in $H^{s-1}(\R^n;\R^n)$. Therefore we have for some $\tilde C >0$
\begin{eqnarray*}
||f \circ \varphi - f \circ \psi||_{s-1} &\leq& \int_0^1 \tilde C ||\nabla f \big(\psi + t (\varphi - \psi)\big)||_{s-1} ||\varphi - \psi||_{s-1} dt \\
&\leq& \tilde C ||f||_s ||\varphi - \psi||_{s-1}
\end{eqnarray*}
where we used \eqref{uniform_estimate1} implying
\[
 ||\nabla f \big(\psi + t (\varphi - \psi)\big)||_{s-1} \leq C ||\nabla f||_{s-1} \leq C ||f||_s
\] 
and the Banach algebra property of $H^{s-1}(\R^n)$. This finishes the proof. 
\end{proof}

\noindent
Now we can prove the main proposition. We will do this using some ''energy'' estimates. We take $\mathcal G_\varepsilon^s$ as in Lemma \ref{lemma_lipschitz_type} 

\begin{proof}[Proof of Proposition \ref{prop_integration}]
Let $u \in C^0\big([0,T];H^s(\R^n;\R^n)$ be the given continuous vector field. We define $u_k=\chi_k(D)u$. We know by Corollary \ref{coro_uniform} that $u_k(t) \to u(t)$ in $H^s$ uniformly in $t \in [0,T]$. By Lemma we know that $u_k \in C^0\big([0,T];H^{s+1}(\R^n;\R^n)\big)$. Now Lemma \ref{lemma_integration} gives us corresponding flows $\varphi_k \in C^1\big([0,T];\Ds^{s+1}(\R^n)\big)$ solving
\[
 \partial_t \varphi_k = u_k \circ \varphi_k \mbox{ on } [0,T];\quad \varphi_k(0)=\operatorname{id}.
\]
We will show first that $\varphi_k$ converges at least on some short time interval $[0,\delta]$ to the desired solution. Consider the integral relation
\[
 \varphi_k(t) = \operatorname{id} + \int_0^t u_k(\tau) \circ \varphi_k(\tau) \,d\tau, \quad t \in [0,T].
\]
We reason as in the proof of Lemma \ref{lemma_integration}. For $k \geq 1$ fixed, assume that $\varphi_k(t) \in \mathcal G_\varepsilon^s$ for all $0 \leq t \leq \delta'$, for some $\delta' > 0$. Then we have for $t \in [0,\delta']$
\[
 ||\varphi_k(t)-\operatorname{id}||_s \leq \int_0^t ||u_k(\tau) \circ \varphi_k(\tau)||_s \,d\tau \leq C \int_0^t ||u_k(\tau)||_s \, d\tau
\]
where we used \eqref{uniform_estimate2}. Now as we have $u_k \to u$ uniformly in $t \in [0,T]$ there is some $M >0$ with
\[ 
 ||u_k(t)||_s < M
\]
for all $t \in [0,T]$ and for all $k \geq 1$. Thus we see that for $\delta \leq \frac{\varepsilon}{2CM}$ we have $||\varphi_k(t) - \operatorname{id}||_s < \varepsilon$ for all $t \in [0,\delta]$, i.e. we have $\varphi_k(t) \in \mathcal G_\varepsilon^s$. Now we will show that $\varphi_k$ converges on $[0,\delta]$. We have for $t \in [0,\delta]$
\begin{eqnarray*}
 \varphi_k(t)-\varphi_j(t) &=& \int_0^t u_k \circ \varphi_k - u_j \circ \varphi_j \,d\tau \\
&=& \int_0^t u_k \circ \varphi_k - u_j \circ \varphi_k \,d\tau + \int_0^t u_j \circ \varphi_k - u_j \circ \varphi_j \,d\tau.
\end{eqnarray*}
Taking the $H^{s-1}$-norm we get for any $t \in [0,\delta]$
\[
||\varphi_k(t)- \varphi_j(t)||_{s-1} \leq C \int_0^t ||u_k - u_j||_s \,d\tau + C \int_0^t ||u_j||_s ||\varphi_k- \varphi_j||_{s-1} \,d\tau 
\]
where we used Lemma \ref{lemma_lipschitz_type}. Thus from Gronwall's lemma we get
\[
 ||\varphi_k(t)-\varphi_j(t)||_{s-1} \leq \left[C \int_0^\delta ||u_k-u_j||_s\right] e^{\delta C M}
\]
for all $t \in [0,\delta]$. Thus we see that $\varphi_k - \varphi_j$ is Cauchy in $C^0\big([0,\delta];H^{s-1}(\R^n;\R^n)\big)$. Now take a $\lambda \in (0,1)$ with
\[
 s'=\lambda (s-1) + (1-\lambda) s > n/2+1.
\]
As we have $s > n/2+1$ such a $\lambda$ exists. We then have by Lemma \ref{lemma_interpolation}
\begin{eqnarray*}
 ||\varphi_k(t)-\varphi_j(t)||_{s'} &\leq& ||\varphi_k(t)-\varphi_j(t)||_{s-1}^\lambda ||\varphi_k(t)-\varphi_j(t)||_s^{1-\lambda} \\
&\leq& ||\varphi_k(t)-\varphi_j(t)||_{s-1}^\lambda \left(||\varphi_k(t)-\operatorname{id}||_s + ||\varphi_j(t)-\operatorname{id}||_s\right)^{1-\lambda} \\
&\leq& ||\varphi_k(t)-\varphi_j(t)||_{s-1}^\lambda (2\varepsilon)^{1-\lambda}
\end{eqnarray*}  
showing that $\varphi_k$ converges in $H^{s'}$ on $[0,\delta]$. Thus there exists a $\varphi$ with $\varphi- \operatorname{id} \in C^0\big([0,\delta];H^{s'}(\R^n;\R^n)\big)$ such that we have 
\[
 ||\varphi_k(t) - \varphi(t)||_{s'} \to 0
\]
uniformly in $t \in [0,\delta]$. As $s' > n/2 +1$ we have by the Sobolev imbedding \eqref{sobolev_imbedding} for all $x \in \R^n$
\[
 \det(d_x \varphi) = \lim_{k \to \infty} \det(d_x \varphi_k) \geq \varepsilon > 0.
\]
Hence $\varphi \in C^0\big([0,\delta];\Ds^{s'}(\R^n)\big)$. We claim that $\varphi=\tilde \varphi$ on $[0,\delta]$. Recall that we denote by $\tilde \varphi$ the flow of the vector field $u:[0,T]\times \R^n \to \R^n$. To show that $\varphi$ and $\tilde \varphi$ agree on $[0,\delta]$ consider for $x \in \R^n$ and $t \in [0,\delta]$
\begin{equation}\label{integral_limit}
\varphi_k(t,x) = x + \int_0^t u_k\big(\tau,\varphi_k(\tau,x)\big) \,d\tau.
\end{equation}
By the Sobolev imbedding \eqref{sobolev_imbedding} we have (denoting by $|\cdot|$ the euclidean norm in $\R^n$)
\begin{multline*}
\left| u_k\big(t,\varphi_k(t,x)\big) - u\big(t,\varphi(t,x)\big)\right| \leq \left|u_k\big(t,\varphi_k(t,x)\big) - u\big(t,\varphi_k(t,x)\big)\right| \\
+ \left|u\big(t,\varphi_k(t,x)\big) - u\big(t,\varphi(t,x)\big)\right| \leq C ||u_k(t) - u(t)||_{s'} + C ||u||_{s'} ||\varphi_k(t) - \varphi(t)||_{s'}
\end{multline*}
which goes to $0$ uniformly in $t \in [0,\delta]$. Thus taking the limit in \eqref{integral_limit} we arrive at
\begin{equation}\label{limit_integral}
 \varphi(t,x) = x + \int_0^t u\big(\tau,\varphi(\tau,x)\big) \,d\tau.
\end{equation}
By continuity of the composition in $H^{s'}$ we see that the identity \eqref{limit_integral} holds in $H^{s'}$, i.e. we have
\begin{equation}\label{hs_identity}
\varphi(t) = \operatorname{id} + \int_0^t u(\tau) \circ \varphi(\tau) \,d\tau.
\end{equation}
Taking the differential in \eqref{limit_integral} and denoting by $I_n$ the $n \times n$ identity matrix, we have
\[
 d\varphi(t,x) = I_n + \int_0^t du\big(\tau,\varphi(\tau,x)\big) d\varphi(\tau,x) \,d\tau.
\]
Thus $dg:=d\varphi - I_n$ solves for fixed $x \in \R^n$ the ODE
\begin{equation}\label{ode_dg}
\partial_t dg = du \circ \varphi + du \circ \varphi \cdot dg.
\end{equation}
From \cite{composition}, as $s' > n/2+1$ and $s' \geq s-1$, we know that
\[
du \circ \varphi \in C^0\big([0,\delta];H^{s-1}(\R^n;\R^{n \times n})\big). 
\]
Since $H^{s-1}$ is an algebra, we can view \eqref{ode_dg} as a linear inhomogeneous ODE with coefficients in $H^{s-1}$. Thus $dg$ lies actually in 
\[
 C^1\big([0,\delta];H^{s-1}(\R^n;\R^{n \times n})\big).
\]
Thus we get $\varphi \in C^1\big([0,\delta];\Ds^s(\R^n)\big)$ and the identity \eqref{hs_identity} holds in $\Ds^s(\R^n)$. To get $\varphi$ on the whole time interval $[0,T]$ we proceed as in the proof of Lemma \ref{lemma_integration}. As $\delta$ just depends on $M$ we can extend $\varphi$ by $\delta$-steps. After finitely many steps we end up with the desired flow $\varphi \in C^1\big([0,T];\Ds^s(\R^n)\big)$ solving
\[
\partial_t \varphi = u \circ \varphi \mbox{ on } [0,T];\quad \varphi(0)=\operatorname{id}
\]
and this proves the proposition.
\end{proof}

\chapter{Auxiliary lemmas}\label{auxiliary}

We collected here some lemmas needed in Section \ref{section_proof_main}.\\
We denote by $C^\infty_{\sigma,c}(\R^n;\R^n)$ the space of smooth compactly supported divergence-free vector fields, i.e.
\[
 C^\infty_{\sigma,c}(\R^n;\R^n) = \big\{ \varphi \in C_c^\infty(\R^n;\R^n) \,\big| \, \operatorname{div} \varphi =0 \big\}
\]
and recall for $s \geq 0$ the space
\[
 H_\sigma^s(\R^n;\R^n) = \big\{ f \in H^s(\R^n;\R^n) \,\big| \, \operatorname{div} f = 0 \big\}
\]
where $\operatorname{div}$ is understood in the sense of distributions. Clearly we have $C^\infty_{\sigma,c}(\R^n;\R^n) \subseteq H_\sigma^s(\R^n;\R^n)$. We even have

\begin{Lemma}\label{lemma_dense}
The subspace $C^\infty_{\sigma,c}(\R^n;\R^n)$ is dense in $H_\sigma^s(\R^n;\R^n)$.
\end{Lemma}

As an application we have

\begin{Coro}\label{coro_nablap}
Let $p \in H^1_{\operatorname{loc}}(\R^n)$ with $\nabla p \in L^2(\R^n;\R^n)$. Then for any $s \geq 0$
\[
 \langle w,\nabla p \rangle_{L^2} =0,\quad \forall w \in H^s_\sigma(\R^n;\R^n).
\]
\end{Coro}

\begin{proof}
First let $w \in C_{c,\sigma}^\infty(\R^n;\R^n)$. We then have for any $s \geq 0$
\[
 \langle w, \nabla p \rangle_{L^2} = -\langle \operatorname{div} w,p \rangle_{L^2} = 0.
\]
Take now an arbitrary $w \in H^s_\sigma(\R^n;\R^n)$. By Lemma \ref{lemma_dense} take a sequence $(w_k)_{k \geq 1} \subseteq C_{\sigma,c}^\infty(\R^n;\R^n)$ with
\[
 w_k \to w \mbox{ in } L^2(\R^n;\R^n).
\]
We then have
\[
\langle w,\nabla p \rangle_{L^2}=\lim_{k \to \infty} \langle w_k,\nabla p \rangle_{L^2} = 0
\]
showing the claim.
\end{proof}

To prove Lemma \ref{lemma_dense} we need the following version of the Poincar\'e Lemma.

\begin{Lemma}\label{lemma_poincare}
Let $V=(V_1,\ldots,V_n) \in C^\infty(\R^n;\R^n)$ be a smooth vector field with
\begin{equation}\label{curl_free}
\partial_i V_j - \partial_j V_i = 0, \quad \forall 1 \leq i,j \leq n.
\end{equation}
Then there is a $p \in C^\infty(\R^n)$ with
\[
 V= \nabla p.
\]
\end{Lemma}

\begin{proof}
 We define for $x=(x_1,\ldots,x_n) \in \R^n$ the following $C^\infty$-function
\[
 p(x) = \sum_{j=1}^n \int_0^1 V_j(tx) x_j \,dt.
\]
Then we have
\[
 \partial_k p = \left(\sum_{j=1}^n \int_0^1 \partial_k V_j(tx) t x_j \,dt\right) + \int_0^1 V_k(tx) \,dt.
\]
Using \eqref{curl_free} we get by inserting $\partial_j V_k$ for $\partial_k V_j$
\begin{eqnarray*}
 \partial_k p &=& \int_0^1 \left(\sum_{j=1}^n t \partial_j V_k(tx) x_j\right) + V_k(tx) \, dt \\
 &=& \int_0^1 \partial_t \left(t V_k(tx)\right) \,dt = V_k(x).
\end{eqnarray*}
Thus we have $V=\nabla p$.
\end{proof}

We will also need the following well-known lemma.

\begin{Lemma}\label{lemma_orthogonal}
Let $V \in C^\infty(\R^n;\R^n)$ be a smooth vector field with the property
\[
 \langle V,\varphi \rangle_{L^2}:=\int_{\R^n} V \cdot \varphi \,dx = 0
\]
for all $\varphi \in C^\infty_{\sigma,c}(\R^n;\R^n)$. Then there is a $p \in C^\infty(\R^n)$ with $V=\nabla p$.
\end{Lemma}

\begin{proof}
We show that we have \eqref{curl_free}. We then get the claim by the Poincar\'e Lemma (Lemma \ref{lemma_poincare}). For this it will be enough to show 
\[
 \int_{\R^n} \sum_{1 \leq i < j \leq n} (\partial_i V_j - \partial_j V_i) \psi_{ij} \,dx = 0
\]
for all $\psi_{ij} \in C_c^\infty(\R^n)$, $1 \leq i < j \leq n$. Indeed we then get \eqref{curl_free} by the fundamental lemma of Calculus of Variations. Using integration by parts we get
\begin{multline}
\label{inner_product}
\int_{\R^n} \sum_{1 \leq i < j \leq n} (\partial_i V_j - \partial_j V_i) \psi_{ij}\, dx = \int_{\R^n} \sum_{1 \leq i < j \leq n} -V_j \partial_i \psi_{ij} + V_i \partial_j \psi_{ij} \, dx \\
= \int_{\R^n} \sum_{k=1}^n V_k \left(\sum_{j < k} -\partial_j \psi_{jk} + \sum_{j > k} \partial_k \psi_{kj} \right) \, dx = \langle V,\varphi \rangle_{L^2}
\end{multline}
where $\varphi=(\varphi_1,\ldots,\varphi_n)$ is given by
\[
 \varphi_k = \sum_{j < k} -\partial_j \psi_{jk} + \sum_{j>k} \partial_j \psi_{kj}
\]
for $1 \leq k \leq n$. For this $\varphi$ we have
\[
 \operatorname{div} \varphi = \sum_{k=1}^n \partial_k \varphi_k = \sum_{k=1}^n \left(\sum_{j < k} -\partial_j \partial_k \psi_{jk} + \sum_{j > k} \partial_j \partial_k \psi_{kj} \right) = 0
\]
i.e. we have $\varphi \in C^\infty_{\sigma,c}(\R^n;\R^n)$. Thus we see by assumption that \eqref{inner_product} is $0$ for all $\psi_{ij} \in C_c^\infty(\R^n)$, $1 \leq i < j \leq n$. Hence the claim.
\end{proof}

For the proof of Lemma \ref{lemma_dense} we will need the concept of mollifiers (for the details see e.g. \cite{majda}). So let $\rho \in C_c^\infty(\R^n)$ be a non-negative radial function with
\[
 \int_{\R^n} \rho(y)\,dy = 1.
\]
For $\varepsilon > 0$ we define the operator $\mathcal J_\varepsilon$ on locally integrable functions $v$ in $\R^n$ as
\[
 (\mathcal J_\varepsilon v)(x) = \frac{1}{\varepsilon^{n}} \int_{\R^n} \rho\left(\frac{x-y}{\varepsilon}\right) v(y)\,dy
\]
where $x \in \R^n$. For vector-valued functions $\mathcal J_\varepsilon$ is defined to act componentwise.

\begin{proof}[Proof of Lemma \ref{lemma_dense}]
For $k \in \N$ we introduce the following scalar product on $H^k(\R^n;\R^n)$
\[
 \langle f, g \rangle_{H^k} := \sum_{|\alpha| \leq k} (\partial^\alpha f, \partial^\alpha g)_{L^2}.
\]
This is equivalent to $\langle \cdot,\cdot \rangle_s$ for $s=k$. First we will prove that $C^\infty_{\sigma,c}(\R^n;\R^n)$ is dense in $H_\sigma^k(\R^n;\R^n)$. Assume that for $w \in H_\sigma^k(\R^n;\R^n)$ we have
\[
 \langle w,\varphi \rangle_{H^k} = 0
\]
for all $\varphi \in C^\infty_{\sigma,c}(\R^n;\R^n)$. We then have for all $\varphi \in C^\infty_{\sigma,c}(\R^n;\R^n)$ and $\varepsilon > 0$
\[
 \langle \mathcal J_\varepsilon w,\varphi \rangle_{H^k} = \langle w, \mathcal J_\varepsilon \varphi \rangle_{H^k} = 0
\]
as $\mathcal J_\varepsilon \varphi \in C^\infty_{\sigma,c}(\R^n;\R^n)$ for all $\varepsilon > 0$. Thus we have for all $\varphi \in C^\infty_{\sigma,c}(\R^n;\R^n)$
\[
 0 = \langle \mathcal J_\varepsilon w,\varphi \rangle_{H^k} = \langle A \mathcal J_\varepsilon w,\varphi \rangle_{L^2}
\]
where $A$ is the invertible elliptic operator
\[
 A=\sum_{|\alpha| \leq k} (-1)^{|\alpha|} \partial^{2\alpha}.
\]
Thus by Lemma \ref{lemma_orthogonal} there is a $p \in C^\infty(\R^n)$ such that
\[
 A \mathcal J_\varepsilon w = \nabla p.
\]
On the other hand we have as $w \in H_\sigma^k(\R^n;\R^n)$
\[
 \operatorname{div} A \mathcal J_\varepsilon w = A \mathcal J_\varepsilon \operatorname{div} w =0.
\]
Hence we have $\Delta p = \operatorname{div} A \mathcal J_\varepsilon w = 0$. This means that the components of $\nabla p$ are harmonic functions. But as we have
\[
 \nabla p = A \mathcal J_\varepsilon w \in L^2(\R^n;\R^n)
\]
we get $\nabla p =0$. Therefore we have $\mathcal J_\varepsilon w =0$ for all $\varepsilon > 0$ as $A$ is invertible. But as $\mathcal J_\varepsilon w \to w$ in $H^k$ as $\varepsilon \to 0$ we get $w=0$. This shows that $C^\infty_{\sigma,c}(\R^n;\R^n)$ is dense in $H_\sigma^k(\R^n;\R^n)$. For general $H_\sigma^s(\R^n;\R^n)$, $s \geq 0$, we take a $k \in \N$ with $k \geq s$. Now $H_\sigma^k(\R^n;\R^n)$ is dense in $H_\sigma^s(\R^n;\R^n)$. Indeed just take $\mathcal J_\varepsilon w$ to approximate $w \in H_\sigma^s(\R^n;\R^n)$ by functions in $H_\sigma^k(\R^n;\R^n)$. By the consideration above we can approximate each $w \in H_\sigma^k(\R^n;\R^n)$ by elements in $C^\infty_{\sigma,c}(\R^n;\R^n)$ in the $H^k$-norm. So in particular also in the $H^s$-norm. This proves the case for general $s \geq 0$. 
\end{proof}

\begin{Lemma}\label{lemma_biot_savart}
Let $s > n/2+1$ and let $u \in H_\sigma^s(\R^n;\R^n)$ with compactly supported vorticity $\Omega=\Omega(u)$. Then we have for all $x \in \R^n$
\[
 u(x)=\frac{1}{\omega_n} \int_{\R^n} \Omega(y) \frac{x-y}{|x-y|^n} \,dy
\]
or in components for $1 \leq i \leq n$
\begin{equation}\label{biot_savart}
u_i(x) = \frac{1}{\omega_n} \int_{\R^n} \sum_{j=1}^n \Omega_{ij}(y) \frac{x_j-y_j}{|x-y|^n} \,dy
\end{equation}
where $\omega_n$ is the surface area of a unit sphere in $\R^n$
\end{Lemma}
Note that the singularity at $y=x$ in \eqref{biot_savart} is integrable and that the integrand is compactly supported and bounded by the Sobolev imbedding \eqref{sobolev_imbedding}. So the integral is well-defined. 

\begin{proof}
 Recall that we have
\[
 \Omega_{ij}=\partial_j u_i - \partial_i u_j
\]
for $1 \leq i,j \leq n$. Therefore 
\begin{eqnarray}
\nonumber
\sum_{j=1}^n \partial_j \Omega_{ij} &=& \sum_{j=1}^n \partial_j \partial_j u_i - \partial_i \partial_j u_j\\
\label{laplace_ui}
&=& \sum_{j=1}^n \partial_j \partial_j u_i = \Delta u_i
\end{eqnarray}
as $\sum_{j=1}^n \partial_j u_j=\operatorname{div} u = 0$. We denote by $E$ the fundamental solution of the Laplacian, i.e. 
\[
 E(x) = \begin{cases} \frac{1}{2\pi} \log |x|, \quad & n=2 \\ \frac{1}{(2-n)\omega_n} |x|^{2-n}, \quad & n \geq 3 \end{cases}
\]
where $\omega_n$ is as above (see \cite{majda}). Note that we have $\lim_{|x| \to \infty} u(x)=0$ by the Sobolev imbedding theorem \eqref{sobolev_imbedding}. Thus we get from \eqref{laplace_ui} and the theory for the Poisson equation (see e.g. \cite{majda})
\begin{equation}\label{E_convolution}
 u_i(x) = \left[ E \ast \left(\sum_{j=1}^n \partial_j \Omega_{ij} \right) \right] (x)
\end{equation}
where $E \ast g(x):=\int_{\R^n} E(x-y)g(y) \;dy$ denotes the convolution. Note that $\sum_{j=1}^n \partial_j \Omega_{ij}$ is compactly supported. We can rewrite \eqref{E_convolution} as
\[
 u_i(x) = \sum_{j=1}^n \left[(\partial_j E) \ast \Omega_{ij} \right] (x)
\]
since $\partial_j E(x) = \frac{1}{\omega_n} \frac{x_j}{|x|^n}$ is locally integrable. In integral form this reads as \eqref{biot_savart} showing the claim. 
\end{proof}

The following lemma tells that the gradient of the velocity can be estimated by the vorticity.

\begin{Lemma}\label{lemma_gradient_velocity}
We have for some $C>0$
\[
 ||du||_{s-1} \leq C ||\Omega(u)||_{s-1}
\]
for all $u \in H_\sigma^s(\R^n;\R^n)$. Here $du$ denotes the differential of $u$.
\end{Lemma}

\begin{proof}
The fact $u \in H_\sigma^s(\R^n;\R^n)$ reads on the Fourier side as $\sum_{j=1}^n \xi_j \hat u_j (\xi) \equiv 0$. Thus we have
\begin{eqnarray*}
 \frac{1}{i} \sum_{j=1}^n \widehat \Omega_{\ell j}(\xi) \frac{\xi_k \xi_j}{|\xi|^2} &=& \sum_{j=1}^n \left(\xi_j \hat u_k(\xi) - \xi_k \hat u_j(\xi) \right) \frac{\xi_k \xi_j}{|\xi|^2} \\
&=& \xi_k \sum_{j=1}^n \frac{\xi_j^2}{|\xi|^2} \hat u_k(\xi) = \xi_k \hat u_\ell(\xi).
\end{eqnarray*}
Thus we see
\begin{eqnarray*}
||\partial_k u_\ell||_{s-1} &=& ||(1+|\xi|^2)^{\frac{s-1}{2}} \xi_k \hat u_\ell(\xi)||_{L^2} \\
&\leq& ||(1+|\xi|^2)^{\frac{s-1}{2}} \sum_{j=1}^n \widehat \Omega_{\ell j} (\xi) \frac{\xi_k \xi_j}{|\xi|^2} ||_{L^2}\\
&\leq& \sum_{j=1}^n ||\Omega_{\ell j}||_{s-1} \leq n ||\Omega||_{s-1}
\end{eqnarray*}
which shows the claim.
\end{proof}

\addcontentsline{toc}{chapter}{Bibliography}
\bibliographystyle{plain}

\end{document}